\pdfoutput=1
\newif\ifpersonal
\documentclass[12pt,a4paper,dvipsnames,x11names]{amsart}

\usepackage{amsmath,amsthm,amssymb,mathrsfs,mathtools,tensor,eucal} 
\usepackage[all,cmtip]{xy}

\usepackage[LGR,T1]{fontenc} 
\usepackage[utf8]{inputenc} 
\usepackage{CJKutf8}

\usepackage{xr}
\usepackage{microtype,inconsolata} 
\usepackage{enumerate,comment,braket,xspace,tikz,tikz-cd,csquotes} 
\usetikzlibrary{shapes.geometric}
\usepackage[centering,vscale=0.7,hscale=0.7]{geometry}
\usepackage[hidelinks,colorlinks=true,linkcolor=Red4,citecolor=RoyalBlue]{hyperref}
\usepackage[capitalize]{cleveref}
\usepackage{xcolor}

\usepackage{mathpazo}
\usepackage{euler}
\usepackage{tikzarrows}

\tikzcdset{
	cells={font=\everymath\expandafter{\the\everymath\displaystyle}},
}

\numberwithin{equation}{subsection}
\newcounter{equation-intro}[section]
\numberwithin{equation-intro}{section}

\theoremstyle{plain}
\newtheorem{thm}[equation]{Theorem}
\newtheorem{lem}[equation]{Lemma}
\newtheorem{prop}[equation]{Proposition}

\newtheorem{cor}[equation]{Corollary}

\newtheorem{thm-intro}[equation-intro]{Theorem}
\newtheorem{slogan-intro}[equation-intro]{Slogan}

\theoremstyle{definition}

\newtheorem{defin}[equation]{Definition}
\newtheorem{notation}[equation]{Notation}
\newtheorem{eg}[equation]{Example}
\newtheorem{rem}[equation]{Remark}
\newtheorem{variant}[equation]{Variant}
\newtheorem{observation}[equation]{Observation}
\newtheorem{recollection}[equation]{Recollection}
\newtheorem{warning}[equation]{Warning}
\newtheorem{construction}[equation]{Construction}

\newtheorem{defin-intro}[equation-intro]{Definition}
\newtheorem{eg-intro}[equation-intro]{Example}
\newtheorem{rem-intro}[equation-intro]{Remark}
\newtheorem{situation-intro}[equation-intro]{Situation}
\newtheorem{notation-intro}[equation-intro]{Notation}

\ifpersonal
\newcommand{\personal}[1]{\textcolor[rgb]{0,0,1}{(Personal: #1)}}
\newcommand{\discussion}[1]{\textcolor{violet}{(Discussion: #1)}}
\else
\newcommand{\personal}[1]{\ignorespaces}
\newcommand{\discussion}[1]{\ignorespaces}
\fi

\newcommand{\C}{\mathbb C}

\newcommand{\Q}{\mathbb Q}
\newcommand{\R}{\mathbb R}

\newcommand{\rB}{\mathrm B}

\newcommand{\cA}{\mathcal A}
\newcommand{\cB}{\mathcal B}
\newcommand{\cC}{\mathcal C}

\newcommand{\cE}{\mathcal E}
\newcommand{\cF}{\mathcal F}

\newcommand{\cI}{\mathcal I}
\newcommand{\cJ}{\mathcal J}
\newcommand{\cK}{\mathcal K}
\newcommand{\cL}{\mathcal L}
\newcommand{\cM}{\mathcal M}

\newcommand{\cO}{\mathcal O}

\newcommand{\cU}{\mathcal U}
\newcommand{\cV}{\mathcal V}

\newcommand{\cX}{\mathcal X}
\newcommand{\cY}{\mathcal Y}

\DeclareFontFamily{U}{BOONDOX-calo}{\skewchar\font=45 }
\DeclareFontShape{U}{BOONDOX-calo}{m}{n}{<-> s*[1.05] BOONDOX-r-calo}{}
\DeclareFontShape{U}{BOONDOX-calo}{b}{n}{<-> s*[1.05] BOONDOX-b-calo}{}
\DeclareMathAlphabet{\mathcalboondox}{U}{BOONDOX-calo}{m}{n}

\makeatletter
\let\save@mathaccent\mathaccent
\newcommand*\if@single[3]{%
	\setbox0\hbox{${\mathaccent"0362{#1}}^H$}%
	\setbox2\hbox{${\mathaccent"0362{\kern0pt#1}}^H$}%
	\ifdim\ht0=\ht2 #3\else #2\fi
}
\newcommand*\rel@kern[1]{\kern#1\dimexpr\macc@kerna}
\newcommand*\widebar[1]{\@ifnextchar^{{\wide@bar{#1}{0}}}{\wide@bar{#1}{1}}}
\newcommand*\wide@bar[2]{\if@single{#1}{\wide@bar@{#1}{#2}{1}}{\wide@bar@{#1}{#2}{2}}}
\newcommand*\wide@bar@[3]{%
	\begingroup
	\def\mathaccent##1##2{%
		\let\mathaccent\save@mathaccent
		\if#32 \let\macc@nucleus\first@char \fi
		\setbox\z@\hbox{$\macc@style{\macc@nucleus}_{}$}%
		\setbox\tw@\hbox{$\macc@style{\macc@nucleus}{}_{}$}%
		\dimen@\wd\tw@
		\advance\dimen@-\wd\z@
		\divide\dimen@ 3
		\@tempdima\wd\tw@
		\advance\@tempdima-\scriptspace
		\divide\@tempdima 10
		\advance\dimen@-\@tempdima
		\ifdim\dimen@>\z@ \dimen@0pt\fi
		\rel@kern{0.6}\kern-\dimen@
		\if#31
		\overline{\rel@kern{-0.6}\kern\dimen@\macc@nucleus\rel@kern{0.4}\kern\dimen@}%
		\advance\dimen@0.4\dimexpr\macc@kerna
		\let\final@kern#2%
		\ifdim\dimen@<\z@ \let\final@kern1\fi
		\if\final@kern1 \kern-\dimen@\fi
		\else
		\overline{\rel@kern{-0.6}\kern\dimen@#1}%
		\fi
	}%
	\macc@depth\@ne
	\let\math@bgroup\@empty \let\math@egroup\macc@set@skewchar
	\mathsurround\z@ \frozen@everymath{\mathgroup\macc@group\relax}%
	\macc@set@skewchar\relax
	\let\mathaccentV\macc@nested@a
	\if#31
	\macc@nested@a\relax111{#1}%
	\else
	\def\gobble@till@marker##1\endmarker{}%
	\futurelet\first@char\gobble@till@marker#1\endmarker
	\ifcat\noexpand\first@char A\else
	\def\first@char{}%
	\fi
	\macc@nested@a\relax111{\first@char}%
	\fi
	\endgroup
}
\makeatother

\newcommand{\PSh}{\mathrm{PSh}}
\newcommand{\Sh}{\mathrm{Sh}}

\newcommand{\St}{\mathrm{St}}

\newcommand{\Top}{\mathcal T\mathrm{op}}

\newcommand{\CAlg}{\mathrm{CAlg}}

\newcommand{\dAff}{\mathrm{dAff}}

\newcommand{\Perf}{\mathrm{Perf}}
\newcommand{\bfPerf}{\mathbf{Perf}}

\newcommand{\st}{\mathrm{st}}

\newcommand{\IV}{\mathscr{I} \mathscr{V}}

\newcommand{\Exit}{\mathrm{Exit}}

\DeclareMathOperator{\AnStrat}{\mathbf{AnStrat}}
\DeclareMathOperator{\StAnStrat}{\mathbf{StAnStrat}}

\DeclareMathOperator{\FStStrat}{\mathbf{FStStrat}}

\DeclareMathOperator{\ho}{ho}
\DeclareMathOperator{\mode}{mod}
\DeclareMathOperator{\LSt}{LSt}

\DeclareMathOperator{\Gr}{Gr}

\DeclareMathOperator{\ens}{set}

\DeclareMathOperator{\cocart}{cocart} 
\DeclareMathOperator{\Fil}{Fil} 

\DeclareMathOperator{\spe}{sp}

\DeclareMathOperator{\Loc}{Loc}

\DeclareMathOperator{\GL}{GL}

\DeclareMathOperator{\PS}{PS}

\DeclareMathOperator{\Supp}{Supp}

\DeclareMathOperator{\Env}{Env}
\DeclareMathOperator{\order}{ord}
\DeclareMathOperator{\StFil}{StFil}

\newcommand{\PrLR}{\categ{Pr}^{\kern0.05em\operatorname{L}, \operatorname{R}}}

\newcommand{\Funcocart}{\Fun^{\cocart}}

\DeclareMathOperator{\at}{at}

\newcommand{\ev}{\mathrm{ev}}

\newcommand{\inv}{^{-1}}
\newcommand{\id}{\mathrm{id}}

\newcommand{\op}{^\mathrm{op}}

\usetikzlibrary{decorations.markings} 
\tikzset{
  closed/.style = {decoration = {markings, mark = at position 0.5 with { \node[transform shape, xscale = .8, yscale=.4] {/}; } }, postaction = {decorate} },
  open/.style = {decoration = {markings, mark = at position 0.5 with { \node[transform shape, scale = .7] {$\circ$}; } }, postaction = {decorate} }
}

\DeclareMathOperator{\lb}{lb}

\DeclareMathOperator{\cofib}{cofib}

\DeclareMathOperator{\Fun}{Fun}

\DeclareMathOperator{\Hom}{Hom}
\DeclareMathOperator{\Image}{Im}

\DeclareMathOperator{\Map}{Map}
\DeclareMathOperator{\Mor}{Mor}

\DeclareMathOperator{\Spec}{Spec}

\DeclareMathOperator*{\colim}{colim}

\DeclareMathOperator{\hyp}{hyp}

\DeclareMathOperator{\idem}{idem}

\DeclareMathOperator{\Real}{Re}

\DeclareMathOperator{\Cons}{Cons}

\newcommand{\ConsP}{\Cons_{P}}
\newcommand{\ConsS}{\Cons_{S}}
\newcommand{\ConsR}{\Cons_{R}}
\newcommand{\ConsQ}{\Cons_{Q}}
\newcommand{\ConsPhyp}{\ConsP^{\hyp}}
\newcommand{\ConsRhyp}{\ConsR^{\hyp}}
\newcommand{\ConsShyp}{\ConsS^{\hyp}}
\newcommand{\ConsQhyp}{\ConsQ^{\hyp}}
\newcommand{\Shhyp}{\Sh^{\hyp}}

\newcommand{\categ}[1]{\textbf{\textup{#1}}}

\newcommand{\Spc}{\categ{Spc}}
\newcommand{\Sp}{\categ{Sp}}
\newcommand{\Stratc}{\categ{ExStrat}}
\newcommand{\PrR}{\categ{Pr}^{\kern0.05em\operatorname{R}}}
\newcommand{\PrL}{\categ{Pr}^{\kern0.05em\operatorname{L}}}
\newcommand{\PrLomega}{\categ{Pr}^{\kern0.05em\operatorname{L},\omega}}
\newcommand{\PrLkappa}{\categ{Pr}^{\kern0.05em\operatorname{L},\kappa}}
\newcommand{\PrLRomega}{\categ{Pr}^{\kern0.05em\operatorname{L}, \operatorname{R},\omega}}
\newcommand{\PrLotimes}{\categ{Pr}^{\kern0.05em\operatorname{L},\otimes}}
\newcommand{\PrLat}{\categ{Pr}^{\kern0.05em\operatorname{L}, \operatorname{at}}}
\newcommand{\Cat}{\categ{Cat}}
\newcommand{\CAT}{\textbf{\textsc{Cat}}}
\newcommand{\Catidem}{\Cat_{\infty}^{\idem}}
\newcommand{\PosFib}{\categ{PosFib}}
\newcommand{\StStrat}{\categ{StStrat}}

\newcommand{\hCoCart}{\textbf{\textsc{CoCart}}}
\newcommand{\frakFil}{\mathfrak{Fil}}

\newcommand{\frakStokes}{\mathfrak{St}}
\newcommand{\Setcat}{\categ{Set}}
\newcommand{\Poset}{\categ{Poset}}
\newcommand{\Mod}{\mathrm{Mod}}
\newcommand{\PiinftySigma}{\Pi_{\infty}}

\newcommand{\bfSt}{\mathbf{St}}
\newcommand{\bfStflat}{\mathbf{St}^{\mathrm{flat}}}

\newcommand{\Ind}{\mathrm{Ind}}

\newcommand{\Filco}{\Fil^{\mathrm{co}}}

\newcommand{\CoCart}{\categ{CoCart}}

\newcommand{\PrFibL}{\categ{PrFib}^{\mathrm{L}}}

\begin{document}

\title{The derived moduli of Stokes data}

\author{Mauro PORTA}
\address{Mauro PORTA, Institut de Recherche Mathématique Avancée, 7 Rue René Descartes, 67000 Strasbourg, France}
\email{porta@math.unistra.fr}

\author{Jean-Baptiste Teyssier}
\address{Jean-Baptiste Teyssier, Institut de Mathématiques de Jussieu, 4 place Jussieu, 75005 Paris, France}
\email{jean-baptiste.teyssier@imj-prg.fr}

\subjclass[2020]{}
\keywords{}

\begin{abstract}
The goal of this paper is to show that Stokes data coming from  flat bundles form a locally geometric derived stack locally of finite presentation.
This generalizes existing geometricity results on Stokes structures in four different directions: our result applies in any dimension, $\infty$-categorical coefficients are allowed, derived structures on moduli spaces are considered and more general spaces than those arising from flat bundles are permitted. 
\end{abstract}

\maketitle


\tableofcontents

\section{Introduction}

Let $(E,\nabla)$ be a rank $n$ algebraic flat bundle on a smooth complex algebraic variety $X$.
Then,  analytic continuation of the solutions of the differential system $\nabla =0$ gives rise to a representation $\rho \colon  \pi_1(X) \to \GL_n(\mathbb{C})$ called the \textit{monodromy representation}.
If  favourable  conditions are imposed, the data of  $\rho$  and  $(E,\nabla)$ are equivalent.
In that case $(E,\nabla)$ is called \textit{regular singular} \cite{Del} and this case is characterized by the fact that the formal solutions to  $\nabla = 0$ automatically converge.
In general, the monodromy representation is not enough to capture all the analytic information contained in $(E,\nabla)$.
As already seen by Stokes on the Airy equation \cite{Stokes_2009},  formal solutions to $\nabla = 0$ may not converge anymore,  but their interplay with analytic solutions is highly structured and gives rise to what is nowadays called a \textit{Stokes structure} or a \textit{Stokes filtered local system} \cite{DeligneLettreMalgrange,BV,Stokes_Lisbon}.
To picture it, let us suppose that $X$ is the affine line and let $S^1_{\infty}$ be the circle of directions emanating from $\infty$.
Then, the flat bundle $(E,\nabla)$ has \textit{good formal structure at $\infty$}, meaning roughly that when restricted to a formal neighbourhood of $\infty$, it decomposes as a direct sum of regular flat bundles twisted by rank one bundles.
The theory of asymptotic developments \cite{SVDP} then  ensures  the existence of a finite set  $\St(E,\nabla)\subset S^1_{\infty}$ of \textit{Stokes directions}  such that for every $d \notin \St(E,\nabla)$, any formal solution $\hat{f}$ to $\nabla = 0$ at $\infty$ lifts to  an analytic solution $f$ in some small enough sector $S$ containing $d$.
We also say that $\hat{f}$  is the asymptotic development of $f$.
By Cauchy's theorem, $f$ admits an analytic continuation to any sector obtained by rotating $S$.
However, the asymptotic development is not preserved under the analytic continuation procedure and may jump when crossing a Stokes line.
This is the \textit{Stokes phenomenon}.
In practice, these jumps are measured by matrices (one for each Stokes direction) called \textit{Stokes matrices}.
Note that Stokes matrices are subjected to choices of basis.  
To get a more intrinsic presentation, let $L$ be the local system of solutions to $\nabla =0$ on  $S^1_{\infty}$.  
Deligne and Malgrange observed in \cite{DeligneLettreMalgrange} that the Stokes phenomenon  is recorded by a filtration of $L$ by constructible subsheaves on $S^1_\infty$ indexed by $\cO_{\mathbb{P}^1,\infty}(\ast \infty)/\cO_{\mathbb{P}^1,\infty}$.
To define them, first observe that any arc $U \subset S^1_\infty$ gives rise to a sector $S(U)$ around $\infty$.
Then for $a\in  \cO_{\mathbb{P}^1,\infty}(\ast \infty)/\cO_{\mathbb{P}^1,\infty}$ and any arc $U$ in $S^1_\infty$,  we put
\[
L_{\leq a}(U) =\{\text{$f \in L$ such that $e^{-a} f$ has moderate growth at $\infty$ in the sector $S(U)$} \} \ .
\]
Although this filtration is indexed by an infinite dimensional parameter space, only a finite number of elements, called \textit{irregular values of $(E,\nabla)$} contribute in a non trivial way. 

\medskip

On the other hand, representations of the fundamental group naturally form an algebraic variety, the \textit{character variety}.
It is thus a natural question to ask whether Stokes structures also form an algebraic variety.
This question was answered in \cite{Boalch_curve,boalch2015twisted,huang2023moduli} in the curve case via GIT methods.
See also \cite[\S13]{Boalch_Topology_Stokes} and \cite{bezrukavnikov2022nonabelian} for a stacky variant in the curve case.
In dimension $\geq 2$, several major obstacles arise.
The first one is that good formal structures breaks down.
Still, Sabbah conjectured \cite{Stokes} that good formal structure can be achieved at the cost of enough blow-up above the divisor at infinity.
This problem was solved independently by Kedlaya \cite{Kedlaya1,Kedlaya2} and Mochizuki \cite{Mochizuki1,Mochizuki2}.
Furthermore, given a smooth compact algebraic variety $X$ and a normal crossing divisor $D$, Mochizuki attached to every flat bundle $(E,\nabla)$ on $U\coloneqq X\setminus D$ with good formal structure along $D$ a Stokes filtered local system $(L,L_{\leq})$ on the real blow-up $\pi \colon \widetilde{X}\rightarrow X$ along the components of $D$, and showed that the data of $(E,\nabla)$ and $(L,L_{\leq})$ are equivalent. 
Once strapped in this setting, a second major obstacle in dimension $\geq 2$ pertains to the \textit{stratified} nature of good formal structure.
To explain it, suppose that $X=\mathbb{C}^2$, let $D_1, D_2$ be the coordinate axis and let $D$ be their union. 
Then, very roughly, good formal structure holds separately on the formal neighbourhoods of $0$, $D_1 \setminus \{0\}$ and $D_2 \setminus \{0\}$.
In dimension 1, the points at infinity are isolated so their contributions to the moduli of Stokes structures don't interact and can thus be analysed separately.
In higher dimension, the contributions of $D_1 \setminus \{0\}$, $D_2 \setminus \{0\}$ and $0$ are necessarily intricated.
This in particular makes it unclear how  to use the moduli of Stokes torsors from \cite{TeySke} as smooth atlases in global situations.
In this paper, we generalize all known construction of the moduli of (possibly \textit{ramified}) Stokes filtered local systems in four different ways: 

\begin{enumerate}\itemsep=0.2cm
\item our result applies in any dimension;
\item $\infty$-categorical 
coefficients are allowed;
\item derived structures on moduli spaces are considered;
\item more general spaces than those arising from flat bundles are permitted. 
\end{enumerate}
For representability results along these lines in the De Rham side see \cite{pantev:hal-02003691}.

\medskip

The strongest versions of our theorems apply to the following situation:

\begin{situation-intro}\label{situation_intro}
	Let $X$ be a complex manifold admitting a smooth compactification.
	Let $D$ be a normal crossing divisor in $X$ and put $U\coloneqq X\setminus D$.
	Let $\pi \colon \widetilde{X}\to X$  be the real-blow up along $D$ (see \cref{real_blow_up}) and $j \colon U \to  \widetilde{X}$ the inclusion.
	Let $\mathscr{I} \subset (j_{\ast}\cO_{U})/ (j_{\ast}\cO_{U})^{\lb}  $ be a good sheaf of irregular values (see \cref{good sheaf}).
	A point $x\in \widetilde{X}$ with $\pi(x)\in D$ can be thought of as a line passing through $\pi(x)$ and a section of $\pi^{\ast}\mathscr{I}$ near $x$ as a meromorphic function defined on some small multi-sector emanating from $\pi(x)$.
	For two such sections $a$ and $b$, the relation
	\[
	\text{$a\leq_x b$ if and only if $e^{a-b}$ has moderate growth at $x$}
	\]
	defines an order on the germs of $\pi^{\ast}\mathscr{I}$ at $x$.
	This collection of orders upgrades $\pi^{\ast}\mathscr{I}$ into a sheaf of posets that turns out to be constructible for a suitable choice of finite subanalytic stratification $P$ of $\widetilde{X}$.
\end{situation-intro}

The starting point of our work is the observation that Mochizuki's notion of higher dimensional Stokes filtered local systems only depends on the stratified homotopy type of $(\widetilde{X},P)$ and on the $P$-constructible sheaf of posets $\mathscr I$.
Out of this data, we convert $\mathscr I$ into a more combinatorial object in two steps.
The first one is channeled by the \textit{topological exodromy equivalence}, originally envisioned by MacPherson and for which a rich literature is nowadays available \cite{Lurie_Higher_algebra,Exodromy_coefficient,jansen2023stratified,Beyond_conicality}.
This equivalence converts a (hyper)constructible (hyper)sheaf with respect to a stratification $Q$ of a topological space $Y$ into a functor from the $\infty$-category of \textit{exit paths} $\Pi_{\infty}(Y,Q)$ attached to $(Y,Q)$.
By design, the objects of $\Pi_{\infty}(Y,Q)$ are the points of $Y$ and the morphisms between two points $x$ and $y$ can be thought of as continuous paths $\gamma \colon  [0,1]\to Y$ from $x$ to $y$ such that $\gamma((0,1])$ lies in the same stratum as $y$.
In the setting of \cref{situation_intro}, $\pi^{\ast}\mathscr{I}$ thus corresponds via the exodromy equivalence to a functor $\Pi_{\infty}(\widetilde{X},P)\to \Poset$.
The second step of our conversion consists in passing to the associated cocartesian fibration in posets $\cI \to \Pi_{\infty}(\widetilde{X},P)$ via the Grothendieck construction \cite[Theorem 3.2.0.1]{HTT}.
This procedure is essentially combinatorial in nature and, while it requires heavy technology to be made sense of $\infty$-categorically, it is quite simple to grasp it in practice, as the following toy example illustrates:

\begin{eg-intro}\label{eg:Stokes_on_interval_I}
	On $Y = (0,1)$ consider the constructible sheaf of posets $\cI$ depicted as follows:
	
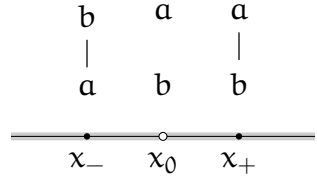
\begin{center}
		\begin{tikzpicture}
			\draw[thin] (-1,-0.5) -- (3,-0.5) ;
			\draw[line width=3pt,opacity=0.2] (-1,-0.5) -- (0.95,-0.5) ;
			\fill (0,-0.5) circle (1.2pt) ;
			\draw[thin,fill=white] (1,-0.5) circle (1.5pt) ;
			\draw[line width=3pt,opacity=0.2] (1.05,-0.5) -- (3,-0.5) ;
			\fill (2,-0.5) circle (1.2pt) ;
			\node[below=2pt] at (0,-0.5) {$x_-$} ;
			\node[above=2pt,anchor=north] at (0,1.5) {$\begin{tikzcd}[row sep=small] b \arrow[-]{d} \\ a \end{tikzcd}$} ;
			\node[below=2pt] at (1,-0.5) {$x_0$} ;
			\node[above=2pt,anchor=north] at (1,1.5) {$\begin{tikzcd}[column sep=tiny,row sep=small] a \\ b \end{tikzcd}$} ;
			\node[below=2pt] at (2,-0.5) {$x_+$} ;
			\node[above=2pt,anchor=north] at (2,1.5) {$\begin{tikzcd}[row sep=small] a \arrow[-]{d} \\ b \end{tikzcd}$} ;
		\end{tikzpicture}
	\end{center}
	where we drew the Hasse diagrams of the corresponding poset (over $x_-$ we have $a \leqslant b$ and over $x_+$ we have $b \leqslant a$).
	The underlying sheaf of sets is the constant sheaf associated to the set $\{a,b\}$, and the stratification on $Y$ is given by the single point $x_0$.
	The following picture represents the exit path category of this stratification (bottom line), as well as the total space of the associated cocartesian fibration (upper diagram):
	
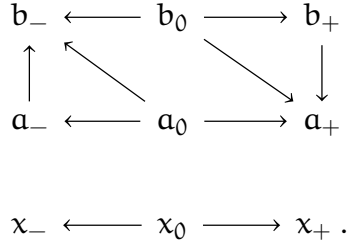
\[ \begin{tikzcd}
		b_- & b_0 \arrow{r} \arrow{l} \arrow{dr} & b_+ \arrow{d} \\
		a_- \arrow{u} & a_0 \arrow{r} \arrow{l} \arrow{ul} & a_+ \\
		x_- & x_0 \arrow{r} \arrow{l} & x_+ \ .
	\end{tikzcd} \]
\end{eg-intro}

In this language, Stokes filtered local systems are special functors $F \colon  \cI \to \Mod_\C^\heartsuit$ that we call \textit{Stokes functors}, where $\Mod_\C^\heartsuit$ is the abelian category of $\mathbb{C}$-vector spaces.
We define Stokes functors with coefficients in any presentable $\infty$-category $\cE$. 
They are characterized by the following two conditions:

\medskip

\paragraph{\textbf{Splitting condition.}}

This condition is punctual.
For $x\in \widetilde{X}$,  let $\cI_x \in \Poset$ be the fibre of $\cI \to \Pi_{\infty}(\widetilde{X},P)$ above $x$ and consider the restricted functor $F_x \colon  \cI_x \to \cE$. 
Let $i_{\cI_x} \colon  \cI_x^{\ens}\to \cI_x$ be the underlying set of $\cI_x$.
Let $i_{\cI_x, ! } \colon   \Fun(\cI_x^{\ens},\cE)\to \Fun(\cI_x,\cE)$ be the left Kan extension of  $i_{\cI_x }^* \colon   \Fun(\cI_x,\cE)\to \Fun(\cI_x^{\ens},\cE)$.
Then $F_x$ is requested to lie in the essential image of $i_{\cI_x, ! }$.
Unravelling the definition, this means that there is   $V \colon   \cI_x^{\ens}\to \cE$ such that for  every $a\in \cI_x$, we have
\[
F_x(a)\simeq \bigoplus_{b\leq a \text{ in }\cI_x} V(b)   \ .
\]

\medskip

\paragraph{\textbf{Induction condition.}}

If $\gamma \colon  x\to y$ is an exit path for $(\widetilde{X},P)$, it pertains to a prescription of $F_y$  by $F_x$ via $\gamma$ referred  as \textit{induction} in \cite{Mochizuki1}.
If $\gamma \colon  \cI_x \to \cI_y$ is the  morphism of posets 
induced by  $\gamma \colon  x\to y$ and if $
\gamma_{ ! } \colon   \Fun(\cI_x,\cE)\to \Fun(\cI_y,\cE)$ is the 
left Kan extension of the pull-back $\gamma^* \colon   
\Fun(\cI_y,\cE)\to \Fun(\cI_x,\cE)$, Mochizuki's 
condition translates purely categorically into the requirement that the natural map $\gamma_!(F_x)\to F_y$ is an equivalence.

\begin{notation-intro}
	We write $\St_{\cI,\cE}$ for the full subcategory of $\Fun(\cI,\cE)$ consisting of those functors satisfying the splitting and the induction condition.
	We also write $\St_{\cI,\cE,\omega}$ for the full subcategory of $\St_{\cI,\cE}$ consisting of those Stokes functors $F$ taking values in the full subcategory $\cE^\omega$ of compact objects in $\cE$ (when $\cE = \Mod_\C^\heartsuit$, this means finite dimensional vector spaces).
\end{notation-intro}

\begin{rem-intro}\label{trivial_I_intro}
	When the sheaf $\mathscr{I} = \ast$ is trivial, the splitting condition is trivial and the induction condition is an instance of parallel transport from $x$ to $y$.
	So in this case, Stokes functors are nothing but local systems on $\widetilde{X}$.
	See \cref{Stokes_sheaf_trivial_fibration}.
\end{rem-intro}

\begin{rem-intro}[Stokes functors vs.\ Stokes filtered local systems]
	Write $p \colon \cI \to \Pi_\infty(X,P)$ for the structural morphism, and notice that $\Pi_\infty(X,P)$ is the cocartesian fibration associated to the trivial sheaf $\ast$.
	We prove in \cref{cor:stokes_functoriality_IHES} that the functor $p_!$ given by left Kan extension along $p$ produces a Stokes functor for the trivial sheaf $\ast$, and hence a local system on $\widetilde{X}$ by \cref{trivial_I_intro}.
	If $F$ is a Stokes functor, we write $|F| \coloneqq p_!(F)$ and we refer to it as the \emph{underlying local system}.
	The local system $|F|$ should be thought as equipped with a filtration given by the functor $F$ itself.
	Building on this perspective, we work out in \S\ref{subsec:comparison} a precise comparison between the notion of Stokes functor and that of classical Stokes filtered local system in dimension 1.
\end{rem-intro}

The following is the main result of this paper :

\begin{thm-intro}[\cref{geometricit_Stokes_classical_case}]\label{Main_theorem}
	Let $k$ be a (possibly animated) commutative ring.
	In the setting of \cref{situation_intro}, the derived prestack
	\[ \bfSt_{\cI} \colon \dAff_k\op \to  \Spc \]
	defined by the rule
	\[ \bfSt_{\cI}(\Spec(A)) \coloneqq (\St_{\cI,\Mod_A,\omega})^\simeq \]
	is locally geometric locally of finite presentation.
	Moreover, for every animated commutative $k$-algebra $A$ and every morphism
	\[ x \colon \Spec(A) \to \bfSt_{\cI} \]
	classifying a Stokes functor $F \colon \cI \to \Perf_A$, there is a canonical equivalence
	\[ x^\ast \mathbb T_{\bfSt_{\cI}} \simeq \Hom_{\Fun(\cI,\Mod_A)}( F, F )[1] \ , \]
	where $\mathbb T_{\bfSt_\cI}$ denotes the tangent complex of $\bfSt_\cI$ and the right hand side denotes the $\Mod_A$-enriched $\Hom$ of $\Fun(\cI,\Mod_A)$.
\end{thm-intro}

There are at least three reasons justifying the use of derived algebraic geometry.
First, it is sensitive to the full stratified homotopy type $\Pi_{\infty}(\widetilde{X},P)$ and not only its underlying 1-category.
In turn, this yields an interpretation of the cohomology of Stokes functors as cotangent complexes for $\bfSt_{\cI}$, leading to a better control of its infinitesimal theory than in the classical context.
Finally, by analogy with character varieties \cite{Goldman_symplectic,toen:hal-01891205} and the curve case \cite{Biquard_Boalch,Boalch_curve,boalch2015twisted,shende2020calabiyau}, we expect  $\bfSt_{\cI}$  to carry a shifted symplectic structure in the sense of \cite{PTVV_2013}, which is typically invisible from the viewpoint of classical algebraic geometry in dimension  $\geq 2$.
These aspects will be the topics of future works. 

\medskip

From the point of view of derived algebraic geometry, one of the main difficulty in proving \cref{Main_theorem} is that we need a very robust theory of Stokes functors with coefficients in derived $\infty$-categories. 
This is the case even for those who are solely interested in the special open substack of higher dimensional Stokes filtered local systems (see \cref{thm-intro:Stokes_filtered_local_systems} below), as in any case the \emph{derived} functor of point of this substack evaluated on a test derived affine involves Stokes functors with coefficients in a derived $\infty$-category.
One of the core results we obtain for the general theory of Stokes functors is the following:

\begin{thm-intro}[\cref{stability_lim_colim_ISt}]\label{stability_lim_colimi_thm}
	In the setting of \cref{situation_intro}, let $\cE$ be a presentable stable $\infty$-category.
	Then, the subcategory $\St_{\cI,\cE}\subset \Fun(\cI,\cE)$ is stable under limits and colimits.
\end{thm-intro}

Let us explain why \cref{stability_lim_colimi_thm} is striking.
Let $F_{\bullet} \colon  I \to \St_{\cI,\cE}$ be a diagram of Stokes functors and let $F\coloneqq  \lim F_i$ be its limit computed in $\Fun(\cI,\cE)$.
Then, for every $i\in I$ and every $x\in  \widetilde{X}$, the splitting condition for $F_i$ at $x$  yields an equivalence $F_{i,x}\simeq i_{\cI_x,!}(V_i)$ where $V_i \colon  \cI_x^{\ens} \to \cE$ is a functor.
Note that these equivalences are non canonical, so they typically cannot be rearranged into a diagram $V_{\bullet} \colon  I\to \Fun(\cI_x^{\ens}, \cE)$ realizing the splitting of $F$ at $x$.
What \cref{stability_lim_colimi_thm} says is that for Stokes stratified spaces coming from flat bundles, such a rearrangement exists.
As immediate corollary of \cref{stability_lim_colimi_thm}, we deduce the following 

\begin{thm-intro}[\cref{thm:St_presentable_stable} and \cref{Grothendieck_abelian}]\label{presentable_stable_thm}
	In the setting of \cref{stability_lim_colimi_thm}, the following hold;
	\begin{enumerate}\itemsep=0.2cm
		\item For every presentable stable $\infty$-category $\cE$, the $\infty$-category $\St_{\cI,\cE}$ is presentable stable.

		\item For every Grothendieck abelian category $\cA$, the category $\St_{\cI,\cA}$ is Grothendieck abelian.
	\end{enumerate}
\end{thm-intro}

When $\cA$ is the category of vector spaces over a field, (2) reproduces a theorem of Sabbah \cite[Corollary 9.20]{Stokes_Lisbon}.
This is again striking since over a point, Stokes functors  neither form a presentable stable $\infty$-category nor an abelian category.

\medskip

However, the proof of \cref{Main_theorem} requires a deeper analysis.
Using Theorems~\ref{stability_lim_colimi_thm} and \ref{presentable_stable_thm}, we identify the derived prestack $\bfSt_{\cI}$ with To\"en-Vaquié moduli of objects of $\St_{\cI,\Mod_k}$.
The main result of \cite{Toen_Moduli} implies therefore \cref{Main_theorem}, \textit{provided that} we can prove that $\St_{\cI,\Mod_k}$ is an $\infty$-category of finite type (see \cref{finite_type_defin}).
In other words, we reduce the proof of \cref{Main_theorem} to the following more fundamental result:

\begin{thm-intro}[\cref{finite_typeness}]\label{smooth_nc_space}
	In the setting of \cref{situation_intro}, let $\cE$ be a $k$-linear presentable stable $\infty$-category of finite type (see \cref{finite_type_defin}).
	Then, $\St_{\cI,\cE}$ is $k$-linear of finite type.
	In particular, $\St_{\cI,\cE}$ is a smooth non-commutative space over $k$.
\end{thm-intro}

The proof of this theorem, which is undoubtedly the core result of this paper, relies on three ingredients.
The first is an adaptation and amplification of the standard dévissage technique for Stokes structures based on level structures.
We will summarize the main ideas below, even though the proof of the main categorical result needed to enact the level induction is carried out in the companion paper \cite{Abstract_Derived_Stokes}.
The second key ingredient is a finiteness result for the stratified homotopy type of compact $\R$-analytic manifolds equipped with finite subanalytic stratifications.
This finiteness result has been obtained by the authors in collaboration with P.\ Haine in \cite[Theorems 0.4.2 \& 0.4.3]{Beyond_conicality}. 
It provides a generalization to the stratified setting of theorems of Lefschetz–Whitehead, Łojasiewicz and Hironaka on the finiteness of the underlying homotopy types of compact subanalytic spaces and real algebraic varieties.
The last ingredient is a careful analysis of the geometry of \cref{situation_intro}, carried out in the paper at hand and which led us to introduce the notion of elementarity and its variants (see \cref{elementary_intro} below).
In this respect, the main results we obtain is a spreading out property for Stokes functors, proved in \cref{thm:spreading_out} and the elementarity criterion described below in \cref{thm-intro:induction_for_adm_Stokes_triple}.

\medskip

One could package the above results in the following

\begin{slogan-intro}
	For Stokes stratified spaces coming from flat bundles, the $\infty$-category  of Stokes functors inherits the properties of its coefficients.
\end{slogan-intro}

Note that the moduli functor from \cref{Main_theorem} parametrizes ``Stokes filtered perfect local systems''.
From this perspective,  classical Stokes filtered local systems correspond to objects concentrated in degree $0$.
It turns out that these can also be organized into an open substack $\bfStflat_{\cI,k} \subset \bfSt_{\cI,k}$ satisfying the following:

\begin{thm-intro}[\cref{thm:moduli_Stokes_vector_bundles}]\label{thm-intro:Stokes_filtered_local_systems}
	In the setting of \cref{Main_theorem}, the derived prestack $\bfStflat_{\cI,k} $ is a derived $1$-Artin stack locally of finite type.
\end{thm-intro}

In particular, the truncation of $\bfStflat_{\cI,k}$, namely its restriction to discrete $k$-algebra is an Artin stack locally of finite type in the classical sense.
See \cref{thm:comparison} and \cref{rem:comparison} for a comparison with wild character stacks and varieties in dimension $1$.

\medskip

The proofs of the main categorical results (Theorems \ref{stability_lim_colimi_thm} and \ref{smooth_nc_space}) heavily rely on a standard dévissage procedure in the classical theory of Stokes structures, known as level induction.
Before saying something more in this regard, let us emphasize that although stated in the context coming from flat bundles, the above theorems (\cref{Main_theorem} included) hold more generally for what we call \textit{families of Stokes analytic stratified spaces locally admitting a piecewise elementary level structure}.
To explain this, let us introduce the following:

\begin{defin-intro}\label{Stokes_stratified_intro}
	Let $M$ be a manifold. 
	Let $X\subset M$ be a locally closed subanalytic subset and let $X\to P$ be a subanalytic stratification.
	A \textit{Stokes fibration} over $(X,P)$ is a cocartesian fibration in posets $\cI \to \Pi_{\infty}(X,P)$.
	The data of $(X,P,\cI)$ is referred to as a \textit{Stokes analytic stratified space}.
\end{defin-intro}

Similarly to Stokes lines, one can define the Stokes loci:

\begin{defin-intro}\label{Stokes_locus_intro}
	Let $(X,P,\cI)$  be a Stokes analytic stratified space and let $a,b$ be cocartesian sections of $\cI \to \Pi_{\infty}(X,P)$.
	Then, the \textit{Stokes locus of $\{a,b\}$} is the set of points $x\in X$ such that $a_x$ and $b_x$ cannot be compared in $\cI_x$.
\end{defin-intro}

Elementarity is a crucial (albeit rare) property that in some sense provides the basis of the level induction procedure mentioned above.

\begin{defin-intro}\label{elementary_intro}
	We say that a Stokes analytic stratified space $(X,P,\cI)$ is \textit{elementary} if for every presentable stable $\infty$-category $\cE$, the  left Kan extension $i_{\cI ! } \colon   \Fun(\cI^{\ens},\cE)\to \Fun(\cI,\cE)$ induces an equivalence between $\St_{\cI^{\ens},\cE}$ and $\St_{\cI,\cE}$.
\end{defin-intro}

\begin{eg-intro}
	The Stokes analytic stratified space of \cref{eg:Stokes_on_interval_I} is elementary.
	See \cref{eg:1_dimensional} for a detailed explanation.
\end{eg-intro}

As for the inductive step, we introduce the following axiomatization of level structures in the classical theory of Stokes structures:

\begin{defin-intro}\label{level_intro_def}
Let $(X,P,\cI)$ be a Stokes analytic stratified space and let $p\colon \cI \to \cJ$ be a morphism of Stokes fibrations over $(X,P)$.
We say that $p\colon \cI \to \cJ$ is a \textit{level morphism} if for every $x\in X$ and every $a,b\in \cI_x$, we have 
\[ p(a) < p(b) \text{ in } \cJ_x \Rightarrow a < b \text{ in } \cI_x \ . \]
\end{defin-intro}

If we consider the fibre product $\pi \colon \cI_p \coloneqq  \cJ^{\ens} \times_{\cJ} \cI\to \cJ^{\ens}$, the classical level dévissage is traditionally used to reduce the study of $(X,P,\cI)$ to that of $(X,P,\cJ)$ and $(X,P,\cI_p)$.
This is effective since the level morphisms naturally occurring classically are so that $\cJ$ has less objects than $\cI$ while $\cI_p$ comes with extra properties.
This reduction procedure has a purely categorical explanation, which seems to be new already in the classical setting:

\begin{thm-intro}[{\cite[\cref*{Abstract_Stokes-prop:Level_induction}]{Abstract_Derived_Stokes}}]\label{level_dévissage_intro}
Let $(X,P,\cI)$ be a Stokes analytic stratified space and let $p\colon \cI \to \cJ$ be a level graduation morphism of Stokes fibrations over $(X,P)$.
Let $\cE$ be a presentable stable $\infty$-category.
Then, there is a pullback square
	\[ \begin{tikzcd}
		\St_{\cI, \cE} \arrow{r} \arrow{d}& \St_{\cJ,\cE} \arrow{d}\\
		\St_{\cI_p, \cE} \arrow{r} & \St_{\cJ^{\ens}, \cE}
	\end{tikzcd} \]
	in $ \CAT_{\infty} $.
\end{thm-intro}
The extra property of $(X,P,\cI_p)$ alluded to is what we call \textit{piecewise elementary} (see \cref{defin_piecewise_elementary}).
In a nutshell, it means that every point admits a subanalytic closed neighbourhood $Z$ such that the induced Stokes analytic stratified space $(Z,P,\cI_p|_Z)$ is elementary in the sense of \cref{elementary_intro}.
That one can find such  cover is typically possible when the differences of irregular values have the same pole order.
This follows from the following result, whose statement is inspired from \cite[Proposition 3.16]{MochStokes}:

\begin{thm-intro}[\cref{cor_induction_for_adm_Stokes_triple}]\label{thm-intro:induction_for_adm_Stokes_triple}
Let $(C,P,\cI)$ be a Stokes analytic stratified space in finite posets where $C\subset \mathbb{R}^n$ is a polyhedron and $\cI^{\ens}\to \Pi_{\infty}(C,P)$ is locally constant.
Assume that for every distinct cocartesian sections $a,b$ of $\cI\to \Pi_{\infty}(C,P)$, there exists a non zero affine form $\varphi \colon  \mathbb{R}^n \to \mathbb{R}$ such that 
\begin{enumerate}\itemsep=0.2cm
\item The Stokes locus of $\{a,b\}$  is $C\cap \{\varphi = 0\}$ (see \cref{Stokes_locus_intro}).

\item $C\setminus  \{\varphi = 0\}$  admits exactly two connected components $C_1$ and $C_2$.

\item $a_x < b_x$ in $\cI_x$ for every $x\in C_1$ and $a_x < b_x$ for every $x\in C_2$.
\end{enumerate}
Then $(C,P,\cI)$ is elementary.
\end{thm-intro}

\medskip

\paragraph*{\textbf{Linear overview}}

After reviewing the exodromy equivalence of \cite{Exodromy_coefficient,Beyond_conicality} in \S\ref{exodromy_recollection}, we introduce the notion of \emph{Stokes stratified space} and the related notion of \emph{Stokes loci} in \S\ref{sec:Stokes_stratified_spaces}.
Via the exponential construction, we introduce the constructible sheaf of Stokes data in \S\ref{sec:Stokes_sheaf}.
\S\ref{Global} is meant as a toolbox compiling the useful working properties of Stokes functors that can be extracted from \cite{Abstract_Derived_Stokes}.
In \S\ref{subsec:elementarity}, we introduce the fundamental notion of elementarity and its variants and we later prove a spreading out theorem for elementary subsets in the setting of Stokes analytic stratified spaces (see \cref{thm:spreading_out}).
Assuming the existence of a ramified piecewise linear level structure, we prove the main theorems concerning Stokes functors: that they form a presentable stable $\infty$-category (see \cref{thm:St_presentable_stable}), their non-commutative smoothness (see \cref{finite_typeness}) and the representability of the derived stack of Stokes structures (see \cref{Representability_via_toen_vaquie}).
In \S\ref{sec:polyhedral}, we develop the elementarity criterion based on the geometry of the Stokes loci (see \cref{cor_induction_for_adm_Stokes_triple}) and in \S\ref{sec:flat_bundles} we study the Stokes stratified spaces arising from flat bundles, notably establishing the existence of ramified piecewise linear level structures (see \cref{strongly_proper_piecewise_level_structure}).
Finally, in \cref{subsec:comparison}, we specialize in the setting of dimension $1$ and compare our construction with the one of \cite{Boalch_Topology_Stokes}.

\medskip

\paragraph*{\textbf{Acknowledgments}}

We are grateful to Enrico Lampetti, Guglielmo Nocera, Tony Pantev, Marco Robalo and Marco Volpe for useful conversations about this paper.
We especially thank Peter J.\ Haine for fruitful collaborations on the exodromy theorems.
We thank the Oberwolfach MFO institute that hosted the Research in Pairs ``2027r: The geometry of the Riemann-Hilbert correspondence''.
We also thank the CNRS for delegations and PEPS ``Jeunes Chercheurs Jeunes Chercheuses'' fundings, as well as the ANR CatAG from which both authors benefited during the writing of this paper.

\begin{notation}
In this paper,  $\cE$ will denote a presentable stable $\infty$-category.
\end{notation}



\section{Stratified spaces and constructible sheaves}\label{exodromy_recollection}

We begin giving a brief review of the exodromy correspondence \cite{Exodromy_coefficient,Beyond_conicality}.

\subsection{Atomic generation}
Let $\cC$ be a presentable $\infty$-category.
Recall that an object $c \in \cC$ is \textit{atomic} if the functor
\[
	\Map_{\cC}(c,-) \colon \cC\to  \Spc
\]
preserves \textit{all} colimits.
Write $\cC^{\at} \subset \cC$ for the full subcategory spanned by atomic objects.
We say that $\cC$ is \textit{atomically generated} if the unique colimit-preserving extension
\[
	 \PSh(\cC^{\at}) \hookrightarrow  \cC
\]
of $\cC^{\at} \subset \cC$ along the Yoneda embedding is an equivalence.

\subsection{Stratifications and hyperconstructible hypersheaves}

\begin{recollection}
If $P$ be a poset, we endow $P$ with the topology whose open subsets are the closed upward subsets $Q\subset P$.
     That is for every $a\in Q$ and $b\in P $ such that $b\geqslant a$, we have $b\in Q$.
\end{recollection}

\begin{defin}\label{notation_defin}
	Let $X$ be a topological space.
	Let $P$ be a poset.
	A  \textit{stratification of $X$ by $P$} is a continuous morphism $ X\to P$.
\end{defin}

\begin{rem}
	We abuse notations by denoting a stratification of $X$ by $P$ as $(X,P)$ instead of $ X\to P$ and refer to $(X,P)$ as a stratified space.
	The collection of stratified spaces organize into a category in an obvious manner.
\end{rem}



\begin{defin}
	Let $(X,P)$ be a  stratified space.
	An hypersheaf $F \colon \mathrm{Open}(X)\op\to \cE$ with value in $\cE$ is \textit{hyperconstructible} if for every $p \in P$, the hypersheaf $i_p^{\ast,\mathrm{hyp}}(F)$ is locally hyperconstant on $X_p$, where $i_p \colon X_p \to X$ denotes the canonical inclusion.
	We denote by 
	$$
	\ConsPhyp(X;\cE)\subset\Shhyp(X;\cE)
	$$ 
	the full-subcategory spanned by hyperconstructible hypersheaves on $(X,P)$.
\end{defin}

\subsection{Exodromic stratified spaces}

Following \cite{Clausen_Jansen,Beyond_conicality} we introduce the following

\begin{defin}
	A stratified space $(X,P)$ is said to be  \emph{exodromic} if it satisfies the following conditions:
	\begin{enumerate}\itemsep=0.2cm
		\item the $\infty$-category $\ConsPhyp(X)$ is atomically generated;
		\item the full subcategory $\ConsPhyp(X) \subset \Shhyp(X)$ is closed under limits and colimits;
		\item the functor $p^\ast \colon \Fun(P,\Spc) \to \ConsPhyp(X)$ commutes with limits.
	\end{enumerate}
	We denote by $\Stratc$ the category of exodromic stratified spaces with stratified morphisms between them.
\end{defin}

\begin{eg}[{\cite[Theorem 5.18]{Exodromy_coefficient}}]\label{conical_is_exodromic}
Every conically stratified space with locally weakly contractible strata is exodromic.
\end{eg}

\begin{defin}
	Let $(X,P)$ be an exodromic stratified space.
	We define the $\infty$-category of exit paths $\PiinftySigma(X,P)$ as the opposite of  the full subcategory of $\ConsPhyp(X)$ spanned by atomic objects.
\end{defin}

\begin{recollection}
Let $f \colon (X,P) \to (Y,Q)$ be a morphism between exodromic stratified spaces.
By \cite[Theorem 3.2.3]{Beyond_conicality} the functor $f^{\ast,\hyp} \colon \ConsQhyp(Y) \to \ConsPhyp(X)$ admits a left adjoint
\[ f^{\hyp}_\sharp \colon \ConsPhyp(X) \to  \ConsQhyp(Y) \]
preserving atomic objects.
It therefore induces a well defined functor
\[ \PiinftySigma(f) \colon \PiinftySigma(X,P) \to  \PiinftySigma(Y,Q) \ . \]
Using the equivalence $\PrLat \simeq \Catidem$, this can be promoted to a functor
\[ \PiinftySigma \colon \Stratc \to  \Cat_\infty \ . \]
\end{recollection}

\begin{recollection}\label{exodromy_functorialities}
For $(X,P)\in \Stratc $, there is a canonical equivalence 
\begin{equation}\label{exodromy_equivalence}
\Fun(\PiinftySigma(X,P), \Cat_\infty) \simeq \ConsPhyp(X,\Cat_\infty) 
\end{equation}
referred to as the \textit{exodromy equivalence}.
By \cite[Theorem 0.3.1]{Beyond_conicality}, the exodromy equivalence is  functorial.
Namely for every morphism $f \colon (X,P) \to (Y,Q)$  between exodromic stratified spaces, the following square 
	\[ \begin{tikzcd}
		\Fun(\PiinftySigma(Y,Q), \Cat_\infty)  \arrow{r}{\sim} \arrow{d}{\PiinftySigma(f)^*}  & \ConsQhyp(Y,\Cat_\infty) \arrow{d}{f^{\hyp,\ast}} \\
	\Fun(\PiinftySigma(X,P), \Cat_\infty) \arrow{r}{\sim} & \ConsPhyp(X,\Cat_\infty) 
	\end{tikzcd} \]
commutes.	
In particular, if $\cF \in \ConsPhyp(X,\Cat_\infty)$  corresponds to $F \colon \PiinftySigma(X,P) \to \Cat_\infty$ trough the exodromy equivalence, then are canonical equivalences
$\cF(X)\simeq \displaystyle{\lim_{\PiinftySigma(X,P) }} F$
and $\cF_x \simeq F(x)$ for every $x\in X$.
\end{recollection}

\begin{rem}\label{exodromy_PrL}
The exodromy equivalence and its functorialities also hold with coefficients in $\PrL$ (see \cite[Proposition 4.2.5]{Beyond_conicality}).
\end{rem}

\begin{prop}[{\cite[Theorem 3.3.6]{Beyond_conicality}}]\label{refinement_localization}
	Let $(X,P)$ be a stratified space and let $R\to P$ be a refinement such that $(X,R)$ is exodromic.
     Then,  $(X,P)$ is exodromic and the induced functor 
   \begin{equation} \label{eq:refinement_localization}
	\Pi_{\infty}(X,R) \to  \Pi_{\infty}(X,P)
	\end{equation}
	exhibits $ \Pi_{\infty}(X,P)$ as the  localization of $\Pi_{\infty}(X,R)$ at the set of arrows sent to equivalences by $\Pi_{\infty}(X,R) \to R\to P$.
	In particular,   \eqref{eq:refinement_localization} is final and cofinal.
\end{prop}


\begin{defin}[{\cite[Definition 5.2.4]{Beyond_conicality}}]\label{conically_refineable}
Let $(X,P)$ be a stratified space.
We say that $(X,P)$ is \textit{conically refineable} if there exists a refinement $R\to P$ such that $(X,R)$ is conically stratified with locally weakly contractible strata.
\end{defin}

\begin{rem}\label{conically_refineable_implies_exodromic}
A conically refineable stratified space is exodromic in virtue of \cref{conical_is_exodromic} and \cref{refinement_localization}.
\end{rem}

\begin{defin}\label{defin_final}
Let $(X,P)$ be an exodromic stratified space.
Let $Z\subset X$ be a locally closed subset such that $(Z,P)$ is exodromic.
Let $U\subset X$ be an open neighbourhood of $Z$.
We say that \textit{$U$ is final at $Z$} if $(U,P)$ is exodromic and if the  functor 
\[
\Pi_{\infty}(Z,P) \to  \Pi_{\infty}(U,P)
\]
is final.
\end{defin}



\begin{defin}[{\cite[Definition 2.3.2]{Exodromy_coefficient}}]\label{Excellent_at_S_stratified space}
		Let $(X,P)$ be an exodromic stratified space.
Let $Z\subset X$ be a locally closed subset such that $(Z,P)$ is exodromic.
	We say that $(X,P)$ is \emph{final at $Z$} if the collection of final at $Z$ open neighbourhoods of $Z$ forms a fundamental system of neighbourhoods of $Z$.
\end{defin}

\begin{defin}\label{def:hereditary_excellent}
Let $(X,P)$ be an exodromic stratified space.
Let $Z\subset X$ be a locally closed subset such that $(Z,P)$ is exodromic.
	We say that $(X,P)$ is \emph{hereditary final at $Z$} if  for every open subset $U \subseteq X$, the stratified space $(U,P)$ is final at $U\cap Z$.
\end{defin}

\subsection{Triangulations and hereditary finality} \label{subsec:triangulations_hereditary_excellent}

The goal of this subsection is to prove some hereditary final property for stratified spaces admitting a locally finite triangulation.
Before doing this, we need intermediate notations and lemmas.

\medskip

Let $K=(V,F)$ be a simplicial complex. 
We denote by $|K|$ the \textit{geometric realization of $K$}.
By construction, a point in $|K|$ is a function  $x\colon  V\to [0,1]$ supported on a face of $K$ and such that $\sum_{v\in V} x(v)=1$.
Let us endow the set of faces $F$ of $K$ with the inclusion. 
Let $\Supp_K  \colon  |K|\to F$ be the support function.

\begin{thm}[{\cite[Theorem A.6.10]{Lurie_Higher_algebra}}]\label{Exit_simplicial_complex}
Let $K=(V,F)$ be a locally finite simplicial complex. 
The stratified space $(|K|,F)$ is conically stratified with contractible strata and the structural morphism 
\[
	\Pi_{\infty}(|K|,F)\to  F
\]
is an equivalence of $\infty$-categories. 
\end{thm}

\begin{defin}\label{full_subcomplex}
Let $K$ be a simplicial complex and let $S$ be a simplicial subcomplex of $K$.
We say that $S$ is \textit{full} if for every face $\sigma$ of $K$, the subset $\sigma \cap S$ is empty or is a face of $S$.
\end{defin}

\begin{lem}\label{simplicial_complex_full_excellent}
	Let $K=(V,F)$ be a locally finite simplicial complex.
	Let $S=(V(S),F(S))$ be a full subcomplex of  $K$.
	Put
	\[
	U(S,K) \coloneqq \{x\in |K| \text{ such that } V(S)\cap \Supp_K(x) \neq \emptyset \} \ .
	\]
	Then, $U(S,K)$ is final at  $|S|$.
\end{lem}
\begin{proof}
	By \cref{Exit_simplicial_complex}, the category $\Pi_{\infty}(U(S,K),F)$ identifies with the subposet $P(S)$ of $F$ of faces  containing at least one vertex in $S$.
	We have to show that the inclusion $F(S)\to P(S)$ is final.
	Let $\sigma\in P(S)$.
	Then,  $F(S)\times_{P(S)}P(S)_{/\sigma}$ identifies with the poset of faces of $K$ contained in $S$ and $\sigma$.
	Since $\sigma$ contains at least one vertex of $S$, the poset $F(S)\times_{P(S)}P(S)_{/\sigma}$ is not empty.
	Since $S$ is full in $K$, we deduce that $F(S)\times_{P(S)}P(S)_{/\sigma}$ admits a maximal element, and is thus weakly contractible.
	This finishes the proof of \cref{simplicial_complex_full_excellent}.
\end{proof}

\begin{defin}
	Let $(X,P)$ be a stratified space.
	A \textit{triangulation of $(X,P)$} is the data of $(K,r)$ where $K=(V,F)$ is a simplicial complex and $r \colon (|K|, F) \to (X,P)$ is a refinement. 
	We say that $(K,r)$ is locally finite if $K$ is locally finite.
\end{defin}

The existence of a locally finite triangulation propagates to open subsets:

\begin{lem}[{\cite[Theorem 1]{Fuglede}}]\label{open_is_triangulable}
	Let $(X,P)$ be a stratified space and let $U\subset X$ be an open subset.
	If $(X,P)$ admits a locally finite triangulation, so does $(U,P)$.
\end{lem}


\begin{lem}\label{exellence_and_refinement}
	Let $r \colon  (X,P)\to (Y,Q)$ be a refinement between exodromic stratified spaces.
       Let $Z\subset Y$ be a locally closed subset and put $T\coloneqq r^{-1}(Z)$.		
	Let $U\subset X$ be an open subset final at $T$ (\cref{Excellent_at_S_stratified space}).
	Then $r(U)$ is final at $Z$.
	In particular, if $(X,P)$ is final at $T$, then $(Y,Q)$ is final at $Z$.
\end{lem}
\begin{proof}
	There is a commutative diagram of $\infty$-categories
	$$
	\begin{tikzcd}
		\Pi_{\infty}(T,P) \arrow{d} \arrow{r}  & \Pi_{\infty}(Z,Q)  \arrow{d}   \\
		\Pi_{\infty}(U,P)  \arrow{r}   &  \Pi_{\infty}(r(U),Q)
	\end{tikzcd}		
	$$
	where the left vertical functor is final.
	By \cref{refinement_localization}, the horizontal functors are localizations.
	They are thus final functors from \cite[7.1.10]{Cisinski_Higher_Category}.
	By \cite[4.1.1.3]{HTT}, we deduce that $\Pi_{\infty}(Z,Q) \to \Pi_{\infty}(r(U),Q)$ is  final.
	\cref{exellence_and_refinement} is thus proved.
\end{proof}


\begin{prop}\label{triangulability_and_excellence}
	Let $(X,P)$ be an exodromic stratified space admitting a locally finite triangulation.
	Then, for every locally closed subposet $Q\subset P$, $(X,P)$ is hereditary final at $X_Q$ (\cref{def:hereditary_excellent}).
\end{prop}
\begin{proof}
	Let $U\subset X$ be an open subset.
	We have to show that $(U,P)$ is final at $U\cap X_Q$.
	By \cref{open_is_triangulable},  $(U,P)$ admits a locally finite triangulation.
	At the cost of replacing $X$ by $U$, we are left to show that $(X,P)$ is final at $X_Q$.
	Write $Q=F\cap O$ where $F\subset P$ is closed and where $O\subset P$ is open.
	To show that $(X,P)$ is final at $X_Q$ amounts to show that $(X_O,O)$ is final at $X_Q$.
	From \cref{open_is_triangulable} again, we are left to show that $(X,P)$ is final at $X_Q$ where $Q\subset P$ is closed.
	Applying \cref{open_is_triangulable} one last time, we are left to show that there exists an open subset $U\subset X$  final at $X_Q$ where $Q\subset P$ is closed.
	Let $K=(V,F)$ be a locally finite simplicial complex and let $r \colon  (|K|,F)\to (X,P)$ be a refinement.
	Since $Q\subset P$ is closed, $r^{-1}(Q)\subset F$ is closed.
	Hence, $r^{-1}(Q)$ is the set of faces of a simplicial subcomplex $S=(V(S),F(S))$ of $K$.
	At the cost of replacing $K$ by its barycentric subdivision, we can suppose that $S$ is full (\cref{full_subcomplex}).
	By \cref{exellence_and_refinement}, it is enough to show that there exists an open subset $U\subset |K|$ containing $|S|$ such that $U$ is final at $|S|$.
	We conclude by  \cref{simplicial_complex_full_excellent}.
\end{proof}

\subsection{Subanalytic stratified space}
We now introduce the class of exodromic stratified spaces relevant for the study of Stokes structures of flat bundles.

\begin{defin}\label{subanalytic_stratified_space}
A \textit{subanalytic stratified space} is the data of $(M,X,P)$ where $M$ is a smooth real analytic space, $X\subset M$ a locally closed subanalytic subset and where $X\to P$ is a locally finite stratification by subanalytic subsets. 
\\ \indent
A morphism $f\colon (M,X,P)\to (N,Y,Q)$ of subanalytic stratified spaces is an analytic morphism $f \colon M\to N$ inducing a  stratified morphism $f\colon (X,P)\to (Y,Q)$ such that the graph of $f \colon X\to Y$ is subanalytic.
\end{defin}

\begin{notation}
We denote by $\AnStrat$ the category of subanalytic stratified spaces and  subanalytic stratified morphisms between them.
\end{notation}

\begin{rem}\label{Withney_are_conical}
	If  $(X,P)$ satisfies Whitney's conditions, a theorem of Mather \cite{Mather_topological_stability} implies that $(X,P)$ is conically stratified with locally weakly contractible strata.
	In that case we say that  $(M,X,P)$ is a \textit{Whitney stratified space}.
	Note that every subanalytic stratified space admits a Whitney refinement.
\end{rem}

\begin{rem}[{\cite[Theorem 5.3.9]{Beyond_conicality}}]\label{eg:subanalytic_implies_combinatorial}
	For every subanalytic stratified space $(X,P)$ and every open subset $U\subset X$, the stratified space $(U,P)$ is conically refineable in virtue of \cref{Withney_are_conical}.
	Hence it is exodromic by \cref{conically_refineable_implies_exodromic}.
\end{rem}

\begin{rem}
	For a subanalytic stratified space $(M,X,P)$, we will often drop the reference to $M$ and denote it by $(X,P)$.
\end{rem}

\begin{prop}[{\cite[Proposition 5.2.9]{Beyond_conicality}}]\label{prop:locally_contractible_strata}
	Let $(M,X,P)$ be a subanalytic stratified space.
	Then, every point $x\in X$ admits a fundamental system of open neighbourhoods $U$ such that $x$ is an initial object in $\Pi_\infty(U,P)$.
\end{prop}

\begin{prop}[{\cite[Theorem 5.3.9]{Beyond_conicality}}]\label{categorical_compactness}
Let $(M,X,P)$ be a subanalytic stratified space.
Assume that $X$ is relatively compact in $M$.
Then, $(X,P)$ is categorically finite, that is $\Pi_{\infty}(X,P)$ is a finite $\infty$-category.
\end{prop}

\begin{lem}\label{subanalytic_excellent}
	Let $(M,X,P)$ be a subanalytic stratified space.
	Then for every locally closed subset $Q\subset P$, $(X,P)$ is hereditary final at $X_Q$ (\cref{def:hereditary_excellent}).
\end{lem}
\begin{proof}
	By \cite{Goresky_triang}, the stratified space $(X,P)$ admits a locally finite triangulation.
	Then \cref{subanalytic_excellent} follows from \cref{triangulability_and_excellence}.
\end{proof}

\begin{prop}\label{proper_analytic_direct_image}
Let $f \colon  (M,X,P)\to (N,Y,Q)$ be a proper morphism between subanalytic stratified spaces.
Then the following hold
\begin{enumerate}\itemsep=0.2cm
\item There is a subanalytic refinement $S\to Q$ such that for $\cF\in \ConsPhyp(X;\Cat_{\infty})$, we have  $f_\ast(\cF)\in \Cons_{S}^{\hyp}(Y;\Cat_{\infty})$.
\item For every $\cF\in \ConsPhyp(X;\Cat_{\infty})$, the formation of $f_\ast(\cF)$ commutes with base change.

\end{enumerate}
\end{prop}

\begin{proof}
By \cite[1.7]{Goresky_MacPherson_Stratified_Morse_theory}, there is a refinement
	\[ \begin{tikzcd}
		(M,X,R)  \arrow{r} \arrow{d} & (M,X,P)  \arrow{d} \\
		(N,Y,S)   \arrow{r} & (N,Y,Q)  
	\end{tikzcd} \]
	by a morphism of Whitney stratified spaces  submersive on each strata.
	By Thom first isotopy lemma \cite{Mather_topological_stability}, we deduce that $ (X,R) \to (Y,S)$ is a stratified bundle above each stratum of $(Y,S)$.
	By \cite[3.7]{Verdier1976}, the fibres of $f$ are Whitney stratified spaces.
	They are thus conically stratified spaces with locally weakly contractible strata by \cref{Withney_are_conical}.
	By \cref{subanalytic_excellent}, for every locally closed subset $T\subset R$,  the stratified space $(X,R)$ is hereditary final at $X_T$.
	Hence, \cite[Proposition 6.10.7-(a)]{Exodromy_coefficient} shows that $S\to Q$ satisfies (1).
	To prove (2), it is enough to prove base change along the inclusion of a point.
	Then, one further reduces to the case where $f \colon  (M,X,P)\to (N,Y,Q)$ is a morphism of Whitney stratified spaces  submersive on each strata.
	In this case, (2) follows from \cite[Proposition 6.10.7-(b)]{Exodromy_coefficient}.
\end{proof}

\section{Stokes stratified spaces} \label{sec:Stokes_stratified_spaces}

Following the companion paper \cite{PortaTeyssier_Day} we introduce the $\infty$-category $\CoCart$.
We start from the cartesian fibration
\[ \mathrm{t} \colon \Cat_\infty^{[1]} \coloneqq \Fun(\Delta^1, \Cat_\infty) \to  \Cat_\infty \  \]
sending a functor $\cA \to \cX$ to its target $\infty$-category.
We then pass to the dual cocartesian fibration, in the following sense:

\begin{defin}\label{dual_cartesian}
	Let $p \colon \cA \to \cX$ be a cartesian fibration and let $\Upsilon_\cA \colon \cX\op \to \Cat_\infty$ be its  straightening.
	The \emph{dual cocartesian fibration $p^\star \colon \cA^\star \to \cX\op$} is the cocartesian fibration classified by $\Upsilon_\cA$.
\end{defin}

\begin{recollection}\label{recollection:dual_fibration}
	In the setting of \cref{dual_cartesian}, recall from \cite{Barwick_Dualizing} that objects of $\cA^\star$ coincide with the objects of $\cA$, while $1$-morphisms $a \to b$ in $\cA^\star$ are given by spans
	\[ \begin{tikzcd}[column sep = 20pt]
		a & c \arrow{l}[swap]{u} \arrow{r}{v} & b
	\end{tikzcd} \]
	where $u$ is $p$-cocartesian and $p(v)$ is equivalent to the identity of $p(b)$.
\end{recollection}

We let
\[ \mathbb B \colon \Cat_\infty^{[1]\star} \to  \Cat_\infty\op \]
be the cocartesian fibration dual to $\mathrm{t}$.
Specializing \cref{recollection:dual_fibration} to this setting, we see that objects of $\Cat_\infty^{[1]\star}$ are functors $\cA \to \cX$, and morphisms $\mathbf f = (f,u,v)$ from $\cB \to \cY$ to $\cA \to \cX$ are commutative diagrams in $\Cat_\infty$ of the form
\begin{equation}\label{eq:morphism_in_Cocart}
	\begin{tikzcd}
		\cB \arrow{d} & \cB_\cX \arrow{d} \arrow{r}{v} \arrow{l}[swap]{u} & \cA \arrow{dl} \\
		\cY & \cX \arrow{l}[swap]{f}
	\end{tikzcd}
\end{equation}
where the square is a pullback.
With respect to this description, $\mathbb B$ sends $\cA \to \cX$ to its target (or base) $\cX$, and a diagram as above defines a $\mathbb B$-cocartesian morphism if and only if $v$ is an equivalence.

\medskip

We define $\CoCart$ to be the (non-full) subcategory of $\Cat_\infty^{[1]\star}$ whose objects are cocartesian fibrations, and whose $1$-morphisms are commutative diagrams as above where $v$ is required to preserve cocartesian edges.
In this way, $\CoCart$ becomes a cocartesian fibration over $\Cat_\infty\op$ such that $\CoCart\to \Cat_\infty^{[1]\star}$ preserves cocartesian edges.
Notice that the fiber at $\cX \in \Cat_\infty\op$ coincides with the $\infty$-category $\CoCart_{/\cX}$.
We will also need a couple of variants of this construction:

\begin{variant}
 We let $\PosFib \subset \CoCart$ be the full subcategory spanned by those cocartesian fibrations $\cA \to \cX$ whose fibers are posets.
\end{variant}

\begin{variant}
	Let $\CAT_\infty$ be the $\infty$-category of large $\infty$-categories and consider the following fiber product:
	\[ \cC \coloneqq \Fun(\Delta^1, \CAT_\infty) \times_{\CAT_\infty} \Cat_\infty \ , \]
	where we used the target morphism $\mathrm t \colon \Fun(\Delta^1, \CAT_\infty) \to \CAT_\infty$.
	In other words, objects in $\cC$ are morphisms $p \colon \cA \to \cX$ where $\cX$ is a small $\infty$-category and the fibers of $p$ are not necessarily small $\infty$-categories.
	The induced morphism $\mathrm t \colon \cC \to \Cat_\infty$ is a cartesian fibration.
	Inside the dual cocartesian fibration $\cC^\star$, we define $\hCoCart$ as the subcategory spanned by cocartesian fibrations and whose $1$-morphisms are diagrams \eqref{eq:morphism_in_Cocart} where  $v$ preserves cocartesian edges.
\end{variant}		

\begin{variant}
We let $\PrFibL \subset \hCoCart$ be the subcategory spanned by cocartesian fibrations with presentable fibres and whose $1$-morphisms are diagrams \eqref{eq:morphism_in_Cocart} that are morphisms in $\hCoCart$ such that  for every $x \in \cX$, the induced functor $v_x \colon \cB_{f(x)} \to \cA_x$ is a morphism in $\PrL$, i.e.\  is cocontinuous.
$\PrFibL$ is the $\infty$-category of \textit{presentable cocartesian fibrations} \cite[\S3.4]{PortaTeyssier_Day}.
\end{variant}

\subsection{Stokes stratified spaces}

We are now ready to introduce the main geometric object of interest of this paper:

\begin{defin}\label{def:Stokes_stratified_spaces}
	The \emph{category of Stokes stratified spaces} $\StStrat$ is the fiber product
	\[ \begin{tikzcd}
		\StStrat \arrow{r} \arrow{d} & \PosFib\op \arrow{d}{\mathbb B\op} \\
		\Stratc \arrow{r}{\PiinftySigma} & \Cat_\infty \ .
	\end{tikzcd} \]
\end{defin}

\begin{rem}\label{rem:Stokes_stratified_spaces}
	It immediately follows from \cite[Proposition 2.4.4.2]{HTT} that mapping spaces in $\StStrat$ are discrete.
	Therefore \cite[Proposition 2.3.4.18]{HTT} guarantees that $\StStrat$ is (categorically equivalent to) a $1$-category.
\end{rem}

%
%

\begin{rem}\label{rem:cocartesian_fibration_vs_constructible_sheaf}
	Objects of $\StStrat$ can be explicitly described as triples $(X,P,\cI)$ where $(X,P)$ is an exodromic stratified space and $\cI \to \PiinftySigma(X,P)$ is a cocartesian fibration in posets.
	Combining the straightening equivalence \cite[Theorem 3.2.0.1]{HTT}
	\[ \CoCart_{/\PiinftySigma(X,P)} \simeq \Fun(\PiinftySigma(X,P), \Cat_\infty) \]
	with the exodromy equivalence \eqref{exodromy_equivalence}
	\[ \Fun(\PiinftySigma(X,P), \Cat_\infty) \simeq \ConsPhyp(X,\Cat_\infty) \ , \]
	we can equivalently describe the datum $\cI \to \PiinftySigma(X,P)$ as the datum of a hypersheaf of posets $\mathscr{I}$ on $X$ on  $(X,P)$.
	With respect to this translation, the stalk of $\mathscr{I}$ at a point $x \in X$ coincides with the fiber of $\cI$ at $x$ seen as an object in $\PiinftySigma(X,P)$.
	We occasionally refer to the datum of a cocartesian fibration in posets $\cI$ over $\Pi_\infty(X,P)$ as a \emph{Stokes fibration on $(X,P)$}.
\end{rem}

\begin{rem}\label{rem:cartesian_morphism_in_StStrat}
	The forgetful map $\StStrat \to \Stratc$ is a cartesian fibration, and a morphism $f \colon (Y, Q, \cJ) \to (X, P, \cI)$ is cartesian if and only if the   square
	\[ \begin{tikzcd}
		\cJ \arrow{r} \arrow{d} & \cI \arrow{d} \\
		\Pi_\infty(Y,Q) \arrow{r} & \Pi_\infty(X,P)
	\end{tikzcd} \]
	is a pullback.
\end{rem}

We will be interested in a more restricted class of Stokes stratified spaces:

\begin{defin}\label{def:Stokes_an_stratified_spaces}
	The $\infty$-category of \emph{Stokes analytic stratified spaces} $\StAnStrat$ is the fiber product
	\[ \begin{tikzcd}
		\StAnStrat \arrow{r} \arrow{d} & \StStrat \arrow{d}  \\
		\AnStrat \arrow{r} & \Stratc 
	\end{tikzcd} \]
	where  $ \AnStrat $ is the category of subanalytic stratified spaces from \cref{subanalytic_stratified_space} and the bottom horizontal functor is supplied by \cref{eg:subanalytic_implies_combinatorial}.
\end{defin}

\subsection{Stokes loci}

An important feature of the classical theory of Stokes data is the existence of Stokes lines.
Remarkably, it is possible to define Stokes loci for any Stokes stratified space $(X,P,\cI) \in \StStrat$, as we are going to discuss now.

\begin{defin}\label{cocartesian_section}
	For $(X,P,\cI)\in \StStrat$,  we denote by $\mathscr{I}$ the hyperconstructible hypersheaf on $(X,P)$ corresponding to the cocartesian fibration $\cI \to \Pi_\infty(X,P)$ as in \cref{rem:cocartesian_fibration_vs_constructible_sheaf}.
	The objects of
	\[ \mathscr{I}(X)\simeq \Funcocart_{/\Pi_{\infty}(X,P)}(\Pi_{\infty}(X,P),\cI) \]
	are the \textit{cocartesian sections of $\cI$ over $\Pi_\infty(X,P)$}.
\end{defin}

\begin{defin}\label{Stokes_locus}
	Let $(X,P,\cI)$ be a Stokes analytic stratified space.
	Let $\sigma, \tau\in \mathscr{I}(X)$ be cocartesian sections.
	The \textit{Stokes locus $X_{\sigma, \tau}$ of $\sigma, \tau$} is the set of points $x\in X$ such that $\sigma(x), \tau(x)\in \cI_x$ cannot be compared.
\end{defin}

\begin{observation}\label{pullback_Stokes_locus}
	Let $f \colon (Y,Q,\cJ)\to (X,P,\cI)$ be a cartesian morphism between Stokes analytic stratified spaces (see \cref{rem:cartesian_morphism_in_StStrat}).
	Let $\sigma, \tau\in \mathscr{I}(X)$ be cocartesian sections.
	Then, we have
	\[ Y_{f^*\sigma,f^*\tau}=f^{-1}(X_{\sigma,\tau}) \ . \]
\end{observation}

\begin{lem}\label{Stokes_locus_is_closed}
	Let $(X,P,\cI)$ be a Stokes analtyic stratified space.
	Let $\sigma, \tau\in \mathscr{I}(X)$ be cocartesian sections.
	Then,
	\begin{enumerate}\itemsep=0.2cm
		\item $X_{\sigma, \tau}$ is closed in $X$.
		\item For every $p\in P$, the set $X_{\sigma, \tau}\cap X_p$ is open and closed in $X_p$.
		In particular, $X_{\sigma, \tau}$ is a union of connected components of strata of $(X,P)$.
		
	\end{enumerate}
\end{lem}
\begin{proof}
	At the cost of refining $(X,P)$ by a Whitney stratified space, \cref{pullback_Stokes_locus} implies that we can suppose $(X,P)$ to be conically stratified with locally weakly contractible strata.
	Let $x\in X-X_{\sigma, \tau}$.
	We can suppose that $\sigma(x)\leq \tau(x)$ in $\cI_x$.
	Since the strata of $(X,P)$ are locally weakly contractible, \cref{prop:locally_contractible_strata} yields the existence of an open subset $U\subset X$ containing $x$ such that $x$ is an initial object of $\Exit(U,P)$.
	Hence, for every $y\in U$, there is an exit path $\gamma \colon x\to y$ giving rise to a morphism of posets $\cI_x \to \cI_y$ sending $\sigma(x)$ to $\sigma(y)$ and $\tau(x)$ to $\tau(y)$.
	Thus $\sigma(y)\leq \tau(y)$.
	Hence $U \subset X-X_{\sigma, \tau}$.
	This proves $(1)$.
	We now prove $(2)$.
	From \cref{pullback_Stokes_locus}, we can suppose that $X$ is trivially stratified and show that  $X_{\sigma, \tau}$  is open and closed in $X$.
	From $(1)$, it is enough to show that $X_{\sigma, \tau}$ is open in $X$.
	Let $x\in X_{\sigma, \tau}$ and let $U\subset X$ be an open subset containing $x$ such that $x$ is an initial object of $\Exit(U,P)$.
	Let $y\in U$.
	Let $\gamma \colon x \to y$ be a path.
	Since the stratification is trivial, $\gamma$ is an isomorphism.
	Thus $\gamma$ gives rise to an \textit{isomorphism} of posets $\cI_x \to \cI_y$ sending $\sigma(x)$ to $\sigma(y)$ and $\tau(x)$ to $\tau(y)$.
	Since $\sigma(x),\tau(x)\in \cI_x$ cannot be compared, nor do $\sigma(y),\tau(y)\in \cI_y$. 
	Hence, $U \subset X_{\sigma, \tau}$.
	The proof of \cref{Stokes_locus_is_closed} is thus complete.
\end{proof}

\section{The filtered and the Stokes hyperconstructible hypersheaves}\label{sec:Stokes_sheaf}

Given a Stokes stratified space $(X,P,\cI)$, we attach to it two hyperconstructible hypersheaves of $\infty$-categories on $(X,P)$.

\begin{construction}\label{construction:exponential}
	Let $p \colon \cA \to \cX$ be a cocartesian fibration.
	Let $\Upsilon_\cA \colon \cX \to \Cat_\infty$ be its straightening and consider the functor
	\[ \Fun_!(\Upsilon_\cA(-), \cE) \colon \cX \to  \PrL \ , \]
	where $\Fun_!$ denotes the functoriality given by left Kan extensions.
	We write
	\[ \exp_\cE(\cA/\cX) \to \cX \]
	for the presentable cocartesian fibration classifying $\Fun_!(\Upsilon_\cA(-),\cE)$.
	We refer to $\exp_\cE(\cA/\cX)$ as the \emph{exponential fibration with coefficients in $\cE$ associated to $p \colon \cA \to \cX$}.
From \cite[Variant 3.20 \& Remark 3.21]{PortaTeyssier_Day} this construction can be promoted to an $\infty$-functor
\[ \exp_\cE \colon \CoCart \to  \PrFibL \ . \]
In more concrete terms, we have 
\begin{equation}\label{eq:exp_on_morphisms}
	\begin{tikzcd}
		\cB \arrow{d} & \cB_\cX \arrow{d} \arrow{r}{v} \arrow{l}[swap]{u} & \cA \arrow{dl} \\
		\cY & \cX \arrow{l}[swap]{f}
	\end{tikzcd} \stackrel{\exp_\cE}{\longmapsto} \begin{tikzcd}
		\exp_\cE(\cB / \cY) \arrow{d} & \exp_\cE(\cB_\cX / \cX) \arrow{l}[swap]{\cE^u} \arrow{r}{\cE^v_!} \arrow{d} & \exp_\cE(\cA / \cX) \arrow{dl} \\
		\cY & \cX \arrow{l}[swap]{f} 
\end{tikzcd}
\end{equation}
where $\cE^u$ makes the right square a pullback and $\cE^v_!$ preserves cocartesian edges in virtue of \cite[\cref*{Abstract_Stokes-prop:functoriality_exponential}]{Abstract_Derived_Stokes}.
\end{construction}

\subsection{The hyperconstructible hypersheaves of filtered functors}

\begin{observation}\label{observation:exodromy_straightening}
   By \cref{exodromy_PrL} and \cite[\cref*{Abstract_Stokes-PrL_straightening}]{Abstract_Derived_Stokes}, we have canonical equivalences
	\[ \ConsP(X;\PrL) \simeq \Fun(\PiinftySigma(X,P), \PrL) \simeq \PrFibL_{\PiinftySigma(X,P)} \ . \]
	These equivalences give rise to the following canonically commutative diagram:
	\[ \begin{tikzcd}
		\ConsPhyp(X;\PrL) \arrow{r}{\sim} \arrow{dr}[swap]{\Gamma_{X,\ast}} & \Fun(\PiinftySigma(X,P),\PrL) \arrow{d}{\lim} & \PrFibL_{\PiinftySigma(X,P)} \arrow{l}[swap]{\sim} \arrow{dl}{\Sigma^{\cocart}} \\
		{} & \PrL
	\end{tikzcd} \]
	where 
	\[
\Sigma^{\cocart}(\cA / \PiinftySigma(X,P)) \coloneqq \Fun_{/\PiinftySigma(X,P)}^{\cocart}(\PiinftySigma(X,P), \cA) 
\]
is the presentable $\infty$-category of cocartesian sections.	
	Similar considerations hold if we replace $\PrL$ by $\Cat_\infty$ or by $\CAT_\infty$.
\end{observation}

\begin{defin}\label{def:cocartesian_filtered_functor}
	Let $(X,P,\cI)$ be a Stokes stratified space.
	The \emph{categorical hypersheaf of $\cI$-filtered functors on $(X,P)$ with coefficients in $\cE$} is the object $\mathfrak{Fil}_{\cI,\cE}$ in $\ConsPhyp(X;\PrL)$ corresponding to $\exp_\cE(\cI / \PiinftySigma(X,P))$ via the equivalences of \cref{observation:exodromy_straightening}.
	The $\infty$-category of \emph{cocartesian $\cI$-filtered functors on $(X,P)$} is the presentable $\infty$-category
	\[ \Filco_{\cI,\cE} \coloneqq \frakFil_{\cI,\cE}(X)  \]
	of global sections of $\frakFil_{\cI,\cE}$.
\end{defin}

\begin{rem}\label{rem:filtered_functors}
    Let $ (X,P,\cI)$ be a Stokes stratified space.
	We can give an explicit description of the hypersheaf $\frakFil_{\cI,\cE}$ as follows.
	For every open subset $U\subset X$, write
	\[ j_U \colon \PiinftySigma(U,P) \to  \PiinftySigma(X,P) \]
	for the canonical map.
	Let $\Upsilon_{\cI} \colon \PiinftySigma(X,P) \to \Poset$ be the straightening of $\cI$.
	Unraveling the equivalences of \cref{observation:exodromy_straightening}, we can identify $\frakFil_{\cI,\cE}$ with the presheaf $ \mathrm{Open}(X)\op \to  \PrL$ sending an open subset $U\subset X$ to
	\[ \lim_{\PiinftySigma(U,P)} \Fun_!(\Upsilon_{\cI}\circ j_U(-), \cE)  \ .\]
	It is not obvious from this description that $\frakFil_{\cI,\cE}$ satisfies hyperdescent nor that it is $P$-hyperconstructible: it is rather a consequence of the exodromy equivalence.
\end{rem}

\begin{eg}\label{description_filco_initial_object}
	    Let $ (X,P,\cI)$ be a Stokes stratified space.
	    Let $U\subset X$ be an open subset such that  $\PiinftySigma(U,P)$ admits an initial object $x$.
	    Then the description of $\frakFil_{\cI,\cE}$ given in \cref{rem:filtered_functors} yields a canonical equivalence $\frakFil_{\cI,\cE}(U)\simeq \Fun(\cI_x,\cE)$.
\end{eg}		
		

In the trivial stratification situation, $\mathfrak{Fil}_{ \Pi_{\infty}(X),\cE}$ gives back locally constant hypersheaves.
Before seeing this, let us introduce the following

\begin{defin}\label{def_loc_presheaf}
Let $X$ be a topological space. 
We denote by 
$$
\frak{Loc}_{X,\cE} \colon \mathrm{Open}(X)\op \to \cE
$$ 
the presheaf sending an open subset $U\subset X$ to  $\Loc^{\hyp}(U,\cE)$.
\end{defin}

\begin{prop}\label{Filco_trivial_stratification}
Consider a Stokes stratified space of the form $(X,\ast, \Pi_{\infty}(X))$.
Then, $\mathfrak{Fil}_{ \Pi_{\infty}(X),\cE}$ is canonically equivalent to $\frak{Loc}_{X,\cE}$.
\end{prop}
\begin{proof}
In that case, $\exp_\cE(\Pi_{\infty}(X) / \Pi_{\infty}(X))$ is the constant fibration $\Pi_{\infty}(X)\times \cE   \to \Pi_{\infty}(X)$.
Since $\Pi_{\infty}(X)$ is an $\infty$-groupoïd, every  section of 
$\exp_\cE(\Pi_{\infty}(X) / \Pi_{\infty}(X))$ is cocartesian.
Thus \cref{rem:filtered_functors}  yields a canonical equivalence
\[
\mathfrak{Fil}_{ \Pi_{\infty}(X),\cE}(U)\simeq \Fun(\Pi_{\infty}(U),\cE)  
\]
for every $U\in   \mathrm{Open}(X)$. 
Since $X$ is exodromic, the conclusion follows from the monodromy equivalence.
\end{proof}

\subsection{The hyperconstructible hypersheaves of Stokes functors}

The categorical  hypersheaf $\frakFil_{\cI,\cE}$ is not yet our main object of interest.

\begin{notation}\label{notation_Iset}
	We let
	\[ (-)^{\mathrm{set}} \colon \Poset \to \Poset \]
	be the functor sending a poset $(I,\leqslant)$ to the underlying set $I$, seen as a poset with trivial order.
	This construction promotes to a global functor
	\[ (-)^{\mathrm{set}} \colon \PosFib \to  \PosFib \ , \]
	equipped with a natural transformation $i \colon (-)^{\mathrm{set}} \to \id_{\PosFib}$.
\end{notation}

Let $(X,P,\cI)$ be Stokes stratified space.
The functoriality of the exponential construction induces a well defined exponential induction functor
\[ \cE^{i_{\cI}}_! \colon \exp_\cE(\cI^{\ens} / \PiinftySigma(X,P)) \to  \exp_\cE(\cI / \PiinftySigma(X,P))  
\]
in  $\PrFibL_{\PiinftySigma(X,P)}$.
Let 
$$
\exp_\cE^{\PS}( \cI / \Pi_\infty(X,P) ) \subset  \exp_\cE(\cI / \PiinftySigma(X,P))  
$$
be its essential image.
One can check  (see  \cite[\cref*{Abstract_Stokes-EssIm}]{Abstract_Derived_Stokes}) that 
$$
\exp_\cE^{\PS}( \cI / \Pi_\infty(X,P) ) \to \PiinftySigma(X,P)
$$ 
is a cocartesian fibration whose formation commutes with base change.
In particular, its fiber at $x \in X$ canonically coincide with the essential image of 
$$
i_{\cI_x ! } \colon   \Fun(\cI^{\ens}_x,\cE)\to \Fun(\cI_x,\cE) \ .
$$


\begin{defin}\label{def:Stokes_sheaf}
Let $(X,P,\cI)$ be Stokes stratified space.
	The \emph{categorical sheaf of $\cE$-valued $\cI$-Stokes functors on $(X,P)$} is the object $\frakStokes_{\cI,\cE}$ in $\ConsP(X;\CAT_\infty)$ corresponding to $\exp_\cE^{\PS}( \cI / \Pi_\infty(X,P) )$ via the equivalences of \cref{observation:exodromy_straightening}.
	The \emph{$\infty$-category of $\cE$-valued $\cI$-Stokes functors} is the (large) $\infty$-category
	\[ \St_{\cI,\cE} \coloneqq \frakStokes_{\cI,\cE}(X) \in \CAT_\infty \]
	of global sections of $\frakStokes_{\cI,\cE}$.
\end{defin}

\begin{rem}
If the fibers of $\cI$ are discrete, then $i_{\cI}\colon\cI^{\ens}\to \cI$ is an equivalence.
Thus, $\St_{\cI,\cE} \simeq \Filco_{\cI,\cE}$.
\end{rem}


\begin{eg}\label{eg:local_systems_as_Stokes_structures}
	Let $(X,P)$ be an exodromic stratified space.
	Review it as a Stokes stratified space $(X,P,\Pi_\infty(X,P))$, with the trivial cocartesian fibration given by the identity of $\Pi_\infty(X,P)$.
	Then there is a canonical equivalence
	\[ \frakStokes_{\Pi_{\infty}(X,P),\cE} \simeq \frak{Loc}_{X,\cE} \ , \]
	where $\frak{Loc}_{X,\cE}$ is the categorical sheaf of locally hyperconstant hypersheaves on $X$ (see \cref{def_loc_presheaf}).
	For a proof, see \cref{Stokes_sheaf_trivial_fibration}.
	In other words, Stokes functors provide an extension of the theory of locally hyperconstant hypersheaves.
\end{eg}

At the other extreme, we have:

\begin{eg}\label{eg:Stokes_structures_at_a_point}
	Let $(X,P,\cI)$ be a Stokes stratified space.
	Assume that $\PiinftySigma(X,P)$ admits an initial object $x$.
	Then, in virtue of \cref{rem:filtered_functors}, the pullback over $x$ induces an equivalence between $\St_{\cI,\cE}$ and $\St_{\cI_x,\cE}$, that is the essential image of
	\[ i_{\cI_x,!} \colon \Fun(\cI_x^{\mathrm{set}}, \cE) \to  \Fun(\cI_x,\cE) \ . \]
	Notice that  this   essential image is typically \emph{not} stable nor presentable.
\end{eg}


The following example is a particularly simple situation in dimension $1$, but covers a large part of the ideas covered in this paper.

\begin{eg}\label{eg:1_dimensional}
	On the circle $S^1 = \{ z \in \C \mid |z| = 1 \}$ consider the stratification over the poset $P = \{0 < 1\}$ whose closed stratum is $\{1,-1\}$.
	Write
	\[ U \coloneqq \{z \in S^1 \mid \Im(z) > 0\} \qquad \textrm{and} \qquad V \coloneqq \{z \in S^1 \mid \Im(z) < 0\} \ . \]
	Consider the $P$-constructible sheaf of posets $\mathscr I$ whose underlying sheaf of sets $\mathscr I^{\mathrm{set}}$ is the constant sheaf associated to $\{a,b\}$, and whose order is determined by the requirement that $a < b$ over $U$ and $b < a$ over $V$, while $a$ and $b$ are not comparable at $1$ and $-1$.
	The situation can be visualized as follows:
	
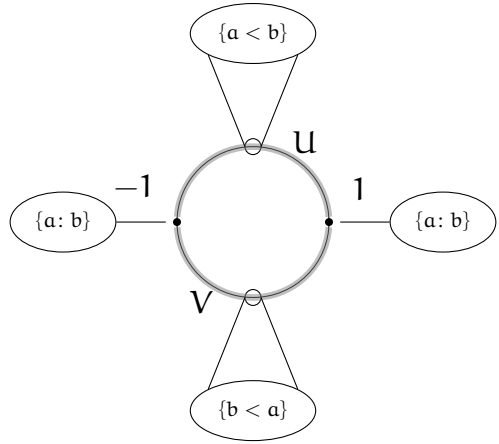
\begin{center}
		\begin{tikzpicture}
			\node[label=45:{$1$}] (X1) at (1,0) {} ;
			\node[shape=circle,shape=circle,minimum size=6pt,inner sep=0pt, outer sep=0pt,draw] (X2) at (0,1) {} ;
			\node[label=135:{$-1$}] (X3) at (-1,0) {} ;
			\node[shape=circle,minimum size=6pt,inner sep=0pt, outer sep=0pt,draw] (X4) at (0,-1) {} ;
			\node[shape = ellipse,draw] (A1) at (2.5,0) {\tiny $\{a \colon b\}$} ;
			\node[shape = ellipse,draw] (A2) at (0,2.5) {\tiny $\{a < b\}$} ;
			\node[shape = ellipse,draw] (A3) at (-2.5,0) {\tiny $\{a \colon b\}$} ;
			\node[shape = ellipse,draw] (A4) at (0,-2.5) {\tiny $\{b < a\}$} ;
			\draw[thin] (0,0) circle (1cm) ;
			\draw[line width = 2.5pt,gray,opacity=0.5] (4:1cm) arc (4:176:1cm) node[near start,above,color=black,opacity=100] {$U$} ;
			\draw[line width = 2.5pt,gray,opacity=0.5] (184:1cm) arc (184:356:1cm) node[near start,below,color=black,opacity=100] {$V$} ;
			\fill (1,0) circle (1.5pt) ;
			\fill (-1,0) circle (1.5pt) ;
			\draw (X1.east) -- (A1.west) ;
			\draw (X2.west) -- (A2.south west) ;
			\draw (X2.east) -- (A2.south east) ;
			\draw(X3.west) -- (A3.east) ;
			\draw(X4.west) -- (A4.north west) ;
			\draw(X4.east) -- (A4.north east) ;
		\end{tikzpicture}
	\end{center}
	After applying the exodromy and the straightening equivalence, we are left with the following cocartesian fibration in posets over $\PiinftySigma(S^1,P)$:
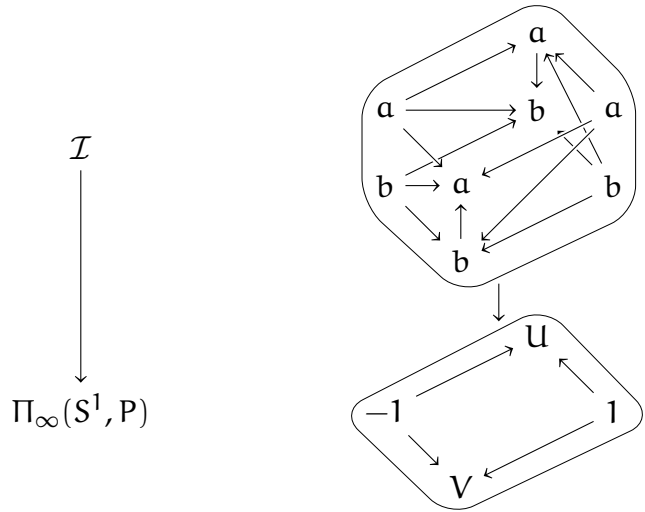
	\begin{center}
		\begin{tikzpicture}
			\node (E) at (-5,0) {$\PiinftySigma(S^1,P)$} ;
			\node (I) at (-5,3.5) {$\cI$} ;
			\node (U0) at (1,1) {$U$} ;
			\node (U1) at (1,4) {$b$} ;
			\node (U2) at (1,5) {$a$} ;
			\node (I0) at (2,0) {$1$} ;
			\node (I1) at (2,3) {$b$} ;
			\node (I2) at (2,4) {$a$} ;
			\node (II0) at (-1,0) {$-1$} ;
			\node (II1) at (-1,3) {$b$} ;
			\node (II2) at (-1,4) {$a$} ;
			\node (V0) at (0,-1) {$V$} ;
			\node (V1) at (0,2) {$b$} ;
			\node (V2) at (0,3) {$a$} ;
			
			\draw[rounded corners = 10pt,very thin] (-0.15,-1.4) -- (2.55,-0.1) -- (1.1,1.4) -- (-1.6,0) -- cycle ;
			\draw[->] (I0) -- (V0) ;
			\draw[->] (I0) -- (U0) ;
			\draw[->] (II0) -- (U0) ;
			\draw[->] (II0) -- (V0) ;
			
			\draw[rounded corners = 10pt,very thin] (0,1.55) -- (2.3,2.65) -- (2.3,4.3) -- (1.1,5.5) -- (-1.3,4.25) -- (-1.3,2.75) -- cycle ;

			\draw[->] (I.south) -- (E.north) ;
			\draw[->] (0.5,1.7) -- (0.5,1.2) ;
			\draw[->] (II1) -- (V1) ;
			\draw[->] (II1) -- (U1) ;
			\draw[->] (I1) -- (U1) ;
			\draw[->] (I1) -- (V1) ;
			\draw[->] (I1) -- (U2) ;
			\draw[line width=2pt,white] (II2) -- (V2) ;
			\draw[line width=2pt,white] (I2) -- (V2) ;
			\draw[->] (II2) -- (V2) ;
			\draw[->] (II2) -- (U2) ;
			\draw[->] (I2) -- (U2) ;
			\draw[->] (I2) -- (V2) ;
			\draw[line width=2pt,white] (I2) -- (V1);
			\draw[->] (I2) -- (V1) ;
			\draw[->] (V1) -- (V2) ;
			\draw[->] (U2) -- (U1) ;
			\draw[->] (II2) -- (U1) ;
			\draw[->] (II1) -- (V2) ;
		\end{tikzpicture}
	\end{center}
	Beware that different copies of $a$ and $b$ represent different objects in $\cI$, lying over different objects of $\PiinftySigma(X,P)$.
	Arrows between identical letters correspond to cocartesian edges in $\cI$.
	Take $\cE \coloneqq \categ{Mod}_k$, where $k$ is some field.
	Then both $\Filco_{\cI,\cE}$ and $\St_{\cI,\cE}$ can be realized as full subcategories of $\Fun(\cI,\categ{Mod}_k)$.
	Although practical for many purposes, this is not the best way to handle these categories.
	Let us explain in this example how to exploit the sheaf theoretic nature of both $\Filco_{\cI,\cE}$ and $\St_{\cI,\cE}$.
	Define the two opens
	\[ W_1 \coloneqq \{ z \in S^1 \mid \Re(z) > -1\} \qquad \text{and} \qquad W_{-1} \coloneqq \{ z \in S^1 \mid \Re(z) < 1 \} \ , \]
	and let $W \coloneqq W_1 \cap W_{-1}$ be their intersection.
	For $i \in \{1,-1\}$, put
	\[ \cI_{W_i} \coloneqq \PiinftySigma(W_i) \times_{\PiinftySigma(S^1,P)} \cI \ . \]
	Since $\frakFil_{\cI, \cE}$ and $\frakStokes_{\cI,\cE}$ are sheaves, we deduce that the squares
	\[ \begin{tikzcd}
		\Filco_{\cI,\cE} \arrow{r} \arrow{d} & \Filco_{\cI_{W_1},\cE} \arrow{d} \\
		\Filco_{\cI_{W_{-1}},\cE} \arrow{r} & \Filco_{\cI_{W},\cE}
	\end{tikzcd} \qquad \text{and} \qquad \begin{tikzcd}
		\St_{\cI,\cE} \arrow{r} \arrow{d} & \St_{\cI_{W_1},\cE} \arrow{d} \\
		\St_{\cI_{W_{-1}},\cE} \arrow{r} & \St_{\cI_W,\cE} 
	\end{tikzcd} \]
	are pullbacks.
	Now, observe that:
	\begin{enumerate}[(i)]\itemsep=0.2cm
		\item since $1$ is initial in $\PiinftySigma(W_1,P)$, we have $\Filco_{\cI_{W_1},\cE} \simeq \Fun(\cI_1, \cE) \simeq \cE \times \cE$ and $\St_{\cI_{W_1},\cE}\simeq \St_{\cI_{1},\cE}$;
		\item since the order on $\cI_1 = \{a \colon b\}$ is trivial, we have $\cI_1^{\mathrm{set}} = \cI_1$, and therefore $\St_{\cI_{W_1},\cE} = \Filco_{\cI_{W_1},\cE}$.
	\end{enumerate}
	A symmetrical reasoning applies with $-1$ in place of $1$.
	Full faithfulness of $\St_{\cI_W,\cE} \hookrightarrow \Filco_{\cI_W,\cE}$ ensures that the induced map
	\[ \St_{\cI_{W_1},\cE} \times_{\St_{\cI_W,\cE}} \St_{\cI_{W_{-1}},\cE} \to  \St_{\cI_{W_1},\cE} \times_{\Filco_{\cI_W,\cE}} \St_{\cI_{W_{-1}},\cE} \]
	is an equivalence.
	Hence, the canonical map
	\[ \St_{\cI,\cE} \to  \Filco_{\cI,\cE} \]
	is an equivalence.
	In particular, $\St_{\cI,\cE}$ is stable.
\end{eg}

\section{Stokes functors: the global viewpoint}\label{Global}

\subsection{The specialization equivalence}

Let $(X,P,\cI)$ be a Stokes stratified space.
Let
$$
\Sigma(\exp_\cE(\cI/\PiinftySigma(X,P)))\coloneqq \Fun_{/\PiinftySigma(X,P)}(\PiinftySigma(X,P),\exp_\cE(\cI/\PiinftySigma(X,P)))
$$
be the $\infty$-category of sections of $\exp_\cE(\cI/\PiinftySigma(X,P))$.
Then,  there is an equivalence of $\infty$-categories, called the \textit{specialization equivalence} (See \cite[Proposition 4.1]{PortaTeyssier_Day})
\begin{equation}\label{spe_eq}
\spe_{\cI}\colon\Fun(\cI,\cE) \simeq 	\Sigma(\exp_\cE(\cI/\PiinftySigma(X,P)))  \ .
\end{equation}

The functorialities of the exponential construction admit a simple description in terms of $\Fun(\cI,\cE)$.
Let $f \colon (Y, Q, \cJ) \to (X,P,\cI)$ be a morphism in $\StStrat$. 
Recall that $f$ amounts to the datum of a morphism of exodromic stratified spaces $f \colon (Y,Q) \to (X,P)$ and a commutative diagram
	\[ \begin{tikzcd}
		\cI \arrow{d} & \cI_Y \arrow{d} \arrow{r}{v} \arrow{l}[swap]{u} & \cJ \arrow{dl} \\
		\Pi_\infty(X,P) & \Pi_\infty(Y,Q) \arrow{l}[swap]{\Pi_\infty(f)} 
	\end{tikzcd} \]
	where the square is cartesian.

\begin{prop}[{\cite[\cref*{Abstract_Stokes-prop:global_functoriality}]{Abstract_Derived_Stokes}}]\label{prop:global_functoriality_IHES}
	\hfill
	\begin{enumerate}\itemsep=0.2cm
		\item \label{prop:global_functoriality:pullback} There exists a canonically commutative square
\[
			\begin{tikzcd}
				\Sigma(\exp_\cE(\cI/\PiinftySigma(X,P))) \arrow{r}{\Sigma(\cE^u)} \arrow{d}{\spe_{\cI}} &\Sigma(\exp_\cE(\cI_Y/\PiinftySigma(Y,Q))) \arrow{d}{\spe_{\cI_Y}} \\
				\Fun(\cI,\cE) \arrow{r}{u^\ast} & \Fun(\cI_Y, \cE) \ ,
			\end{tikzcd}
\]

		providing a canonical identification $\Sigma(\cE^u) \simeq u^\ast$.
		
		\item \label{prop:global_functoriality:induction} There exists a canonically commutative square
               \[
			\begin{tikzcd}
			\Sigma(\exp_\cE(\cI_Y/\PiinftySigma(Y,Q))) \arrow{d}{\spe_{\cI_Y}}  \arrow{r}{\Sigma(\cE^v_!)} \arrow{d}{\spe_{\cI_Y}} &	\Sigma(\exp_\cE(\cJ/\PiinftySigma(Y,Q)))\arrow{d}{\spe_{\cJ}} \\
				\Fun(\cI_Y, \cE) \arrow{r}{v_!} & \Fun(\cJ,\cE) \ ,
			\end{tikzcd}
			\]
		providing a canonical identification $\Sigma(\cE^v_!) \simeq v_!$.
	\end{enumerate}
\end{prop}

\subsection{Stokes functors as functors}

From the specialization equivalence \eqref{spe_eq}, one can review
$$
 \St_{\cI,\cE} \coloneqq \Sigma^{\cocart}(\exp_\cE^{\PS}( \cI / \Pi_\infty(X,P) )) \subset \Sigma(\exp_\cE(\cI/\PiinftySigma(X,P)))
$$
and 
$$  
\Filco_{\cI,\cE}   \coloneqq   \Sigma^{\cocart}(\exp_\cE(\cI/\PiinftySigma(X,P))) \subset \Sigma(\exp_\cE(\cI/\PiinftySigma(X,P))) 
$$
as full subcategories of $\Fun(\cI,\cE)$.
From this perspective, we will write $\Funcocart(\cI,\cE)$ instead of $\Filco_{\cI,\cE}$.
The following proposition provides a description of cocartesian functors intrinsic to
$\Fun(\cI,\cE)$.

\begin{prop}[{\cite[\cref*{Abstract_Stokes-prop:cocartesian_functors_reformulation}]{Abstract_Derived_Stokes}}]
	Let $F \colon \cI \to \cE$ be a functor.
	The following are equivalent:
	\begin{enumerate}\itemsep=0.2cm
		\item $F$ is cocartesian;
		
		\item let $\gamma \colon x \to y$ be a morphism in $\PiinftySigma(X,P)$ and let $f_\gamma \colon \cA_x \to \cA_y$ be a straightening for $p_\gamma \colon \cA_\gamma \to \Delta^1$.
		Then the  Beck-Chevalley transformation
		\[ f_{\gamma,!} j_x^\ast(F) \to  j_y^\ast(F) \]
		is an equivalence, where $j_x :\cI_x\to \cI$ and  $j_y :\cI_y\to \cI$ are the natural inclusions.
	\end{enumerate}
\end{prop}

\begin{prop}\label{(co)limit_and_cocart}
	Let $(X,P,\cI)$ be a Stokes stratified space.
	Then,
	\begin{enumerate}\itemsep=0.2cm
		\item the category $\Funcocart(\cI,\cE)$ is presentable stable and closed under colimits in $\Fun(\cI,\cE)$.

		\item If  $\cI$ has finite fibers, $\Funcocart(\cI,\cE)$ is closed under limits in $\Fun(\cI,\cE)$.
	\end{enumerate}
\end{prop}

\begin{proof}
	Item (1) is \cite[Corollaries \ref*{Abstract_Stokes-cor:cocartesianization} \& \ref*{Abstract_Stokes-cor:cocart_presentable_stable}]{Abstract_Derived_Stokes}. 
	Item (2) is \cite[\cref*{Abstract_Stokes-cocart_stability_limits}]{Abstract_Derived_Stokes}.
\end{proof}

\begin{defin}\label{defin_split}
	Let $F\colon\cI \to \cE$ be a functor.
	\begin{enumerate}\itemsep=0.2cm
		\item For $x \in X$, we say that $F$ is \emph{split at $x$} if $j_x^*(F)$ lies in the essential image of
		\[ i_{\cI_x,!} \colon \Fun(\cI^{\ens}_x, \cE)\to \Fun(\cI_x, \cE) \]
		where  $j_x :\cI_x\to \cI$ is the natural inclusion.
		\item We say that $F$ is \emph{punctually split} if it is split at every object $x \in \cX$.
		
		\item We say that $F$ is \emph{split} if it lies in the essential image of the induction functor
		\[ i_{\cI,!} \colon \Fun(\cI^{\ens}, \cE)\to \Fun(\cI, \cE) \ . \]
	\end{enumerate}
\end{defin}

\begin{rem}\label{rem:split_implies_punctually_split}
	Since induction commutes with pullback \cite[\cref*{Abstract_Stokes-cor:induction_specialization_Beck_Chevalley}]{Abstract_Derived_Stokes},  split functors are punctually split.
\end{rem}

\begin{eg}\label{eg:split_functor}
	For $a \in \cI$, write $\ev_a^{\cI} \colon \{a\} \hookrightarrow \cI$ for the canonical inclusion.
	Since $\ev_a^{\cI}$ factors through $i_\cI \colon \cI^{\ens} \to \cI$, we see that for every $E \in \cE$ the functor $\ev_{a,!}^{\cI}(E) \in \Fun(\cI, \cE)$ is split, and hence punctually split by \cref{rem:split_implies_punctually_split}.
\end{eg}

The following  provides a description of Stokes functors intrinsic to
$\Fun(\cI,\cE)$.

\begin{prop}\label{prop:Stokes_characterization}
Let  $F\colon\cI \to \cE$ be a functor.
Then the following are equivalent: 
\begin{enumerate}\itemsep=0.2cm
\item $F$ is a Stokes functor.

\item $F$ is cocartesian and  punctually split.
\end{enumerate}
\end{prop}
\begin{proof}
Immediate from \cref{prop:global_functoriality_IHES}.
\end{proof}

\begin{prop}[{\cite[\cref*{Abstract_Stokes-cor:stokes_functoriality}]{Abstract_Derived_Stokes}}]\label{cor:stokes_functoriality_IHES}
	Let $F \colon \cI \to \cE$  be a functor. 
	Then,
	\begin{enumerate}\itemsep=0.2cm
		\item \label{cor:stokes_functoriality:pullback} if $F \colon \cI \to \cE$ is a Stokes functor, the same goes for $u^\ast(F) \colon \cI_Y \to \cE$;
		
		\item \label{cor:stokes_functoriality:induction} if $G \colon \cJ_X \to \cE$ is a Stokes  functor,  the same goes for $v_!(G) \colon \cI \to \cE$.
	\end{enumerate}
	Thus, the functors
	\[ u^\ast \colon \Fun(\cI, \cE) \to  \Fun(\cI_Y, \cE) \qquad \text{and} \qquad v_! \colon \Fun(\cI_Y, \cE) \to  \Fun(\cJ, \cE) \]
	restrict to well-defined functors
	\[ u^\ast \colon \St_{\cI,\cE} \to  \St_{\cI_Y,\cE}  \qquad \text{and} \qquad v_! \colon \St_{\cI_Y,\cE} \to  \St_{\cJ,\cE} \ . \]
\end{prop}

\begin{rem}
\cref{cor:stokes_functoriality_IHES} holds similarly for cocartesian functors.
\end{rem}

\begin{rem}\label{induction_limits}
Since  $v_! \colon \Fun(\cI_Y, \cE) \to  \Fun(\cJ, \cE)$ is a left adjoint, it commutes with colimits.
If furthermore $\cI$ and $\cJ$ have finite fibers, $v_!$ commutes with limits \cite[\cref*{Abstract_Stokes-Stokes_p_commutation_limcolim}]{Abstract_Derived_Stokes}.
\end{rem}

Stokes functors behave well with respect to change of coefficients:

\begin{prop}[{\cite[\cref*{Abstract_Stokes-prop:change_of_coefficients_stokes}]{Abstract_Derived_Stokes}}]\label{prop:change_of_coefficients_stokes_IHES}
Let $f \colon \cE \to \cE'$ be a morphism in $\PrL$. 
Then, the induced functor $f \colon \Fun(\cI, \cE) \to \Fun(\cI, \cE')$
	induces a well defined functor
	\[ f \colon \St_{\cI,\cE} \to \St_{\cI,\cE'} \ . \]
\end{prop}

\begin{prop}[{\cite[\cref*{Abstract_Stokes-cocart_and_localization}]{Abstract_Derived_Stokes}}]\label{refinement_and_cocart}
Let $f\colon(Y,Q,\cJ)\to (X,P,\cI)$ be a cartesian refinement between exodromic stratified spaces.
Then the following holds:
\begin{enumerate}\itemsep=0.2cm
\item The induction  $f_!  \colon\Fun(\cJ,\cE) \to \Fun(\cI,\cE)$
preserves Stokes functors.

\item The adjunction $f_! \dashv f^*$ induces an equivalence of $\infty$-categories between 
$\St_{\cJ,\cE}$ and $\St_{\cI,\cE}$.
\end{enumerate}
\end{prop}

\subsection{Graduation}\label{graduction_section}

Let $(X,P)$ be a an exodromic stratified space.
Starting with a morphism $p \colon \cI \to \cJ$ in $\PosFib$ over $\PiinftySigma(X,P)$, consider the fiber product
	\[ \begin{tikzcd}
		\cI_p \arrow{r} \arrow{d}{\pi} & \cI \arrow{d}{p} \\
		\cJ^{\ens} \arrow{r}{i_\cJ} & \cJ \ .
	\end{tikzcd} \]
	When $X$ is  a point, we can identify $\cI_p$ with the poset $(\cI, \le_p)$, where
	\[ a \leq_p a' \stackrel{\text{def.}}{\iff} p(a) = p(a') \text{ and } a \leq a' \ . \]
	In other words, if $p(a) \ne p(a')$, then $a$ and $a'$ are incomparable for $\le_p$.\\ \indent

	The source  and the formation of the identity morphism induce morphisms of cocartesian fibration in posets $s \colon \cJ^{\Delta^1} \to \cJ $ and $\id \colon  \cJ \to  \cJ^{\Delta^1}$.
Then, objects of 
$$
\cI_{\leq}  \coloneqq \cJ^{\Delta^1} \times_{\cJ} \cI
$$	
 are triples $(x,a,b)$ where $a\in \cI_x$, $b\in \cJ_x$ and where $p(a)\leq b$ in $\cJ_x$.
	We also consider the full subcategory $i_{<} \colon \cI_< \hookrightarrow \cI_\leq$ spanned by objects $(x,a,b)$ with $p(a) < b$.
	In general $\cI_<\to \PiinftySigma(X,P)$ is not a cocartesian fibration.
	We thus introduce the following
	
\begin{defin}\label{def:locally_constant}
	We say that a functor $p \colon \cA \to \cX$ of $\infty$-categories is a  \emph{locally constant fibration} if it is a cocartesian fibration and its straightening $\Upsilon \colon \cX \to \CAT_\infty$ sends every arrows of $\cX$ to equivalences in $\CAT_\infty$.
\end{defin}	
	
\begin{defin}
We say that $p \colon \cI \to \cJ$ is a \textit{graduation morphism} if $\cJ^{\ens}\to \PiinftySigma(X,P)$ is locally constant.
\end{defin}
If $p \colon \cI \to \cJ$ is a graduation morphism, which we suppose from this point on,  one checks that $\cI_<\to \PiinftySigma(X,P)$ is  a cocartesian fibration.
Consider the following diagram with pull-back squares:
\begin{equation}\label{eq:graduation_diagram}
	\begin{tikzcd}
		& & \cI_{<} \arrow{d}{i_{<}} & \\
		\cI_p \arrow{d} \arrow{rr}{i_p} & & \cI_{\leq} \arrow{d} \arrow{r}{{\sigma}} & \cI  \arrow{d}{p} \\
		\cJ^{\ens}\arrow{r}{i_{\cJ}} & \cJ \arrow{r}{\id} & \cJ^{\Delta^1} \arrow{r}{s} & \cJ   \ .
	\end{tikzcd} 
\end{equation}
Write
\[ \varepsilon_< \colon i_{<!} i_<^\ast \to  \id_{\Fun(\cI_{\leq }, \cE)} \]
for the counit of the adjunction $i_{<!} \colon \Fun(\cI_<, \cE) \leftrightarrows \Fun(\cI_\leq, \cE) \colon i_<^\ast$.

\begin{defin}\label{defin_Gr}
	The \emph{graduation functor relative to $p \colon \cI \to \cJ$} 
	\[ \Gr_p \colon \Fun(\cI, \cE) \to  \Fun(\cI_{p}, \cE) \]
	is the cofiber
	\[ \Gr_p \coloneqq \cofib\big( i_p^\ast \varepsilon_< \sigma^\ast \colon i_p^\ast \circ i_{<!} \circ i_<^\ast \circ \sigma^\ast \to i_p^{\ast} \circ \sigma^\ast \big) \ . \]
\end{defin}

\begin{notation}

	When $p = \id$, we note $\Gr$ for $\Gr_{\id}$. 
\end{notation}

\begin{rem}
The formation of $\Gr_p$ commutes with pullback (see \cite[\cref*{Abstract_Stokes-lem:Gr_restriction}]{Abstract_Derived_Stokes}).
\end{rem}

\begin{rem}\label{cor:section_Gr_commutes_with_colimits}
Since $\Gr_p$ is defined via left adjoints, it commutes with colimits.
If furthermore the fibers of $\cI$ are finite posets, $\Gr_p$  commutes with limits (see \cite[\cref*{Abstract_Stokes-cor:section_Gr_commutes_with_limits}]{Abstract_Derived_Stokes}).
\end{rem}

\begin{eg}[{\cite[\cref*{Abstract_Stokes-ex_Gr}]{Abstract_Derived_Stokes}}]
	Let $p \colon \cI \to \cJ$ be a morphism of posets.
	Let $V \colon \cI^{\ens}\to \cE$ be a functor and put $F \coloneqq i_{\cI !}(V)$.
	For $a \in \cI_p$,  there is a canonical equivalence
	\[ (\Gr_p(F))_a \simeq \bigoplus_{\substack{a'\leq a \\ p(a')=p(a)}} V_{a'}  \ . \]
\end{eg}

\begin{prop}[{\cite[\cref*{Abstract_Stokes-Gr_of_Stokes}]{Abstract_Derived_Stokes}}]\label{Gr_of_Stokes_IHES}
	Let $(X,P,\cI)$ be a Stokes stratified space and let $p\colon \cI \to \cJ$ be a graduation morphism of Stokes fibrations over $(X,P)$.
	Then, the graduation functor relative to $p$
	\[ \Gr_p \colon \Fun(\cI, \cE) \to  \Fun(\cI_p, \cE) \]
	preserves the category of Stokes functors.
	In other words, it restricts to a functor
	\[ \Gr_p \colon\St_{\cI,\cE}  \to  \St_{\cI_p,\cE}   . \]
\end{prop}

\begin{cor}\label{cocartesian_splitting}
	Let $(X,P,\cI)$ be a Stokes stratified space such that $\cI^{\ens}\to \PiinftySigma(X,P)$ is locally constant.
	Then the commutative square
	\[ \begin{tikzcd}
		\St_{\cI^{\ens},\cE}  \arrow[hook]{d} \arrow{r}{i_{\cI,!}} & \St_{\cI,\cE}  \arrow[hook]{d} \\
		\Fun(\cI^{\ens},\cE) \arrow{r}{i_{\cI,!}}  & \Fun(\cI,\cE)
	\end{tikzcd} \]
is a pullback.
\end{cor}

\begin{proof}
Let $F \colon \cI \to \cE$ be a Stokes functor and let $V \colon  \cI^{\ens}\to \cE$ such that $F\simeq i_{\cI , !}(V)$.
Then,  computation gives
\[
\Gr(F)\simeq \Gr(i_{\cI , !}(V)) \simeq V \ .
\]
By \cref{Gr_of_Stokes_IHES},  the functor $V \colon  \cI^{\ens}\to \cE$ is a Stokes functor.
\end{proof}

Our main use of graduation will be through the following dévissage theorem

\begin{thm}[{\cite[\cref*{Abstract_Stokes-prop:Level_induction}]{Abstract_Derived_Stokes}}]\label{level_dévissage}
Let $(X,P)$ be an exodromic stratified space and let $p\colon \cI \to \cJ$ be a level graduation morphism of Stokes fibrations over $(X,P)$ in the sense of \cref{level_intro_def}.
Then, the square
\begin{equation}\label{eq_level_dévissage}
	 \begin{tikzcd}
		\St_{\cI, \cE} \arrow{r}{p_!} \arrow{d}{\Gr_p}& \St_{\cJ,\cE} \arrow{d}{\Gr} \\
		\St_{\cI_p, \cE} \arrow{r}{\pi_!} & \St_{\cJ^{\ens}, \cE}
	\end{tikzcd} 
	\end{equation}
is a pullback in $ \CAT_{\infty} $.
\end{thm}

In some favourable situations,  the pullback in \cref{level_dévissage} occur in $\PrLR$. 
To this end, we introduce the following definitions.

\begin{defin}\label{bireflexive}
Let $(X,P)$ be an exodromic stratified space.
We say that a cocartesian fibration in posets $\cI \to \Pi_\infty(X,P) $  is \textit{bireflexive} if for every presentable stable $\infty$-category $\cE$, the full subcategory $\St_{\cI,\cE}\subset \Fun(\cI,\cE)$ is closed under limits and colimits.
\end{defin}

\begin{defin}\label{defin_PrLR}
	Define $\PrLR$ as the (non full) subcategory of $\PrL$ whose objects are presentable 
	$\infty$-categories and morphisms are functors that are both left and right adjoints.
\end{defin}

\begin{lem}\label{Peter_lemma_1}
	Let $ A $ be a small $\infty$-category and let $ \cC_{\bullet} \colon A\to \PrLR$ be a diagram of $\infty$-categories.
	Then, the limits of $ \cC_{\bullet} $ when computed in $ \PrR $, $ \PrL $, or $ \CAT_{\infty} $ all agree.
\end{lem}

\begin{lem}\label{level_pullback_in_PrLR}
Let $(X,P)$ be an exodromic stratified space and let  $p\colon \cI \to \cJ$ be a level graduation morphism of Stokes fibrations in finite posets over $(X,P)$.
	Assume that all the cocartesian fibrations occurring in \cref{level_dévissage} are bireflexive.
	 Then the square \eqref{eq_level_dévissage} is a pullback in $\PrLR$.
\end{lem}

\begin{proof}
	The $\infty$-categorical reflection theorem of \cite[Theorem 1.1]{Ragimov_Schlank_Reflection} implies that all the $\infty$-categories of Stokes functors appearing in \eqref{eq_level_dévissage} are presentable.
	Since the fibers of  $\cI$ and $\cJ$ are finite posets, 
we know from \cref{induction_limits} and \cref{cor:section_Gr_commutes_with_colimits} that $p_!$ and $\Gr_p$ commute with limits and colimits.
Then \cref{level_pullback_in_PrLR} follows from \cref{Peter_lemma_1}.
\end{proof}



\begin{prop}[{\cite[\cref*{Abstract_Stokes-Stokes_detection}]{Abstract_Derived_Stokes}}]\label{Stokes_detection_IHES}
Let $(X,P)$ be an exodromic stratified space and let  $p\colon \cI \to \cJ$ be a level graduation morphism of Stokes fibrations in finite posets over $(X,P)$.
Let $F \colon \cI \to \cE$ be a functor.
Then the following are equivalent: 
\begin{enumerate}\itemsep=0.2cm
\item $F$ is a Stokes functor.
\item $\Gr_p(F) \colon \cI_p \to \cE$ and $p_!(F) \colon \cJ \to \cE$ are Stokes functors.
\end{enumerate}
\end{prop}

\begin{prop}[{\cite[\cref*{Abstract_Stokes-image_fully_faithful_induction_Stokes}]{Abstract_Derived_Stokes}}]\label{image_fully_faithful_induction_Stokes_IHES}
	 Let $(X,P)$ be an exodromic stratified space.
	Let $i \colon \cI \hookrightarrow  \cJ$ be a fully faithful graduation morphism of Stokes fibrations in finite posets over $(X,P)$.
	Let $F\colon\cJ \to  \cE$ be a Stokes functor. 
	Then, the following are equivalent :
	\begin{enumerate}\itemsep=0.2cm
		\item $F$ lies in the essential image of $i_! \colon \St_{\cI,\cE} \to \St_{\cJ,\cE}$.
		\item  $(\Gr F)_a\simeq 0$  for every  $a\in \cJ^{\ens}$ not  in the essential image of $i^{\ens} \colon \cI^{\ens} \to \cJ^{\ens}$.
	\end{enumerate}
\end{prop}

\subsection{Splitting criterion}
The goal of what follows is to state a technical splitting criterion used in the proof of the elementarity of polyhedral Stokes stratified spaces (see \cref{polyhedral_essentially_surjective}).

\begin{construction}  \label{constr:F_minus_I}   
     Let $(X,P,\cJ)$ be a Stokes stratified space in finite posets such that $\cJ^{\ens}\to \PiinftySigma(X,P)$ is locally constant.
     Let $i \colon \cI\hookrightarrow  \cJ$ and $k \colon \cK\hookrightarrow  \cJ$ be fully faithful functors in $\PosFib$ over $\PiinftySigma(X,P)$ such that $\cJ^{\ens}= \cI^{\ens} \sqcup \cK^{\ens}$.
      Let $F\colon \cJ\to \cE$ be a functor.
     Suppose that the canonical morphism 
     $$
      i^{\ens, \ast} i^{*}_{\cJ }(F)  \to  i^{\ens \ast}\Gr(F)
      $$ 
      admits a section
     $$
	\sigma \colon i^{\ens \ast}\Gr (F)    \to  i^{\ens, \ast} i^{*}_{\cJ }(F)   \ .
$$
      By adjunction, $\sigma$ yields a morphism
\[
\tau \colon i_{\cJ !} i_{!}^{\ens} i^{\ens \ast} \Gr(F) \to  F
\]
in $\Fun(\cJ, \cE)$.
      We put
$$
\begin{tikzcd}
F^{\backslash \cI}\coloneqq \cofib(\tau\colon i_{\cJ !} i_{!}^{\ens}  i^{\ens \ast} \Gr(F)   \to F )  \ .
\end{tikzcd}
$$
\begin{lem}[{\cite[\cref*{Abstract_Stokes-morphism_comes_from_graded}]{Abstract_Derived_Stokes}}]\label{morphism_comes_from_graded_IHES}
If $F\colon\cI \to \cE$ is a Stokes functor, so is $F^{\backslash \cI}$.
\end{lem}

Let $l \colon \cL\hookrightarrow  \cK$ and $m \colon \cM\hookrightarrow  \cK$ be fully faithful functors in $\PosFib$ over $\PiinftySigma(X,P)$ such that $\cK^{\ens}= \cL^{\ens} \sqcup \cM^{\ens}$.
      Suppose that the canonical morphism 
      $$ 
      l^{\ens \ast} i^{*}_{\cJ }(F)  \to  l^{\ens \ast}\Gr(F)
      $$ 
      admits a section
$$
	\lambda \colon l^{\ens \ast}\Gr (F)    \to  l^{\ens \ast} i^{*}_{\cJ }(F)   
$$
and define $F^{\backslash \cL}$ similarly.
\end{construction}

\begin{lem}[{\cite[\cref*{Abstract_Stokes-fiber_product_is_split}]{Abstract_Derived_Stokes}}]\label{fiber_product_is_split_IHES}
	In the setting of \cref{constr:F_minus_I}, the following are equivalent:
	\begin{enumerate}\itemsep=0.2cm
	\item the functor $F$ split;
	\item the functors $F^{\backslash \cI}$ and $F^{\backslash \cL}$ split. 
	\end{enumerate}
\end{lem}

%
%

\section{Stokes analytic stratified spaces }

We start deepening our analysis of the category $\StStrat$ of Stokes (analytic) stratified spaces and introducing the key notion of elementary morphisms.

\subsection{Functorialities of Stokes stratified spaces}

Recall from  \cref{rem:cocartesian_fibration_vs_constructible_sheaf} that a Stokes stratified space is a triple $(X, P, \cI)$ where $(X,P)$ is an exodromic stratified space and $\cI \to \Pi_\infty(X,P)$ is a cocartesian fibration in posets.

\begin{defin}
	If $\cC \subset \Mor(\Stratc)$ is a class of morphisms, we say that a morphism $(X,P,\cI)\to (Y,Q,\cJ)$ in $\StStrat$ lies in $\cC$ if the induced morphism of analytic stratified spaces $(X,P)\to (Y,Q)$ lies in $\cC$.
\end{defin}

\begin{eg}
	The most relevant classes for our purposes are those of proper morphisms, refinements and Galois covers.
\end{eg}

Recall from Definitions \ref{def:cocartesian_filtered_functor} and \ref{def:Stokes_sheaf} that to every Stokes stratified space $(X, P, \cI)$ we can attach two $P$-hyperconstructible hypersheaves with values in $\CAT_\infty$: 
\[ \frakFil_{\cI, \cE} , \frakStokes_{\cI, \cE} \in \ConsPhyp(X;\CAT_\infty) \ . \]

\begin{construction}\label{construction:Stokes_sheaves_functoriality}
	Let $f \colon (Y, Q, \cJ) \to (X,P,\cI)$ be a morphism in $\StStrat$.
	Recall that $f$ amounts to the datum of a morphism of stratified spaces $f \colon (Y,Q) \to (X,P)$ and a commutative diagram
	\[ \begin{tikzcd}
		\cI \arrow{d} & \cI_Y \arrow{d} \arrow{r}{v_f} \arrow{l}[swap]{u_f} & \cJ \arrow{dl} \\
		\Pi_\infty(X,P) & \Pi_\infty(Y,Q) \arrow{l}[swap]{\Pi_\infty(f)} 
	\end{tikzcd} \]
	where the square is cartesian.
	Applying the exponential construction yields the following commutative diagram
	\begin{equation}\label{eq:Stokes_sheaves_functoriality}
		\begin{tikzcd}
			\exp_\cE^{\PS}( \cI / \Pi_\infty(X,P) ) \arrow[hook]{d} & \exp_\cE^{\PS}( \cI_Y / \Pi_\infty(Y,Q) ) \arrow{r}{\cE^{v_f}_!} \arrow{l}[swap]{\cE^{u_f}} \arrow[hook]{d} & \exp_\cE^{\PS}( \cJ / \Pi_\infty(Y, Q) ) \arrow[hook]{d} \\
			\exp_\cE( \cI / \Pi_\infty(X,P) ) \arrow{d} & \exp_\cE( \cI_Y / \Pi_\infty(Y,Q) ) \arrow{d} \arrow{l}[swap]{\cE^{u_f}} \arrow{r}{\cE^{v_f}_!} & \exp_\cE( \cJ / \Pi_\infty(Y,Q) ) \arrow{dl} \\
			\Pi_\infty(X,P) & \Pi_\infty(Y,Q) \arrow{l}[swap]{\Pi_\infty(f)}
		\end{tikzcd}
	\end{equation}
	The functoriality of the exodromy equivalence with coefficients in $\PrL$ recalled in \cref{exodromy_PrL}  shows that the middle row zig-zag induces transformations
	\[ \mathfrak u_f^{\ast} \colon  \frakFil_{\cI_Y, \cE}\to   f^{\ast, \hyp}( \frakFil_{\cI, \cE} )  \qquad \text{and} \qquad \mathfrak{v}_{f,!} \colon \frakFil_{\cI_Y, \cE} \to  \frakFil_{\cJ, \cE} \]
	in $\ConsQhyp(Y;\PrL)$.
	Similarly the functoriality of the exodromy equivalence  recalled in \cref{exodromy_functorialities} shows that the top row zig-zag induces transformations
	\[ \mathfrak u_f^{\ast} \colon \frakStokes_{\cI_Y, \cE} \to    f^{\ast,\hyp}( \frakStokes_{\cI, \cE} )\qquad \text{and} \qquad \mathfrak{v}_{f,!} \colon \frakStokes_{\cI_Y, \cE} \to  \frakStokes_{\cJ, \cE} \]
	in $\ConsQhyp(Y;\CAT_\infty)$.
	Note that the commutativity of \eqref{eq:Stokes_sheaves_functoriality} shows that these natural transformations are compatible with the inclusion of $\frakStokes_{(-),\cE}$ into $\frakFil_{(-),\cE}$.
\end{construction}

\begin{prop}\label{pullback_Stokes_sheaf}
	Let $f \colon (Y,Q,\cJ)\to (X,P,\cI)$ be a morphism in $\StStrat$ (see \cref{rem:cartesian_morphism_in_StStrat}).
	Then the canonical morphisms
	\[ \mathfrak u_f^{\ast} \colon \frakFil_{\cI_Y,\cE} \to   f^{\ast,\hyp}( \frakFil_{\cI,\cE} )\qquad \text{and} \qquad \mathfrak u_f^{\ast} \colon  \frakStokes_{\cI_Y,\cE}  \to f^{\ast,\hyp}(\frakStokes_{\cI,\cE}) \]
	are equivalences.
	If in addition $f$ is cartesian, then the morphisms
	\[ \mathfrak v_{f,!} \colon \frakFil_{\cI_Y, \cE} \to  \frakFil_{\cJ,\cE} \qquad \text{and} \qquad \mathfrak v_{f,!} \colon \frakStokes_{\cI_Y, \cE} \to  \frakStokes_{\cJ, \cE} \]
	are equivalences.
\end{prop}

\begin{proof}
	Since the exodromy equivalence with coefficients in $\PrL$ and in $\CAT_\infty$ is functorial by \cref {exodromy_functorialities} and \cref{exodromy_PrL}, the first statement follows directly from the fact that the left squares in \eqref{eq:Stokes_sheaves_functoriality} are pullback, see \cref{construction:exponential}.
	The second statement follows from the functoriality of $\exp_\cE$, since when $f$ is cartesian $v_f \colon \cI_Y \to \cJ$ is itself an equivalence.
\end{proof}

\begin{cor}\label{stalk_pullback_Stokes_sheaf}
	Let $(X,P,\cI)\in \StStrat$.
	For every $x\in X$, the stalk of $\frakStokes_{\cI,\cE}$ at $x$ is canonically identified with $\St_{\cI_x,\cE}$, i.e.\ with the essential image of $i_{\cI_x,!} \colon \Fun(\cI_x^{\mathrm{set}}, \cE) \to \Fun(\cI_x,\cE)$.
\end{cor}

\begin{cor}\label{Stokes_sheaf_trivial_fibration}
	Let $(X,P)$ be an exodromic stratified space, considered as a Stokes stratified space $(X,P,\Pi_\infty(X,P))$.
	Then, $\frakStokes_{\Pi_{\infty}(X,P),\cE}$ is canonically equivalent to $\frak{Loc}_{X,\cE}$ (see \cref{def_loc_presheaf}).
\end{cor}

\begin{proof}
	Observe that $(X,P,\Pi_{\infty}(X,P))\to (X,\ast,\Pi_{\infty}(X))$ is a cartesian refinement in $\StStrat$.
	From \cref{pullback_Stokes_sheaf}, we deduce that $\frakStokes_{ \Pi_{\infty}(X,P),\cE}$ is canonically equivalent to $\frakStokes_{ \Pi_{\infty}(X),\cE}$.
	The punctually split condition being empty in that case,   $\frakStokes_{ \Pi_{\infty}(X),\cE}$ is canonically equivalent to $\mathfrak{Fil}_{ \Pi_{\infty}(X),\cE}$.
	Then, the conclusion follows from \cref{Filco_trivial_stratification}.
\end{proof}

\begin{cor}\label{Loc_constant_Stokes_sheaf}
	Let $(X,P,\cI)$ be a Stokes stratified space such that $\cI\to \Pi_{\infty}(X,P)$ is locally constant.
	Then, $\frakStokes_{\cI,\cE}$ is locally hyperconstant on $X$.
\end{cor}

\begin{proof}
	By definition, the straightening of $\cI \to \Pi_{\infty}(X,P)$ sends every exit path to an isomorphism of posets.
	From \cref{refinement_localization}, we deduce the existence of a cartesian refinement $(X,P,\cI)\to (X,\ast ,\cJ)$.
	Hence, \cref{pullback_Stokes_sheaf} ensures that $\frakStokes_{\cI,\cE}$ is canonically equivalent to $\frakStokes_{\cJ,\cE}$.
	By construction, $\frakStokes_{\cJ,\cE}$ lies in $\Loc(X;\PrL)$ so the conclusion follows.
\end{proof}

By design,  $\frakStokes_{\cI,\cE}$ satisfies hyperdescent.
The next proposition shows that actually more is true.
Before stating it, let us introduce the following

\begin{defin}\label{def:stably_universal}
Let $(X,P)$ be an exodromic stratified space.
We say that a cocartesian fibration in posets $\cI \to \Pi_\infty(X,P) $  is \emph{universal} if it is bireflexive and the canonical comparison map (see \cite[\cref*{Abstract_Stokes-construction:Stokes_tensor_comparison}]{Abstract_Derived_Stokes})
$$
\St_{\cI,\cE} \otimes \cE' \to \St_{\cI,\cE\otimes \cE'}
$$
is an equivalence for every presentable stable $\infty$-categories $ \cE,\cE'$.
\end{defin}

Bireflexity and universality are local properties for the étale topology:

\begin{prop}\label{étale_hyperdescent}
Let $(X,P,\cI)\in \StStrat$.
Then, the following holds:
\begin{enumerate}\itemsep=0.2cm
\item for every étale hypercover $U_\bullet$ of $X$ such that $(U_n,P)$ is exodromic for every $[n]\in \mathbf \Delta_s$, the canonical functor
\[
\St_{\cI,\cE}   \to \lim_{[n]\in \mathbf \Delta_s\op} \St_{\cI_{U_n},\cE} 
\]
is an equivalence.
\item If furthermore  $(U_n,P,\cI_{U_n})$ is  bireflexive for every $[n]\in \mathbf \Delta_s$,  then so is $(X,P,\cI)$ and the above limit can be computed in $ \PrLR$.
\item If furthermore $(U_n,P,\cI_{U_n})$ is  universal for every $[n]\in \mathbf \Delta_s$,  then so is $(X,P,\cI)$.
\end{enumerate}

\end{prop}

\begin{proof}
	By the étale version of Van Kampen \cite[Corollary 3.4.5]{Beyond_conicality}, we know that 
\[
\colim \Pi_{\infty}(U_n,P) \to \Pi_{\infty}(X,P)  
\]
	is an equivalence.	
	Then, (1) follows from the van Kampen theorem for Stokes functors \cite[\cref*{Abstract_Stokes-prop:Van_Kampen_Stokes}]{Abstract_Derived_Stokes}.
	Item (2) is an immediate consequence of \cite[\cref*{Abstract_Stokes-cor:descending_geometricity_of_Stokes_via_colimits}]{Abstract_Derived_Stokes}.
	 Item (3) follows from \cite[\cref*{Abstract_Stokes-prop:descent_of_universality}]{Abstract_Derived_Stokes}.
\end{proof}




\subsection{Hyperconstructible hypersheaves and tensor product}\label{cocar_functor_tensor_product}

	Let $(X,P)\in \Stratc$ be an exodromic stratified space.
\personal{Since $\PrL$ is not presentable itself, it's better to avoid mentioning the sheafification functor 
\[
(-) \otimes^{\hyp} \cE  \colon \Shhyp(X; \PrL) \to \Shhyp(X; \PrL)
\]
for which we have no reference.}
	Composition with the colimit-preserving functor
\[ (-) \otimes \cE \colon \PrL \to  \PrL \]
	induces a colimit preserving functor
\[ \Fun(\Pi_\infty(X,P), \PrL) \to  \Fun(\Pi_\infty(X,P), \PrL) \ . \]
	The exodromy equivalence with coefficients in $\PrL$ from \cref{exodromy_PrL} allows therefore to define a functor
\[ (-) \otimes^{\hyp} \cE  \colon \ConsPhyp(X; \PrL) \to \ConsPhyp(X; \PrL) \]
making the diagram
\begin{equation}\label{tensor_exodromy}
\begin{tikzcd}
		\ConsPhyp(X;\PrL) \arrow{r}{\sim} \arrow{d}{(-)\otimes^{\hyp} \cE} & \Fun(\Pi_{\infty}(X,P),\PrL) \arrow{d}{(-)\otimes \cE} \\
		\ConsPhyp(X;\PrL) \arrow{r}{\sim} &  \Fun(\Pi_\infty(X,P),\PrL)
	\end{tikzcd} 
\end{equation}	
commutative.

\begin{notation}
	There is a natural forgetful functor 
$$
	\ConsPhyp(X;\PrL) \to \PSh(X;\PrL)
$$
and $(-) \otimes \cE$ induces a well defined functor
	\[ (-) \otimes \cE \colon \PSh(X; \PrL) \to  \PSh(X;\PrL) \ . \]
	In other words, given $\cF \in \ConsPhyp(X;\PrL)$, $\cF \otimes \cE$ is the presheaf sending an open $U$ of $X$ to $\cF(U) \otimes \cE$.
\end{notation}

\begin{construction}
	Let $\cF \in \ConsPhyp(X;\PrL)$.
	Unraveling the definitions, we see that for every point $x \in X$, there is a natural equivalence
	\[ (\cF \otimes^{\hyp} \cE)_x \simeq \cF_x \otimes \cE \in \PrL \ . \]
	Fix an open $U$ in $X$.
	Then we have a canonical identification
	\[ (\cF \otimes^{\hyp} \cE)(U) \simeq \lim_{x \in \Pi_\infty(U,P)} \cF_x \otimes \cE \ , \]
	and in particular we find a natural comparison map
	\[ \cF(U) \otimes \cE \to  (\cF \otimes^{\hyp} \cE)(U) \ , \]
	which is a particular case of the Beck-Chevalley transformation considered in 
	\cite[\cref*{Abstract_Stokes-cocart_tensor}]{Abstract_Derived_Stokes}.
	In other words, we obtain a natural transformation
	\begin{equation}\label{eq:tensor_product_categorical_sheaves}
		\cF \otimes \cE \to  \cF \otimes^{\hyp} \cE \ .
	\end{equation}
\end{construction}

\begin{notation}\label{defin_const_PrLR}
	We denote by 
$$
\ConsPhyp(X;\PrLR) \subset \ConsPhyp(X;\PrL)
$$ 
the full-subcategory  corresponding to objects in $\Fun(\Pi_{\infty}(X,P),\PrLR)$ through the exodromy equivalence (\ref{tensor_exodromy}).
\end{notation}

Let us recall the following lemma from \cite[2.7.9]{Beyond_conicality}.

\begin{lem}\label{Peter_lemma}
	Let $ A $ be a small $\infty$-category and let $ \cC_{\bullet} \colon A\to \PrLR$ be a diagram of $\infty$-categories.
		Then, for any presentable $\infty$-category $ \cE $, the natural morphism
		\begin{equation*}
			\lim_{\alpha \in A} \cE \otimes \cC_{\alpha}\to  \cE \otimes\lim_{\alpha \in A} \cC_{\alpha}
		\end{equation*}
		in $ \PrL $ is an equivalence.
		(Here, both limits are computed in $ \PrL $).
\end{lem}

\begin{lem}\label{naive_tensor_is_hypersheaf}
	Let $\cF \in \ConsPhyp(X;\PrLR)$.
	Then the comparison map \eqref{eq:tensor_product_categorical_sheaves} is an equivalence, and in particular the presheaf $\cF \otimes \cE$ is a hypersheaf.
\end{lem}
	
\begin{proof}
	It is enough to show that for every open subset $U$ of $X$, the canonical map
	\[ \big( \lim_{x \in \Pi_{\infty}(U,P)} \cF_x \big)\otimes \cE  \to  \lim_{x \in \Pi_{\infty}(U,P)} \cF_x \otimes \cE \]
	is an equivalence, which follows  from \cref{Peter_lemma}.
\end{proof}

%
%

	We conclude by recording the following handy sufficient condition ensuring that a categorical sheaf $\cF \in \ConsPhyp(X;\PrL)$ belongs to $\ConsPhyp(X;\PrLR)$.

\begin{lem}\label{PrLR_criterion}
	Let $(X,P)$ be a subanalytic stratified space.
	Let $\cF \in \ConsP(X; \PrL)$ such that for every open subsets $U\subset V$, the functor $\cF(V)\to \cF(U)$ is a left and right adjoint.
	Then, $\cF$ lies in $\ConsPhyp(X;\PrLR)$.
\end{lem}

\begin{proof}
	Let $F \colon \Pi_\infty(X,P) \to \PrL$ be the functor corresponding to $\cF$ via the exodromy equivalence \eqref{exodromy_equivalence}.
	Let $\gamma \colon x \to y$ be a morphism in $\Pi_{\infty}(X,P)$.
	By \cref{prop:locally_contractible_strata}, choose an open neighbourhood  $V$ of $x$ such that $x$ is initial in $\Pi_{\infty}(V,P)$.
	At the cost of writing $\gamma$ as the composition of a small exit-path followed by an equivalence, we can suppose that $\gamma$ lies in $V$.
	Let $U\subset V$ such that $y$ is initial in $\Pi_{\infty}(U,P)$.
	Then, the vertical arrows of the commutative diagram  
$$
\begin{tikzcd}
      \cF(V)  \arrow{d} \arrow{r} &  \cF(U) \arrow{d}   \\
	F(x) \arrow{r}{F(\gamma)} & F(y)
\end{tikzcd}
$$
are equivalences.
	\cref{PrLR_criterion} thus follows.
\end{proof}

\subsection{Elementarity}\label{subsec:elementarity}

	We now introduce a fundamental concept in the study of Stokes stratified spaces: the notion of elementarity and its variants.
\personal{reference to the coefficients added in the definition below}
\begin{defin}[Absolute elementarity]\label{def:absolute_elementarity}
	Let $(X,P,\cI)$ be a Stokes stratified space.
	We say that $(X,P,\cI)$ is:
	\begin{enumerate}\itemsep=0.2cm
		\item \emph{elementary} if  for every presentable stable $\infty$-category $\cE$, the functor
		\[ i_{\cI,!} \colon \St_{\cI^{\ens}, \cE} \to  \St_{\cI,\cE} \]
		is an equivalence;
		
		\item \emph{locally elementary} if $X$ admits a cover by open subsets $U$ such that $(U, P, \cI_U)$ is elementary.
	\end{enumerate}
\end{defin}

The following example shows that elementarity is a really strong condition.

\begin{eg}\label{eg:elementary_points_are_discrete}
	A poset $\cI$ seen as a Stokes stratified space $(\ast, \ast, \cI)$ is elementary if and only if $\cI$ is discrete.
	Indeed, if $\cI$ is discrete then $i_{\cI} \colon \cI^{\ens} \to \cI$ is an isomorphism and therefore the three arrows in the commutative triangle
	\[ \begin{tikzcd}
		\Fun(\cI^{\ens}, \cE) \arrow{r} \arrow{dr}[swap]{i_{\cI,!}} & \St_{\cI,\cE} \arrow[hook]{d} \\
		{} & \Fun(\cI, \cE)
	\end{tikzcd} \]
	are equivalences.
	Conversely, assume that $\cI$ is elementary.
	Then the top horizontal arrow is an equivalence, and therefore $i_{\cI,!}$ is forced to be fully faithful.
	Fix a non-zero object $E \ne 0$ in $\cE$ and assume by contradiction that there exists two elements $a, b \in \cI$ satisfying $a < b$.
	Then
	\[ \Map_{\Fun(\cI^{\ens}, \cE)}( \ev_{b,!}^{\cI^{\ens}}(E), \ev_{a,!}^{\cI^{\ens}}(E) ) \simeq 0 \ , \]
	while
	\begin{align*}
		\Map_{\Fun(\cI,\cE)}\big( i_{\cI,!} \ev_{b,!}^{\cI^{\ens}}(E), i_{\cI,!} \ev_{a,!}^{\cI^{\ens}}(E) \big) & \simeq \Map_{\Fun(\cI,\cE)}\big( \ev_{b,!}^{\cI}(E), \ev_{a,!}^{\cI}(E) \big) \\
		& \simeq \Map_{\Fun(\cI, \cE)}\big( E, \ev_b^{\cI,\ast} \ev_{a,!}^{\cI}(E) \big) \\
		& \simeq \Map_{\cE}( E, E ) \ne 0 \ ,
	\end{align*}
	which contradicts the full faithfulness of $i_{\cI,!}$.
\end{eg}

\begin{eg}
	We consider again the situation of \cref{eg:1_dimensional}.
	Then the analysis carried out there shows that $(S^1, P, \cI)$ is not elementary while $(W_1,P,\cI_{W_1})$ and $(W_{-1},P,\cI_{W_{-1}})$ are elementary.
	In other words, $(S^1, P, \cI)$ is locally elementary.
\end{eg}

\begin{eg}
	Take $X = (0,1)$ stratified in four points and take $\cI$ the constructible sheaf in posets depicted below:
	
	\medskip
	
	
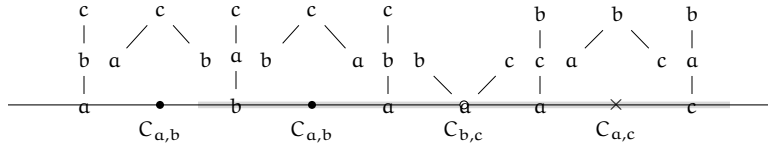
\begin{center}
		\begin{tikzpicture}
			\draw[line width = 2.5pt,gray,opacity=0.3] (1.5,0) -- (8.5,0) ;
			\draw[thin] (-1,0) -- (9,0) ;
			\fill (1,0) circle (1.5pt) ;
			\fill (3,0) circle (1.5pt) ;
			\draw[thin,fill=white] (5,0) circle (1.5pt) ;
			\node at (7,0) {\tiny $\times$} ;
			\node[above=2pt,anchor=north] at (0,1.5) {\tiny $\begin{tikzcd}[row sep=small] c \arrow[-]{d} \\ b \arrow[-]{d} \\ a \end{tikzcd}$} ;
			\node[below=2pt] at (1,0) {\tiny $C_{a,b}$} ;
			\node[above=2pt,anchor=north] at (1,1.5) {\tiny $\begin{tikzcd}[column sep=tiny,row sep=small] & c \\ a \arrow[-]{ur} & & b \arrow[-]{ul} \\ & \phantom{a} \end{tikzcd}$} ;
			\node[above=2pt,anchor=north] at (2,1.5) {\tiny $\begin{tikzcd}[row sep=small] c \arrow[-]{d} \\ a \arrow[-]{d} \\ b \end{tikzcd}$} ;
			\node[below=2pt] at (3,0) {\tiny $C_{a,b}$} ;
			\node[above=2pt,anchor=north] at (3,1.5) {\tiny $\begin{tikzcd}[column sep=tiny,row sep=small] & c \\ b \arrow[-]{ur} & & a \arrow[-]{ul} \\ & \phantom{a} \end{tikzcd}$} ;
			\node[above=2pt,anchor=north] at (4,1.5) {\tiny $\begin{tikzcd}[row sep=small] c \arrow[-]{d} \\ b \arrow[-]{d} \\ a \end{tikzcd}$} ;
			\node[below=2pt] at (5,0) {\tiny $C_{b,c}$} ;
			\node[above=2pt,anchor=north] at (5,1.5) {\tiny $\begin{tikzcd}[column sep=tiny,row sep=small] & \phantom{c} \\ b \arrow[-]{dr} & & c \arrow[-]{dl} \\ & a \end{tikzcd}$} ;
			\node[above=2pt,anchor=north] at (6,1.5) {\tiny $\begin{tikzcd}[row sep=small] b \arrow[-]{d} \\ c \arrow[-]{d} \\ a \end{tikzcd}$} ;
			\node[below=2pt] at (7,0) {\tiny $C_{a,c}$} ;
			\node[above=2pt,anchor=north] at (7,1.5) {\tiny $\begin{tikzcd}[column sep=tiny,row sep=small] & b \\ a \arrow[-]{ur} & & c \arrow[-]{ul} \\ & \phantom{a} \end{tikzcd}$} ;
			\node[above=2pt,anchor=north] at (8,1.5) {\tiny $\begin{tikzcd}[row sep=small] b \arrow[-]{d} \\ a \arrow[-]{d} \\ c \end{tikzcd}$} ;
		\end{tikzpicture}
	\end{center}
	
	\medskip

	\noindent Here we marked with $C_{\alpha,\beta}$ the Stokes locus for the pair $\{\alpha,\beta\}$.
	It follows from \cref{cor_induction_for_adm_Stokes_triple} that the shadowed interval is elementary, because it contains exactly one Stokes direction for every possible pair of elements of $\cI^{\ens} = \{a,b,c\}$.
	On the other hand, \cref{elementarity_constraint} shows that the leftmost $C_{a,b}$ cannot have an elementary open neighborhood.
	Hence, this Stokes stratified space  is not locally elementary.
\end{eg}

\begin{warning}\label{warning:local_elementarity_too_strong}
	Let $(X,P,\cI)$ be a Stokes stratified space.
	In general, the intersection of two elementary open subsets is no longer elementary: for instance, with the notations of \cref{eg:1_dimensional}, the intersection $W_1 \cap W_{-1}$ is no longer elementary.
	Also, \cref{eg:elementary_points_are_discrete} implies that any point $x \in X$ such that $\cI_x$ is not discrete does not have a fundamental system of elementary open neighborhoods.
	So even when $(X,P,\cI)$ is locally elementary, the elementary open subsets of $X$ do \emph{not} form a basis for the topology of $X$.
\end{warning}

Let us discuss two variations on \cref{def:absolute_elementarity}.
The first one concerns adapting the notion of elementarity to a family of Stokes stratified spaces:

\begin{defin}
	A morphism $(X,P) \to (Y,Q)$ in $\Stratc$ is said to be a \emph{family of exodromic stratified spaces} if for every $y \in Y$ the stratified space $(X_y,P)$ is exodromic.
\end{defin}

\begin{notation}
	Note that every exodromic stratified space $(Y,Q)$ gives rise to a Stokes stratified space $(Y,Q,\emptyset)$.
	We will commit a slight abuse of notation and write $(Y,Q)$ in place of $(Y,Q,\emptyset)$.
\end{notation}

\begin{defin}
	A \emph{family of Stokes stratified spaces} is a morphism
	\[ f \colon (X,P,\cI) \to  (Y,Q) \]
	in $\StStrat$ whose underlying morphism $f \colon (X,P) \to (Y,Q)$ is a family of exodromic stratified spaces.
	We denote the ($1$-)category of families of Stokes stratified spaces by $ \FStStrat \subset \StStrat^{[1]}$.
\end{defin}

\begin{eg}
	Let $f \colon (X,P) \to (Y,Q)$ be a morphism of subanalytic stratified spaces.
	Then for each $y \in Y$, the fiber $(X_y,P)$ is again a subanalytic stratified space, so \cref{eg:subanalytic_implies_combinatorial} guarantees that $(X_y,P)$ is again exodromic.
	Therefore $f$ is a family of exodromic stratified spaces.
	In particular, for any Stokes fibration $\cI$ on $(X,P)$, the resulting morphism $f \colon (X,P,\cI) \to (Y,Q)$ is a family of Stokes analytic stratified spaces.
\end{eg}

\begin{defin}[Relative elementarity]\label{defin_elementary}
	Let $f \colon (X,P,\cI) \to (Y,Q)$ be a family of Stokes stratified spaces.
	We say that \emph{$f$ is (locally) elementary at $y \in Y$} if $(X_y, P, \cI_y)$ is (locally) elementary.
	We say that \emph{$f$ is (locally) elementary} if it is (locally) elementary at $y$ for every $y \in Y$.
\end{defin}


	Before moving on to the second variation on elementarity, let us record a couple of important facts.
	The first is the following easy stability property:

\begin{lem}\label{elementarity_pullback}
	Consider a morphism of families of Stokes stratified spaces 
	\begin{equation}\label{eq:elementarity_pullback_0}
		\begin{tikzcd}
			(Y,Q,\cJ) \arrow{r}{g} \arrow{d}{f'} &(X,P,\cI)\arrow{d}{f}      \\
			(Y',Q')\arrow{r}{} & (X',P')
		\end{tikzcd} 
	\end{equation}
	with cartesian horizontal arrows.
	Consider the following conditions:
	\begin{enumerate}\itemsep=0.2cm
		\item The square of stratified spaces underlying \eqref{eq:elementarity_pullback_0} is a pullback.
		\item The horizontal arrows are refinements.
	\end{enumerate}
	Then, in both cases if $f$ is elementary the same goes for $f'$.
	In case (2), the converse holds.
\end{lem}

\begin{proof}
	In case (1), the fibers of $f'$ are fibers of $f$ so there is nothing to prove. 
	For (2), let $x\in X'$ and let $\cE$ be a presentable stable $\infty$-category.
	Then, restricting above $x$ yields a refinement of exodromic stratified spaces $g_x \colon  (Y_x,Q) \to (X_x,P)$.
	Thanks to \Cref{refinement_and_cocart} the horizontal arrows in the commutative square
	\[ \begin{tikzcd}
		\St_{\cJ_x^{\ens},\cE}  \arrow{r}{g^{\ens}_{x,!}} \arrow{d}{i_{\cJ_x,!}} &\St_{\cI_x^{\ens},\cE} \arrow{d}{i_{\cI_x,!}} \\
		\St_{\cJ_x,\cE} \arrow{r}{g_{x,!}}  & \St_{\cI_x,\cE} 
	\end{tikzcd} \]
	are equivalences, so the conclusion follows.
\end{proof}

The second property of relative elementarity is the following important local-to-global principle.
An important idea is that to establish absolute elementarity of some $(X,P,\cI)$, it is useful to fiber $(X,P,\cI)$ over a stratified space $(Y,Q)$, and then establish relative elementarity to apply the following:

\begin{prop}\label{elementary_criterion}
	Let $f \colon (X,P,\cI)\to (Y,Q)$ be an elementary family of Stokes stratified spaces.
	Assume that the underlying morphism $f \colon X \to Y$ is proper and that at least one of the following conditions hold:
	\begin{enumerate}\itemsep=0.2cm
		\item The induced morphism of $\infty$-topoi
		\[ f_\ast \colon \Shhyp(X) \to  \Shhyp(Y) \]
		is proper in the sense of \cite[Definition 7.3.1.4]{HTT}.
		
		\item $f \colon (X,P)\to (Y,Q)$ is a morphism of subanalytic stratified spaces.
	\end{enumerate}
	Then, $(X,P,\cI)$ is elementary.
\end{prop}

\begin{proof}
	Let $\cE$ be a presentable stable $\infty$-category.
	We have to show that
	\[ i_{\cI, !}  \colon  \St_{\cI^{\ens},\cE}  \to  \St_{\cI,\cE} \]
	is an equivalence.
	To do this, it is enough to show that the morphism
	\[ f_{\ast}(\mathscr{I}_{\cI, !} ) \colon  f_{\ast}(\frakStokes_{\cI^{\ens},\cE})\to   f_{\ast}(\frakStokes_{\cI,\cE}) \]
	in $\Shhyp(Y;\CAT_\infty)$ is an equivalence.
	This can be done at the level of stalks.
	Fix therefore $y \in Y$.
	For every $\cF \in \Shhyp(X;\CAT_\infty)$, we have a canonical comparison map
	\begin{equation}\label{eq:proper_base_change}
		y^{\ast,\hyp} f_\ast( \cF ) \to  \Gamma_{X_y,\ast}( j_y^{\ast,\hyp}(\cF) ) \ ,
	\end{equation}
	where $j_y \colon X_y \hookrightarrow X$ is the inclusion of the fiber.
	Notice that \cref{stalk_pullback_Stokes_sheaf} provides canonical  identifications
	\[ j_y^{\ast,\hyp}(\frakStokes_{\cI,\cE}) \simeq \frakStokes_{\cI_y, \cE} \qquad \text{and} \qquad j_y^{\ast,\hyp}(\frakStokes_{\cI^{\ens},\cE}) \simeq \frakStokes_{\cI_y^{\ens}, \cE} \ , \]
	so the result follows from the elementarity assumption as soon as we know that \eqref{eq:proper_base_change} is an equivalence for $\cF = \frakStokes_{\cI,\cE}$ and for $\cF = \frakStokes_{\cI^{\ens},\cE}$.
	In case (1), since $\CAT_\infty$ is compactly generated, \cite[Theorem 0.5]{Haine_Nonabelian_basechange} shows that \eqref{eq:proper_base_change} is an equivalence for every categorical hypersheaf $\cF \in \Shhyp(X;\CAT_\infty)$.
	In case (2), \cref{proper_analytic_direct_image} shows that \eqref{eq:proper_base_change} is an equivalence for any $\cF \in \ConsP(X;\CAT_\infty)$.
	So in both cases the conclusion follows.
\end{proof}

\begin{recollection}
	Let us recall some topological conditions that ensure that assumption (1) in \cref{elementary_criterion} are satisfied.
	Assume that:
	\begin{enumerate}\itemsep=0.2cm
		\item[(a)] $X$ is locally compact and Hausdorff and $f$ is proper;
		\item[(b)] both $X$ and $Y$ admit an open cover by subsets of finite covering dimensions (see \cite[Definition 7.2.3.1]{HTT}).
	\end{enumerate}
	Condition (a) ensures via \cite[Theorem 7.3.1.16]{HTT} that the geometric morphism
	\[ f_\ast \colon \Sh(X) \to  \Sh(Y) \]
	is proper.
	Condition (b) on the other hand guarantees that both $\Sh(X)$ and $\Sh(Y)$ are hypercomplete: combine \cite[Theorem 7.2.3.6, Corollary 7.2.1.12 and Remark 6.5.2.22]{HTT}.
	Finally, notice that any paracompact and finite dimensional space has finite covering dimension, see for instance \cite[Proposition 3.2.2]{Engelking_Dimension_Theory}.
\end{recollection}

We now introduce one final variation on the idea of elementarity in the analytic setting:

\begin{defin}[Absolute piecewise elementarity]
	Let $(X,P,\cI)$ be a Stokes analytic stratified space and let $x \in X$ be a point.
	We say that:
	\begin{enumerate}\itemsep=0.2cm
		\item $(X,P,\cI)$ is \emph{piecewise elementary at $x$} if there exists a closed subanalytic subset $Z$ containing $x$ such that $(Z,P,\cI_Z)$ is elementary;
		
		\item $(X,P,\cI)$ is \emph{strongly piecewise elementary at $x$} if there exists a closed subanalytic neighborhood $Z$ containing $x$ such that $(Z,P,\cI_Z)$ is elementary;
	\end{enumerate}
	We say that $(X,P,\cI)$ is \emph{(strongly) piecewise elementary} if it is (strongly) piecewise elementary at every point.
\end{defin}

\begin{rem}
	We will see in the next section that piecewise elementarity implies local elementarity: in other words, if one can find a closed subanalytic subset $Z$ containing $x$ such that $(Z,P,\cI_Z)$ is elementary, then $Z$ can be spread out to an elementary open neighborhood of $x$.
\end{rem}

Moving to the relative setting, we have:

\begin{defin}[Relative piecewise elementarity]\label{defin_piecewise_elementary}
	Let $f \colon (X,P, \cI) \to (Y,Q)$ be a family of Stokes analytic stratified spaces and let $x \in X$.
	We say that:
	\begin{enumerate}\itemsep=0.2cm
		\item $f$ is \emph{vertically piecewise elementary at $x$} if the fiber $(X_{f(x)}, P, \cI_{f(x)})$ is piecewise elementary at $x$;
		
		\item $f$ is \emph{piecewise elementary at $x$} if there exists a closed subanalytic subset $Z$ containing $x$ and such that $f |_Z \colon (Z,P,\cI_Z) \to (Y,Q)$ is an elementary family of Stokes analytic stratified spaces;
		
		\item $f$ is \emph{strongly piecewise elementary at $x$} if there exists a closed subanalytic neighborhood $Z$ of $x$ such that $f |_Z \colon (Z,P, \cI_Z) \to (Y,Q)$ is an elementary family of Stokes analytic stratified spaces.
	\end{enumerate}
	We say that $f$ is \emph{(vertically, strongly) piecewise elementary} if it is (vertically, strongly) piecewise elementary at every point.
\end{defin}

\begin{rem}\label{piecewise_implies_vertical_piecewise}
If $f \colon (X,P,\cI) \to (Y,Q)$ is strongly piecewise elementary at $x$, it is also piecewise elementary at $x$.
	If $f$ is piecewise elementary at $x$, then it is vertically piecewise elementary at $x$.
\end{rem}


We conclude with a couple of easy facts concerning piecewise elementarity:

\begin{lem}\label{piecewise_elementarity_pullback}
	Consider a morphism of families of Stokes analytic stratified spaces 
	\begin{equation}\label{eqelementarity_pullback}
		\begin{tikzcd}
			(Y,Q,\cJ) \arrow{r}{g} \arrow{d}{f'} &(X,P,\cI)\arrow{d}{f}      \\
			(Y',Q')\arrow{r}{} & (X',P')
		\end{tikzcd} 
	\end{equation}
	with cartesian horizontal arrows.
	Let $y\in Y$ and put $x=g(y)$. 
	Consider the following conditions:
	\begin{enumerate}\itemsep=0.2cm
		\item The square of stratified spaces induced by (\ref{eqelementarity_pullback}) is a pull-back.
		\item The horizontal arrows are refinements.
	\end{enumerate}
	Then, in either case  $f'$ is (strongly) piecewise elementary  at $y$ if $f$ is (strongly) piecewise elementary at $x$.
	In case (2), the converse holds.
\end{lem}

\begin{proof}
	Immediate from \cref{elementarity_pullback}.
\end{proof}


%

\subsection{Spreading out for Stokes analytic stratified spaces}

The goal is to prove a \emph{spreading out} property for closed subanalytic subsets of Stokes analytic stratified spaces that does not change the category of Stokes functors.
The proof combines all the known functoriality results concerning Stokes functors with the results of Thom, Mather, Goresky and Verdier on the local structure of analytic stratified spaces.
We will need the results from the theory of simplicial complexes (\cref{subsec:triangulations_hereditary_excellent}).
We will also need the following 

\begin{lem}[{\cite[\cref*{Abstract_Stokes-induction_poset_cocart} \& \cref*{Abstract_Stokes-lem:propagation_split_for_cocartesian}]{Abstract_Derived_Stokes}}]\label{induction_poset_cocart_IHES}
	Let $f \colon \mathrm{J}\to \mathrm{I}$ be a fully faithful functor between posets and consider a pullback square in $\Cat_{\infty}$
	\[ \begin{tikzcd}
		\cA \arrow{r}{u}\arrow{d}{q}&\cB \arrow{d}{p}\\
		\mathrm{J} \arrow{r}{f}&\mathrm{I} 
	\end{tikzcd} \]
	where in addition $p$ is a cocartesian fibration.
	Assume that for every object $i$ in $\mathrm{I}$, the subposet $\mathrm{J}_{/i}$ of $\mathrm{J}$ admits a final object.
	Then, the functor 
	$$
	u_! \colon \Fun(\cA,\cE)\to \Fun(\cB,\cE)
	$$ 
	preserves Stokes functors.
\end{lem}

\begin{thm}[Spreading out]\label{thm:spreading_out}
	Let $(X,P,\cI)$ be a Stokes analytic stratified space.
	Then any closed subanalytic subset $Z \subset X$ admits a fundamental system of open neighborhoods $i \colon Z \hookrightarrow U$ such that:
	\begin{enumerate}\itemsep=0.2cm
		\item $U$ final at $Z$ (see \cref{defin_final}).
		\item The induction $i_!   \colon  \Fun(\cI_Z,\cE) \to \Fun(\cI_U,\cE)$ preserves Stokes functors.
		\item The adjunction $i_! \dashv i^*$ induces an equivalence of $\infty$-categories between $\St_{\cI_Z,\cE}$ and $\St_{\cI_U,\cE}$.
		\item $(Z,P,\cI_Z)$ is elementary if and only if $(U,P,\cI_U)$ is elementary.
		\item If $Z$ is compact, the open set $U$ can be chosen to be subanalytic.
	\end{enumerate}
\end{thm}

\begin{proof}
    It is enough to find one open neighbourhood $U\subset X$ of $Z$ satisfying the conditions (1)-(5).
	Observe that the claim (4) follows from (3) and the commutativity of the following square 
	\[ \begin{tikzcd}
		\St_{\cI_Z^{\ens},\cE}\arrow{d}{i_{\cI_Z,!}}  \arrow{r}{i_!}  &  \St_{\cI_U^{\ens},\cE} \arrow{d}{i_{\cI_U,!}}  \\
		\St_{\cI_Z,\cE}     \arrow{r}{i_!}    &  \St_{\cI_U,\cE}  \ .
	\end{tikzcd} \]
	Since the pullback along a final functor induces an equivalence on the categories of Stokes functors \cite[\cref*{Abstract_Stokes-invariance_final_pull_back}]{Abstract_Derived_Stokes},  every open subset $U \subset X$ satisfying (1) and (2) automatically satisfies (3).
	We are thus left to find an open neighborhood $U$ of $Z$ satisfying (1) and (2).
	
	\medskip
	
	We first observe that to construct such an open neighborhood we can replace $(X,P,\cI)$ by any cartesian refinement.
	Indeed, let
	\[ r \colon (Y, Q, \cJ) \to  (X,P, \cI) \]
	be a cartesian refinement in $\Stratc$ and set 
	\[ T \coloneqq Z \times_X Y \ . \]
	Let $V$ be an open neighborhood of $T$ inside $Y$.
	Since $r \colon Y \to X$ is a homeomorphism, $U \coloneqq r(V)$ is an open neighborhood of $T$ inside $X$.
	We obtain the following commutative diagram in $\Stratc$:
	\[ \begin{tikzcd}
		(T,Q) \arrow{r}{r|_T} \arrow{d}{j} & (Z,P) \arrow{d}{i} \\
		(V,Q) \arrow{r}{r|_V} & (U,P) \ .
	\end{tikzcd} \]
	Passing to the stratified homotopy types, \cref{refinement_localization} shows that the horizontal maps becomes localizations, and hence final maps.
	Thus, \cite[Proposition 4.1.1.3-(2)]{HTT} implies that if $V$ is final at $T$ the $U$ is final at $Z$.
	Besides, \cref{refinement_and_cocart} shows that both
	\[ (r |_T)_! \colon \Fun(\cI_T, \cE) \to  \Fun(\cI_Z, \cE) \qquad \text{and} \qquad (r|_T)_! \colon \Fun(\cI_T^{\ens}, \cE) \to  \Fun(\cI_Z^{\ens}, \cE) \]
	preserves the full subcategories of Stokes functors and that the induced morphisms
	\[ (r |_T)_! \colon \St_{\cI_T, \cE} \to  \St_{\cI_Z, \cE} \qquad \text{and} \qquad (r_T)_! \colon \St_{\cI_T^{\ens}, \cE} \to  \St_{\cI_Z^{\ens}, \cE} \]
	are equivalences of $\infty$-categories, and similarly for $r|_V$ in place of $r|_T$.
	It follows that if $j_! \colon \Fun(\cI_T, \cE) \to \Fun(\cI_V, \cE)$ preserves Stokes functors, then so does $i_!$.
	
	\medskip
	
	Using \cite[Théorème 2.2]{Verdier1976} we can refine the stratification $(X,P)$ to a Whitney stratification $(X,Q)$ such that $Z$ is union of strata of $(X,Q)$.
	By \cite[Theorem §3]{Goresky_triang}, $(X,Q)$ admits a locally finite triangulation.
	Thus, using the notations from \cref{subsec:triangulations_hereditary_excellent}, we can replace $(X,Q)$ by the geometric realization $(|K|,F)$ of a simplicial complex $K = (V,F)$ and we can furthermore assume that $(Z,Q)$ corresponds to the geometric realization $(|S|, F_S)$ a simplicial subcomplex $S = (V_S, F_S)$ of $K$.
	At the cost of replacing $K$ by its barycentric subdivision, we can suppose that $S$ is full in $K$.
	Define
	\[ U_{S,K}\coloneqq \left\{ w \colon V \to [0,1] \mid \mathrm{supp}(w) \cap V_S \ne \emptyset \text{ and } \sum_{v \in V \smallsetminus V_S} w(v) <1 \right\} \ . \]
	We claim that $U_{S,K}$ satisfies conditions (1) and (2).
	Since $S$ is full in $K$, \cref{simplicial_complex_full_excellent} shows that $U_{S,K}$ is final at $|S|$.
	Concerning (2), observe that via the equivalence
	\[ \Pi_\infty(|K|, F) \simeq F \]
	supplied by \cref{Exit_simplicial_complex}, $\Pi_\infty(|U_{S,K}|, F)$ corresponds to the subposet $G_S \subset F$ of faces having non-empty intersection with $S$.
	Then the inclusion of posets
	\[ F_S \hookrightarrow G_S \]
	satisfies the assumptions of \cref{induction_poset_cocart_IHES}: indeed, since $S$ is full in $K$ we see that for every $\sigma \in G_S$ the intersection $\sigma \cap S$ is a face of $S$ and therefore provides a final object for $(F_S)_{/\sigma}$.
	Thus, denoting $i_S \colon |S| \hookrightarrow U$ the canonical inclusion, we deduce from \cref{induction_poset_cocart_IHES} that the induction functor
	\[ i_{S,!} \colon \Fun(\cI_{|S|}, \cE) \to  \Fun(\cI_{U_{S,K}}, \cE) \]
	preserves Stokes functors, so (2) is satisfied as well.
	\medskip
	
	We are left to prove (5).
	Assume now that $Z$ is compact.
	In particular, the set $G_S$ is finite.
	On the other hand, we have 
	\[ U_{S,K} = \bigcup_{\sigma \in G_S}  \overset{\circ}{|\sigma|} \ . \]
		Furthermore, the triangulation can be  constructed so that the interior of each simplex is subanalytic \cite[Theorem 2]{Hardt1976}.
	See also paragraph $10$ and Remark p1585 of \cite{Lojasiewicz1993}.
	Hence (5) follows from the fact that a finite union of subanalytic subsets is again subanalytic.
\end{proof}

\begin{cor}\label{piecewise_implies_local_elementary}
	Let $f \colon (X,P,\cI)\to (Y,Q)$ be a vertically piecewise elementary family of Stokes analytic stratified spaces.
	Then:
	\begin{enumerate}\itemsep=0.2cm
	\item $f \colon (X,P,\cI)\to (Y,Q)$ is locally elementary.
		\item  $(X,P,\cI)$ is locally elementary.
  
		\item If $f \colon X \to Y$ is proper, there exists a cover of $X$ by subanalytic open subsets $U$ such that $(U,P,\cI_U)$ is elementary.
	\end{enumerate}
\end{cor}

\begin{proof}
	Let $x$ be a point of $X$ and set $y \coloneqq f(x)$.
	Choose a closed subanalytic subset $Z$ of $X_y$ such that $(Z, P, \cI_Z)$ is elementary.
	Then \cref{thm:spreading_out}-(4) implies the existence of an elementary open neighborhood $U$ of $Z$ in $X_y$, so (1) follows.
      \cref{thm:spreading_out}-(4) also implies the existence of an elementary open neighborhood $U$ of $Z$ in $X$, so (2) follows.
	If furthermore $f$ is proper, then $X_y$ is compact and therefore the same goes for $Z$, so (3) follows from \cref{thm:spreading_out}-(5).
\end{proof}

\subsection{Level structures}

	Local elementarity is rarely satisfied.
	Level structures provide the key technical tool needed to bypass this difficulty: performing induction on the length of a level structure allows to reduce the complexity of the Stokes analytic stratified space, eventually reducing to the locally elementary case.



\begin{defin}\label{level_structure}
	Let $(X,P) \to (Y,Q)$ be a family of exodromic stratified spaces and let
	\[ p \colon \cI \to  \cJ \]
	be a morphism of Stokes fibrations over $(X,P)$.
	Fix a full subcategory $\cC \subseteq \FStStrat$.
	We say that $p$ is a \emph{simple $\cC$-level morphism relative to $(Y,Q)$} if the following conditions hold:
	\begin{enumerate}\itemsep=0.2cm
		\item $p$ is a level morphism (\cref{level_intro_def});
		
		\item both $\cI^{\ens}$ and $\cJ^{\ens}$ are pullback of Stokes fibrations in sets over $(Y,Q)$;
		
		\item for every $q \in Q$, the family of Stokes stratified spaces
		\[ (X_q, P_q, (\cI |_{X_q})_{p|_{X_q}}) \to  Y_q \]
		belongs to $\cC$ (see \cref{graduction_section} for the meaning of $\cI_p$).
	\end{enumerate}
	We say that $p$ is a \emph{$\cC$-level morphism relative to $(Y,Q)$} if it can be factored as a finite composition
	\begin{equation}\label{eq:level_structure}
		\begin{tikzcd}
			\cI = \cI^{d} \arrow{r}{p_d} & \cI^{d} \arrow{r}{p_{d-1}} & \cdots \arrow{r}{p_2} & \cI^1 \arrow{r}{p_1} & \cI^0 = \cJ
		\end{tikzcd}
	\end{equation}
	where each $\cI^k$ is a Stokes fibration over $(X,P)$ and each $p_k \colon \cI^k \to \cI^{k-1}$ is a simple $\cC$-level morphism relative to $(Y,Q)$.
	When $\cC = \FStStrat$, we simply say that $p$ is a (simple) level morphism relative to $(Y,Q)$.
\end{defin}

\begin{rem}\label{level_structure_above_strata}
	Assume that the stratification on $Y$ is trivial.
	Then Condition (2) ensures that if $p \colon \cI \to \cJ$ is a simple level morphism, then it is also a level graduation morphism above each stratum of $Y$.
\end{rem}

\begin{defin}\label{def:level_structure_length}
	In the situation of \cref{level_structure}, we refer to a factorization of $p \colon \cI \to \cJ$ of the form \eqref{eq:level_structure} as a \emph{$\cC$-level structure for $p$} and we say that $d$ is its \emph{length}.
	When $\cJ = \Pi_\infty(X,P)$, we say that \eqref{eq:level_structure} is a \emph{$\cC$-level structure for $\cI$}.
\end{defin}

\begin{defin}
	Let $\cC \subseteq \FStStrat$ be a full subcategory.
	We say that a family of Stokes stratified spaces $(X,P,\cI) \to (Y,Q)$ \emph{admits a $\cC$-level structure} if the morphism
	\[ p \colon \cI \to \Pi_\infty(X,P) \]
	is a $\cC$-level morphism relative to $(Y,Q)$.
	\personal{Here $\Pi_\infty(X,P) \simeq \Pi_\infty(X,P) \times \ast$ is seen as the trivial cocartesian fibration over itself via the identity map.}
	Similarly, we say that $(X,P,\cI) \to (Y,Q)$ \emph{locally admits a $\cC$-level structure} if $Y$ can be covered by open subsets $U$ such that each $(X_U, P, \cI_U) \to (U,Q)$ admits a $\cC$-level structure.
\end{defin}



As a consequence of \cref{piecewise_implies_local_elementary} and \cref{piecewise_implies_vertical_piecewise}, we obtain the following :

\begin{cor}\label{cor:piecewise_level_structure_implies_locally_elementary_level_structure}
	Let $f \colon (X,P,\cI) \to (Y,Q)$ be a family of Stokes analytic stratified spaces.
	Then:
	\begin{enumerate}\itemsep=0.2cm
		\item If $f$ has a vertically piecewise elementary  level structure then it has a locally elementary level structure;
		\item if $f$ has a piecewise elementary  level structure, then it has a vertically piecewise elementary  level structure;
		\item if $f$ has a strongly piecewise elementary  level structure, then it has a piecewise elementary  level structure.
	\end{enumerate}
\end{cor}

In the classical theory of Stokes data,  level structures exist only after some suitable ramified cover.
This is axiomatized by the following

\begin{defin}\label{stratified_Galois_cover}
A morphism in $\FStStrat$
\[
\begin{tikzcd}
			(X',P',\cJ) \arrow{r}{\rho'} \arrow{d}{f'} &(X,P,\cI)\arrow{d}{f}     \\
			(Y',Q')\arrow{r}{\rho} & (Y,Q)
		\end{tikzcd} 
\]		
is a \textit{vertically finite Galois cover} if the upper arrow is cartesian in $\StStrat$ 
and for every $y'\in Y'_q$ lying above a point $y\in Y_q$,  the induced map $X'_{y'} \to X_y$ is a finite Galois cover.
\end{defin}

\begin{rem}
For an example of vertically finite Galois cover arising in the theory of flat bundle, see \cref{remified_irregular_values}.
\end{rem}

\begin{defin}\label{ramified_level_structure}
	Let $\cC \subseteq \FStStrat$ be a full subcategory.
	We say that a family of Stokes stratified spaces $(X,P,\cI) \to (Y,Q)$ \emph{admits a ramified $\cC$-level structure} if there exists a vertically finite Galois cover as in \cref{stratified_Galois_cover} such that  $(X',P',\cJ) \to (Y',Q')$ admits a $\cC$-level structure.
	We say that $(X,P,\cI) \to (Y,Q)$ \emph{locally admits a ramified $\cC$-level structure} if $Y$ can be covered by open subsets $U$ such that each $(X_U, P_U, \cI_U) \to (U,Q)$ admits a ramified $\cC$-level structure.
\end{defin}

\subsection{Hybrid descent for Stokes functors}

	As observed in \cref{warning:local_elementarity_too_strong}, even when they exist, elementary open subsets do not form a basis of the topology.
	For this reason, we need to discuss a hybrid descent property for the $\infty$-category of Stokes functors that combines $\frakStokes$ on elementary opens and $\frakFil$ on their further intersections.
	This is achieved via the following:

\begin{construction}\label{construction:hybrid_descent}
	Let $(X,P,\cI)$ be a Stokes stratified space.
	Let $\mathcal U = \{\mathbf U_\bullet\}$ be a hypercover of $X$.
	We define the semi-simplicial diagram
	\[ \StFil_{\cI,\cE}^{\mathcal U} \colon \mathbf \Delta\op_s \to \CAT_\infty \]
	as the subfunctor of
	\[ \frakFil_{\cI, \cE} \circ \mathbf U_\bullet \colon \mathbf \Delta\op_s \to \CAT_\infty \]
	defined by
	\[ \StFil_{\cI,\cE}^{\mathcal U}([n]) \coloneqq \begin{cases}
		\frakStokes_{\cI,\cE}(\mathbf U_0) & \text{if } n = 0 \\
		\frakFil_{\cI, \cE}(\mathbf U_n) & \text{if } n > 0 \ .
	\end{cases} \]
	Notice that it is well defined thanks to the commutativity of \eqref{eq:Stokes_sheaves_functoriality}.
\end{construction}

\begin{prop}\label{Stokes_via_Fil}
	Let $(X,P,\cI)$ be a Stokes stratified space.
	Let $\mathcal U = \{\mathbf{U}_\bullet\}$ be a hypercover of $X$.
	Then the canonical functor
	\[ \St_{\cI,\cE} \to \lim_{\mathbf \Delta_s\op} \StFil_{\cI,\cE}^{\cU} \]
	is an equivalence of $\infty$-categories.
\end{prop}

\begin{proof}
	For every $n\geq 0$, the functors
	\[ \frakStokes_{\cI,\cE}(\mathbf{U}_n)\to  \frakFil_{\cI,\cE}(\mathbf{U}_n)\]
	are fully-faithful.
	Since $\frakStokes_{\cI,\cE}$ and $\frakFil_{\cI,\cE}$ are hypersheaves,  passing to the limit thus yields fully-faithful functors 
	\[ \St_{\cI,\cE} \hookrightarrow  \lim_{[n] \in \mathbf \Delta\op_s} \StFil_{\cI,\cE}^{\mathcal U}([n]) \hookrightarrow \Funcocart(\cI,\cE) \ . \]
	By definition, an object of the middle term is a cocartesian functor $F \colon \cI \to \cE$ such that $F |_{U_0}$ is a Stokes functor.
	In particular, $F$ is punctually split at every point of $X$.
	Hence, $F$ is a Stokes functor.
	This concludes the proof of \cref{Stokes_via_Fil}.
\end{proof}

\begin{rem}\label{Stokes_via_Fil_finite_cover}
	If $\mathbf{U}_\bullet$  is the hypercover induced by a finite cover $U_1,\dots, U_n$ of $X$, then the limit appearing in \cref{Stokes_via_Fil} can be performed over the \textit{finite} subcategory $\Delta\op_{\leq n,s}$ of $\Delta\op_s$.
\end{rem}

Under some suitable finiteness and stability conditions, the diagram $\StFil_{\cI,\cE}$  takes value in $\PrLR$ (\cref{defin_PrLR}).

%

\begin{cor}\label{StFil_in_prLR}
	Let $(X,P,\cI)$ be a Stokes stratified space in finite posets.
	Let $\cU = \{ \mathbf{U}_\bullet \}$ be a hypercover of $X$ such that $(\mathbf{U}_0,P,\cI_{\mathbf{U}_0})$ is elementary.
	Then the semi-simplicial diagram of \cref{construction:hybrid_descent} lifts to a functor
	\[ \StFil_{\cI,\cE}^{\cU} \colon \mathbf \Delta\op_s \to \PrLR \ . \]  
	In particular,  the folllowing equivalence supplied by \cref{Stokes_via_Fil} 
	\[ \St_{\cI,\cE} \simeq  \lim_{\mathbf \Delta\op_s}\StFil_{\cI,\cE}^{\cU} \]
	is an equivalence in $\PrL$, where the limit is computed in $\PrL$.
\end{cor}

\begin{proof}
	Since $U_0$ is elementary, the definition of $\StFil_{\cI,\cE}^\cU$ yields:
	\[ \StFil_{\cI,\cE}^{\cU}( \mathbf U_n ) \simeq \begin{cases}
		\Funcocart(\cI_{\mathbf{U}_0}^{\ens},\cE) & \text{if } n = 0 \\
		\Funcocart(\cI_{\mathbf{U}_n},\cE) & \text{if } n > 0 \ .
	\end{cases} \] 
	In both cases, $\StFil_{\cI,\cE}$ takes values in $\PrL$ by \cref{(co)limit_and_cocart}-(1).
	Let $f \colon [n] \to [m]$ be a morphism in $\mathbf \Delta\op$ and let $i_f \colon \mathbf U_n \to \mathbf U_m$ be the associated morphism.
	When $m > 0$ the corresponding transition functor for $\StFil_{\cI,\cE}^\cU$ is just
	\[ i_f^\ast \colon \Funcocart(\cI_{\mathbf U_m}, \cE) \to \Funcocart( \cI_{\mathbf U_n}, \cE) \ , \]
	while for $m = 0$, it identifies with
	\[ i_f^\ast \circ i_{\cI_{\mathbf U_0},!} \colon \Funcocart(\cI_{\mathbf{U}_0}^{\ens},\cE) \to \Funcocart(\cI_{\mathbf{U}_n},\cE) \ . \]
	In both cases, \cref{(co)limit_and_cocart} and  \cref{induction_limits} show they are both left and right adjoints.
	To conclude the proof of \cref{StFil_in_prLR},  use \cref{Stokes_via_Fil} and the fact that $\PrL \to \Cat_{\infty}$ commutes with limits.
\end{proof}

\section{Finite type property for Stokes functors}\label{nc_space}

	Let $(X,P,\cI)$ be a Stokes stratified space and fix an animated ring $k$.
	We consider the $\infty$-category
\[ \St_{\cI,k} \coloneqq \St_{\cI, \Mod_k} \ . \]
	We saw in \cref{eg:Stokes_structures_at_a_point} that in general $\St_{\cI,k}$ does not inherit any of the good properties of $\Mod_k$: for instance, it is neither presentable nor stable.
	The goal of this section is to prove that on the other hand $\St_{\cI,k}$ is well behaved when $(X,P,\cI)$ admits a \emph{locally ramified vertically piecewise elementary level structure}.
	This is a strong condition forcing a highly non-trivial interaction between $X$ and $\cI$.

\subsection{Stability}

	The following is the key result of this work: 

\begin{thm}\label{thm:St_presentable_stable}
	Let $f \colon (X,P,\cI) \to (Y,Q)$ be a family of Stokes stratified spaces in finite posets.
	Assume that $f$ locally admits a ramified locally elementary level structure.
	Then $\St_{\cI, \cE}$ is presentable and stable.
\end{thm}

      \cref{thm:St_presentable_stable} will follow from a more precise statement (see \cref{presentable_stable} below) exhibiting $\St_{\cI, \cE}$ as a localization of $\Fun(\cI,\cE)$.
	With this goal in mind, we start setting up the stage with a couple of preliminaries lemmas.

\begin{lem}\label{lem:local_elementary_stability_lim_colim}
	Let $(X,P,\cI)$ be a locally elementary Stokes stratified space.
	Then:
	\begin{enumerate}\itemsep=0.2cm
		\item \label{lem:local_elementary_stability_lim_colim:lim} $\St_{\cI,\cE}$ is closed under colimits in $\Fun(\cI,\cE)$;
		\item \label{lem:local_elementary_stability_lim_colim:colim} if in addition the fibers of $\cI$ are finite, then $\St_{\cI,\cE}$ is closed under limits in $\Fun(\cI,\cE)$.
		In other words, $(X,P,\cI)$ is bireflexive.
	\end{enumerate}
\end{lem}

\begin{proof}
	Thanks to \cref{(co)limit_and_cocart}-(1)  we see that $\Funcocart(\cI,\cE)$ is closed under colimits in $\Fun(\cI,\cE)$.
	Similarly, when the fibers of $\cI$ are finite posets,  \cref{(co)limit_and_cocart}-(2) implies that $\Funcocart(\cI,\cE)$ is stable under limits in $\Fun(\cI,\cE)$.
	Let now $F_\bullet \colon I \to \Fun(\cI,\cE)$ be a diagram such that for every $i \in I$, the functor $F_i \colon \cI \to \cE$ is Stokes and set
	\[ F_{\lhd} \coloneqq \lim_{i \in I} F_i \ , \qquad F_{\rhd} \coloneqq \colim_{i \in I} F_i \ , \]
	where the limit and the colimit are computed in $\Fun(\cI,\cE)$.
	To check that $F_{\lhd}$ and $F_{\rhd}$ are Stokes, we are left to check that they are pointwise split.
	This question is local on $X$ and since $(X,P,\cI)$ is locally elementary, we can therefore assume that it is elementary to begin with.
	In this case, the top horizontal arrow in the commutative square
	\[ \begin{tikzcd}[column sep = small]
		\Funcocart(\cI^{\ens},\cE) \arrow[hookrightarrow]{d} \arrow{r}{i_{\cI,!}} &  \St_{\cI,\cE} \arrow[hookrightarrow]{d} \\
		 \Fun(\cI^{\ens},\cE)\arrow{r}{i_{\cI,!}}  & \Fun(\cI,\cE)
	\end{tikzcd} \]
	is an equivalence.
	Thus, we deduce that $F_{\rhd}$ is Stokes from the fact that $i_{\cI,!}$ commutes with colimits.
	Similarly, when the fibers of $\cI$ are finite posets, we deduce that $F_{\lhd}$ is Stokes from \cref{induction_limits}.
\end{proof}

\begin{thm}\label{stability_lim_colim_ISt}
	Let $f \colon (X,P,\cI) \to (Y,Q)$ be a family of Stokes stratified spaces in finite posets locally admitting a ramified locally elementary level structure.
	Then $(X,P,\cI)$ is  bireflexive.
\end{thm}

\begin{proof}
     Let $\cE$ be a presentable stable $\infty$-category.
     Let $F_\bullet \colon I \to \Fun(\cI, \cE)$ be a diagram such that for every $i \in I$ the functor $F_i \colon \cI \to \cE$ is Stokes and set
	\[ F_{\lhd} \coloneqq \lim_{i \in I} F_i \ , \qquad F_{\rhd} \coloneqq \colim_{i \in I} F_i \ , \]
	where the limit and the colimit are computed in $\Fun(\cI,\cE)$.
    By  \cref{(co)limit_and_cocart}, the functors $F_{\lhd}$ and $F_{\rhd}$  are cocartesian.
    We are thus left to show that they are punctually split.
    Hence, we can suppose that $Y$ is a point and that $(X,P,\cI)$ admits a ramified locally elementary level structure.
    Since we are checking a punctual condition on $X$, we can further suppose that $(X,P,\cI)$ admits a locally elementary level structure.
	We argue by induction on the length $d$ of the locally elementary level structure.
	If $d = 0$, then $\cI = \Pi_\infty(X,P)$ is a fibration in sets, so that $\St_{\cI,\cE}=\Funcocart(\cI,\cE)$ and the result follows from \cref{(co)limit_and_cocart}.
	Otherwise, our assumption guarantees the existence of a level morphism $p \colon \cI \to \cJ$ such that:
	\begin{enumerate}\itemsep=0.2cm
		\item $\cJ$ admits a locally elementary level structure of length $< d$;
		\item $(X,P,\cI_p)$ is locally elementary.
	\end{enumerate}
	Notice that since level morphisms are surjective, the fibers of $\cJ$ are again finite posets, so the inductive hypothesis applies to the Stokes stratified space $(X,P,\cJ)$.	
	Consider the pullback square
	\[  \begin{tikzcd}
		\St_{\cI,\cE} \arrow{r}{p_!}\arrow{d}{\Gr_p } &\St_{\cJ,\cE}\arrow{d}{\Gr} \\
		\St_{\cI_p,\cE}\arrow{r}{\pi_!}& \St_{\cJ^{\ens},\cE}
	\end{tikzcd} \]
	supplied by \cref{level_dévissage}.
	The Stokes detection criterion of \cref{Stokes_detection_IHES} implies that $F_{\lhd}$ is Stokes if and only if both $\Gr_p(F_\lhd)$ and $p_!(F_\lhd)$ are Stokes, and similarly for $F_\rhd$ in place of $F_\lhd$.
	\cref{cor:section_Gr_commutes_with_colimits} guarantee that $\Gr_p$ commutes with both limits and colimits.
	Similarly, $p_!$ commutes with colimits because it is a left adjoint; since the fibers of $\cI$ are finite posets \cref{induction_limits}   implies that $p_!$ commutes with limits as well.
	Thus, we are reduced to check that
	\[ \Gr_p(F_\lhd) \simeq \lim_i \Gr_p(F_i) \in \Fun(\cI_p, \cE) \qquad \text{and} \qquad p_!(F_\lhd) \simeq \lim_i p_!(F_i) \in \Fun(\cJ, \cE) \]
	are Stokes functors, and similarly for the colimit in place of the limit and $F_\rhd$ in place of $F_\lhd$.
	\Cref{Gr_of_Stokes_IHES} ensures that $\Gr_p(F_i)$ is Stokes for every $i \in I$, while \cref{cor:stokes_functoriality_IHES}-(\ref{cor:stokes_functoriality:induction}) guarantees that $p_!(F_i)$ is Stokes for every $i \in I$.
	Thus, the induction hypothesis implies that $p_!(F_\lhd)$ and $p_!(F_\rhd)$ are Stokes.
	On the other hand, since $\cI_p$ is locally elementary, \cref{lem:local_elementary_stability_lim_colim} implies that $\Gr_p(F_\lhd)$ and $\Gr_p(F_\rhd)$ are Stokes as well, and the conclusion follows.
\end{proof}

At this point, \cref{thm:St_presentable_stable} follows from the following more precise:

\begin{cor}\label{presentable_stable}
	Let $f \colon (X,P,\cI)\to (Y,Q)$ be a family of Stokes stratified spaces in finite posets locally admitting a ramified locally elementary level structure.
	Then $\St_{\cI,\cE}$ is a localization of $\Fun(\cI,\cE)$, and in particular it is presentable stable.
\end{cor}
\begin{proof}
	Since $\cE$ is presentable stable,  $\Fun(\cI,\cE)$ is presentable stable in virtue of  \cite[Proposition 5.5.3.6]{HTT} and \cite[Proposition 1.1.3.1]{Lurie_Higher_algebra}.
	Then, the conclusion follows from the $\infty$-categorical reflection theorem, see \cite[Theorem 1.1]{Ragimov_Schlank_Reflection}.
\end{proof}

\begin{recollection}\label{recollection:generators_presheaves}
	For every $a \in \cI$ it follows from the Yoneda lemma that 
	\[ \ev_{a,!}^\cA \colon \Spc \to  \Fun(\cI, \Spc) \]
	is the unique colimit-preserving functor sending $*$ to $\Map_\cI(a,-)$.
	The density of the Yoneda embedding implies therefore that $\Fun(\cI, \Spc)$ is generated under colimits by $\{ \ev_{a, !}^\cA(*) \}_{a \in \cI}$.
	More generally, assume that $\cE$ is generated under colimits by a set $\{E_\alpha\}_{\alpha \in I}$.
	Then under the identification
	\[ \Fun(\cI, \cE) \simeq \Fun(\cI, \Spc) \otimes \cE \]
	we see that $\ev_{a,!}(E_\alpha) \simeq \ev_{a,!}^\cI(\ast) \otimes E_\alpha$ and therefore that $\{ \ev_{a, !}^\cI(E_\alpha) \}_{a \in \cA, \alpha \in I}$ generates $\Fun(\cI, \cE)$ under colimits.
\end{recollection}

\begin{cor}\label{compactly_generated}
	Let $f \colon  (X,P,\cI)\to (Y,Q)$ be a family of Stokes stratified spaces in finite posets locally admitting a ramified locally elementary level structure.
	Let $\cE$ be a presentable stable compactly generated $\infty$-category.
	Let $\{E_\alpha\}_{\alpha \in I}$ be a set of compact generators for $\cE$.
	Then $\St_{\cI,\cE}$ is presentable stable compactly generated by the $\{\LSt_{\cI,\cE}(\ev_{a,!}(E_{\alpha}))\}_{\alpha \in I,a\in \cI }$ where the $\ev_a\colon\{a\}\to \cI$ are the canonical inclusions and where $\LSt_{\cI,\cE}$ is left adjoint to the inclusion $\St_{\cI,\cE}\hookrightarrow \Fun(\cI,\cE)$.
\end{cor}

\begin{proof}
$\St_{\cI,\cE}$ is presentable stable in virtue of \cref{presentable_stable}.
By \cref{recollection:generators_presheaves}, the $\{\ev_{a,!}(E_{\alpha})\}_{\alpha \in I,a\in \cI }$ are compact generators of $\Fun(\cI,\cE)$.
Then, \cref{compactly_generated} formally follows.
\end{proof}

	Thanks to 
\cite[\cref*{Abstract_Stokes-Grothendieck_abelian_abstract}]{Abstract_Derived_Stokes}, we obtain the following:

\begin{cor}\label{Grothendieck_abelian}
	Let $f \colon (X,P,\cI) \to (Y,Q)$ be a family of Stokes stratified spaces in finite posets locally admitting a ramified locally elementary level structure.
	Let $\cA$ be a Grothendieck abelian category.
	Then $\St_{\cI,\cA}$ is a Grothendieck abelian category.
\end{cor}


\begin{rem}
	By \cref{cor:piecewise_level_structure_implies_locally_elementary_level_structure}, all the results stated so far hold for families of Stokes analytic stratified spaces in finite posets $f \colon (X,P,\cI) \to (Y,Q)$ locally admitting a ramified vertical piecewise elementary level structure.
\end{rem}

	The following proposition amplifies \cref{presentable_stable} in the analytic setting:

\begin{prop}\label{direct_image_factors_trough_PrL}
	Let $f \colon  (X,P,\cI)\to (Y,Q)$ be a proper family of Stokes analytic stratified spaces in finite posets locally admitting a ramified locally elementary level structure.
	Then, the following hold:
	\begin{enumerate}\itemsep=0.2cm
		\item For every open subsets $U\subset V$, the functor 
$$
		f_\ast(\frakStokes_{\cI,\cE})(V)\to f_\ast(\frakStokes_{\cI,\cE})(U)
$$ is a left and right adjoint.
		\item There exists a subanalytic refinement $R\to Q$ such that $f_\ast(\frakStokes_{\cI,\cE})\in \ConsRhyp(X; \PrL)$.
		\item For every subanalytic refinement $R\to Q$ such that $f_\ast(\frakStokes_{\cI,\cE})\in \ConsRhyp(X; \PrL)$,  the hypersheaf $f_\ast(\frakStokes_{\cI,\cE})$ is an object of $\ConsRhyp(X; \PrLR)$.
	\end{enumerate}      
\end{prop}

\begin{proof}
	Item (1) is an immediate consequence of  \cref{stability_lim_colim_ISt} and \cref{presentable_stable}.
	The existence of an analytic refinement as in (2)  is a consequence of \cref{proper_analytic_direct_image}.
	Then (3) follows from (1) and \cref{PrLR_criterion}.
\end{proof}

\subsection{Stokes functors and tensor product}\label{subsec:Stokes_categorical_tensor_product}

We analyze more thoroughly the interaction between the category of Stokes functor and the tensor product in $\PrL$.
We first recall the following

\begin{lem}
[{\cite[\cref*{Abstract_Stokes-cocart_tensor_for_exp}]{Abstract_Derived_Stokes}}]\label{cocart_tensor_for_exp_IHES}
	Let $(X,P,\cI)$ be a Stokes stratified space in finite posets.
	Let $\cE,\cE'$ be  presentable stable $\infty$-categories.
	Then, the canonical transformation 
	\[
	\Funcocart(\cI,\cE)\otimes\cE'   \to \Funcocart(\cI,\cE\otimes \cE')
	\]
	is an equivalence.
\end{lem}

\begin{lem}\label{tensor_product_locally_split_case}
	Let $(X,P,\cI)$ be a locally elementary Stokes stratified space in finite posets.
	Then $(X,P,\cI)$ is universal.
\end{lem}

\begin{proof}
Note that $(X,P,\cI)$ is bireflexive by \cref{lem:local_elementary_stability_lim_colim}.
Let $\cE,\cE'$ be presentable stable $\infty$-categories.
	Since $(X,P,\cI)$ is locally elementary, we can find a cover $\{U_i\}$ such that $(U_i,P,\cI_{U_i})$ is elementary.
    Let $\cU = \{ \mathbf{U}_\bullet \}$ be its \v{C}ech nerve.
	Recall from \cref{construction:hybrid_descent} the semi-simplicial diagram
	\[ \StFil_{\cI,\cE}^{\cU} \colon \mathbf \Delta\op_s \to \CAT_{\infty} \ . \]
	By \cref{StFil_in_prLR}, this functor takes values in $\PrLR$.
	Therefore, we can tensor it with $\cE'$, finding:
	\begin{align*}
		\St_{\cI,\cE} \otimes \cE' & \simeq \big( \lim_{\mathbf \Delta\op_s} \StFil_{\cI,\cE} \big) \otimes \cE' & \text{By Cor.\ \ref{StFil_in_prLR}}  \\
		&\simeq \lim_{\mathbf \Delta\op_s} \big(\StFil_{\cI,\cE}^\cU \otimes \cE'\big) & \text{By Lem.\ \ref{Peter_lemma}}  \\
		& \simeq \lim_{\mathbf \Delta\op_s} \StFil_{\cI,\cE\otimes \cE'}^\cU & \text{By Lem.\ \ref{cocart_tensor_for_exp_IHES}}  \\
		& \simeq \St_{\cI,\cE\otimes \cE'} &\text{By Cor.\ \ref{StFil_in_prLR}}
	\end{align*}
	The conclusion follows.
\end{proof}

\begin{prop}\label{thm:Stokes_tensor_product_absolute}
	Let $(X,P,\cI)$ be a Stokes stratified space in finite posets admitting a locally elementary level structure.
	Then $(X,P,\cI)$ is universal.
\end{prop}

\begin{proof}
Note that $(X,P,\cI)$ is bireflexive by \cref{stability_lim_colim_ISt}.
Let $\cE,\cE'$ be presentable stable $\infty$-categories.
	We proceed by induction on the length $d$ of the locally elementary level structure.
	When $d = 0$, $\cI = \Pi_\infty(X,P)$ and $(X,P,\cI)$ is elementary, so the conclusion follows from \cref{tensor_product_locally_split_case}.
	Otherwise, our assumption guarantees the existence of a level morphism $p \colon \cI \to \cJ$ such that:
	\begin{enumerate}\itemsep=0.2cm
		\item $\cJ$ admits a locally elementary level structure of length $< d$;
		\item $(X,P,\cI_p)$ is locally elementary.
	\end{enumerate}
	Notice that since level morphisms are surjective, the fibers of $\cJ$ are again finite posets, so the inductive hypothesis applies to the Stokes stratified space $(X,P,\cJ)$.
	Consider the following commutative cube:
	\begin{equation}\label{commutation_tensor_product_cube}
		\begin{tikzcd}
			\St_{\cI,\cE} \otimes\cE'\arrow{rr}{p_!\otimes  \cE'} \arrow{dr} \arrow{dd}{\Gr_p \otimes  \cE'}& & \St_{\cJ,\cE} \otimes\cE' \arrow{dr}{} \arrow[near start]{dd}{\Gr \otimes  \cE'}  \\
			{} & \St_{\cI,\cE\otimes\cE'}  \arrow[crossing over, near start]{rr}{p_!}& &\St_{\cJ,\cE\otimes\cE'}  \arrow{dd}{\Gr } \\
			\St_{\cI_p,\cE} \otimes\cE'\arrow[near end]{rr}{\pi_!\otimes  \cE'}\arrow{dr} & & \St_{\cJ^{\ens},\cE} \otimes\cE'\arrow{dr} \\
			{} &\St_{\cI_p,\cE\otimes\cE'} \arrow{rr}{\pi_!} \arrow[leftarrow, crossing over, near end]{uu}{\Gr_p } & & \St_{\cJ^{\ens},\cE\otimes\cE'} \ .
		\end{tikzcd}
	\end{equation} 
	whose front face is a pull-back in virtue of \cref{level_dévissage}.
	Combining \cref{stability_lim_colim_ISt}, \cref{level_pullback_in_PrLR} and \cref{Peter_lemma} we deduce that the back face is a pullback in $\PrLR$.
	\Cref{tensor_product_locally_split_case} shows that the bottom diagonal arrows are equivalences while the upper right diagonal arrow is an equivalence by the inductive hypothesis.
	Hence, so is the top left diagonal arrow.
\end{proof}

	Working in the analytic setting, we can formulate a stronger version of the above result.

\begin{construction}\label{comparison:Stokes_sheaf_comparison_tensor}
	Let $f \colon (X,P,\cI)\to (Y,Q)$ be a family of Stokes analytic stratified spaces in finite posets admitting a vertically piecewise elementary level structure.
	Fix stable presentable $\infty$-categories $\cE$ and $\cE'$.
	For every open subset $V \subset Y$, the induced family $(X_V, P, \cI_V) \to (Y,Q)$ admits again a vertically piecewise elementary level structure.
	Thus, \cref{stability_lim_colim_ISt} shows that $\St_{\cI_V,\cE}$ and $\St_{\cI_V,\cE \otimes \cE'}$ is closed under limits and colimits in $\Fun(\cI_V,\cE)$ and $\Fun(\cI_V,\cE \otimes \cE')$, respectively.
	Then,  \cite[\cref*{Abstract_Stokes-construction:Stokes_tensor_comparison}]{Abstract_Derived_Stokes}  yields a  fully faithful functor
	\[ \St_{\cI_V, \cE} \otimes \cE' \to \St_{\cI_V, \cE \otimes \cE'} \ . \]
	Since this comparison map depends functorially on $V$, we deduce the existence of a commutative square
	\[ \begin{tikzcd}
		f_\ast(\frakStokes_{\cI,\cE})\otimes \cE'  \arrow[hook]{d} \arrow[hook]{r}&  f_\ast(\frakStokes_{\cI,\cE\otimes \cE'})\arrow[hook]{d}   \\
		f_\ast(\frakFil_{\cI,\cE} )\otimes \cE'  \arrow{r}{\sim} & f_\ast(\frakFil_{\cI,\cE\otimes \cE'} )
	\end{tikzcd} \]
	in $\PSh(Y; \PrL)$.
\end{construction}

\begin{lem}\label{topological_retraction_lemma}
	Let $f \colon (Y,P,\cJ)\to (X,P,\cI)$ be a cartesian finite Galois cover in $\StStrat$ where $(X,P)$ is conically refineable and where $(Y,P,\cJ)$ is universal.
	Then $(X,P,\cI)$  is universal and 
	the canonical functor 
	\[ \Loc(Y;\Sp) \otimes_{\Loc(X;\Sp)} \St_{\cI,\cE} \to \St_{\cJ, \cE} \]
	is an equivalence.
	
\end{lem}

\begin{proof}
	Let $Y_{\bullet} \colon \mathbf \Delta_{s}\op \to \Top_{/X}$ be the Cech complex of $f \colon Y \to X$ and put 
	\[ \cI_\bullet \coloneqq \Pi_{\infty}(Y_{\bullet},P) \times_{\Pi_{\infty}(X,P)} \cI   \ . \]
    Since $f \colon Y \to X$ is Galois, $Y_n$ is a finite coproduct of copies of $Y$ over $X$.	
	Hence,  $(Y_n,P)$ is conically refineable for every $[n] \in \mathbf \Delta_s$ (thus exodromic by \cref{conically_refineable_implies_exodromic}) and  $(Y_n,P,\cI_n)$ is universal for every $[n]\in \mathbf \Delta_{s}$.
    Then $(X,P,\cI)$  is universal by \cref{étale_hyperdescent}.
	Since the $Y\to X$ is a finite étale cover, \cite[\cref*{Abstract_Stokes-lem:finite_etale_fibrations_example}]{Abstract_Derived_Stokes} implies that 
	\[ \Pi_{\infty}(Y,P)  \to \Pi_{\infty}(X,P) \]
	is a finite étale fibration in the sense of \cite[\cref*{Abstract_Stokes-def:finite_etale_fibrations}]{Abstract_Derived_Stokes}.
	We deduce from  \cite[\cref*{Abstract_Stokes-cor:finite_etale_fibration_Stokes_relative_tensor}]{Abstract_Derived_Stokes} that the canonical functor 
	\[ \Loc(Y;\Sp) \otimes_{\Loc(X;\Sp)} \St_{\cI,\Sp} \to \St_{\cJ,\Sp} \]
	is an equivalence.
	Tensoring  with $\cE$ and using the universality thus concludes the proof of \cref{topological_retraction_lemma}.
\end{proof}

\begin{thm}\label{thm:Stokes_tensor_product}
	Let $f \colon (X,P,\cI)\to (Y,Q)$ be a proper family of Stokes analytic stratified spaces in finite posets locally admitting a ramified vertically piecewise elementary level structure.
	Let $\cE$ and $\cE'$ be stable presentable $\infty$-categories.
	Then the canonical functor   
	\begin{equation}\label{commutation_tensor_product_eq}
		f_\ast(\frakStokes_{\cI,\cE})\otimes \cE'  \to  f_\ast(\frakStokes_{\cI,\cE\otimes \cE'})
	\end{equation}
	is an equivalence.
	In particular, $(X,P,\cI)$ is universal.
\end{thm}

\begin{proof}
	The second half follows from the first because $f_\ast( \frakStokes_{\cI,\cE} ) \otimes \cE'$ is by definition the tensor product computed in $\PSh(Y;\CAT_\infty)$.
	To prove the first half, observe that both sides of \eqref{commutation_tensor_product_eq} are hyperconstructible hypersheaves by \cref{direct_image_factors_trough_PrL} and \cref{naive_tensor_is_hypersheaf}.
	Hence, the equivalence can be checked at the level of stalks.
	Since $f$ is proper, Propositions \ref{proper_analytic_direct_image} and \ref{pullback_Stokes_sheaf} allow to reduce ourselves to the case where $Y$ is a point.
	That is, we are left to show that $(X,P,\cI)$ is universal.
     In that case, there exists  a cartesian finite Galois cover $ (Y,P,\cJ)\to (X,P,\cI)$ such that $(Y,P,\cJ)$ admits a vertically piecewise elementary level structure.
	Recall that $(X,P)$ is conically refineable in virtue of \cref{eg:subanalytic_implies_combinatorial}.
	By \cref{topological_retraction_lemma}, it is thus enough  to show that 
	$(Y,P,\cJ)$ is universal.
	Hence, we can suppose that $(X,P,\cI)$ admits a vertically piecewise elementary level structure.
	In this case, \cref{cor:piecewise_level_structure_implies_locally_elementary_level_structure} guarantees that $(X,P,\cI)$ admits a locally elementary level structure, so the conclusion follows from \cref{thm:Stokes_tensor_product_absolute}.
\end{proof}

\subsection{Finite type property for Stokes structures}

	We proved \cref{thm:St_presentable_stable} under two key assumptions on the Stokes stratified space $(X,P,\cI)$: the local existence of a ramified locally elementary level structure and the fibers of $\cI$ are finite posets.
	We now analyze the categorical finiteness properties of $\St_{\cI,\cE}$: under some stricter geometrical assumptions on $(X,P,\cI)$ and working in the \emph{analytic setting} we establish that it is of finite type and hence smooth in the non-commutative sense (see e.g.\ \cite[Definition 11.3.1.1]{Lurie_SAG}).

\begin{defin}\label{strongly_proper}
	Let $f \colon (M,X)\to (N,Y)$ be a subanalytic morphism.
	We say that $f$ is \textit{strongly proper} if it is proper and for every finite subanalytic stratifications $X \to P$ and $Y \to Q$ such that $f \colon (X,P) \to (Y,Q)$ is a subanalytic stratified map,	there exists a categorically \emph{finite} subanalytic refinement $R \to Q$ such that for every $\cF \in \ConsPhyp(X;\CAT_{\infty})$, we have $f_\ast(\cF)\in \ConsRhyp(Y;\CAT_{\infty})$.
\end{defin}

\begin{eg}\label{ex_strongly_proper}
	By \cref{proper_analytic_direct_image} and \cref{categorical_compactness}, every proper subanalytic map $f\colon(M,X)\to (N,Y)$ with $Y$ compact is strongly proper.
\end{eg}

	The following lemma  is our main source of strongly proper morphisms.

\begin{lem}\label{compactifiable_is_stronly_proper}
	Let $f \colon (M,X)\to (N,Y)$ be a proper subanalytic morphism.
	Assume the existence of a commutative diagram
	\[ 
	\begin{tikzcd}
		(M,X) \arrow[hook]{r}{j}\arrow{d}{f} & (\overline{M},\overline{X})  \arrow{d}{g} \\
		(N,Y) \arrow[hook]{r}{i} &  (\overline{N},\overline{Y})
	\end{tikzcd} 
	\]
	such that $g$ is proper, $ \overline{Y} $ is compact and the horizontal arrows are open immersions with subanalytic complements.
	Then $f \colon (M,X)\to (N,Y)$  is strongly proper.
\end{lem}

\begin{proof}
	Let $X\to P$ and $Y\to Q$ be finite subanalytic stratifications such that $f \colon (X,P)\to (Y,Q)$ is a subanalytic stratified map.
	Extend $ X\to P $ to a  subanalytic stratification $ \overline{X}\to P^{\lhd} $ by sending $ \overline{X}\setminus X $ to the initial object of $ P^{\lhd} $.
	Extend $ Y\to Q $ to a subanalytic stratification $ \overline{Y}\to Q^{\lhd} $ by sending $ \overline{Y}\setminus Y $ to the initial object of $ Q^{\lhd} $.
	By \cref{proper_analytic_direct_image}  applied  to the proper map $ g \colon  (\overline{M},\overline{X}) \to  (\overline{N},\overline{Y}) $, there is a finite subanalytic refinement $ S \to Q^{\lhd} $ such that for every $ F\in \ConsPhyp(X;\CAT_{\infty}) $, we have $ g_\ast(j_!(F))\in \ConsShyp(\overline{Y};\CAT_{\infty}) $.
	Let $ R\subset S $ be the (finite) open subset of elements not mapped to the initial object of $ Q^{\lhd} $ by $ S \to Q^{\lhd} $.
	Then,  $f_\ast(\cF)\in \ConsRhyp(Y;\CAT_{\infty})$ with $ (Y,R) $ categorically finite by \cref{categorical_compactness}.
\end{proof}

	From now on, we fix an animated commutative ring $k$ and a compactly generated $k$-linear stable $\infty$-category $\cE$.

\begin{observation}\label{observation:k_linear_structure}
	For every Stokes stratified space $(X,P,\cI)$, we see that $\Fun(\cI, \cE)$ is again compactly generated and $k$-linear.
	When the fibers of $\cI$ are finite, \cref{(co)limit_and_cocart}-(2) implies that $\Funcocart(\cI, \cE)$ is a localization of $\Fun(\cI, \cE)$ and therefore inherits a $k$-linear structure.
	When $(X,P,\cI)$ admits a locally elementary level structure \cref{presentable_stable} implies that $\St_{\cI,\cE}$ inherits a $k$-linear structure.
\end{observation}

\begin{thm}\label{finite_typeness}
	Let $f \colon  (X,P,\cI) \to (Y,Q)$ be a strongly proper family of Stokes analytic stratified spaces in finite posets locally admitting a ramified vertically piecewise elementary level structure.
	Let $k$ be an animated ring and let $\cE$ be a compactly generated $k$-linear stable $\infty$-category of finite type (\cref{finite_type_defin}).
	Then $\St_{\cI,\cE}$ is of finite type relative to $k$ as well.
\end{thm}


\begin{cor}
	In the setting of \cref{finite_typeness},
	\[ \St_{\cI,k} \coloneqq \St_{\cI,\Mod_k} \]
	is a smooth $k$-linear presentable stable $\infty$-category.
\end{cor}

\begin{proof}
	This simply follows because finite type $k$-linear categories are smooth over $k$ in the non-commutative sense, see e.g.\ \cite[Proposition 2.14]{Toen_Moduli}.
\end{proof}



\begin{lem}\label{finite_typeness_piecewise_elementary_case}
	Let $(X,P,\cI)$ be a compact piecewise elementary Stokes analytic stratified space in finite posets.
	Let $k$ be an animated ring and let $\cE$ be a compactly generated $k$-linear stable $\infty$-category of finite type.
	Then $\St_{\cI,\cE}$ is of finite type relative to $k$.
\end{lem}

\begin{proof}
	Thanks to \cref{thm:spreading_out}, $X$ admits a finite cover by relatively compact subanalytic open subsets $U_1,\dots, U_n$ such that $(U_i,P,\cI_{U_i})$ is elementary for every $i=1,\dots,n$.
	In particular, each term of the associated hypercover $\cU = \{ \mathbf{U}_\bullet \}$  is a relatively compact \textit{subanalytic} open subset.
	From \cref{Stokes_via_Fil} and \cref{Stokes_via_Fil_finite_cover}, we have a canonical equivalence
	\[ \St_{\cI,\cE} \simeq  \lim_{\mathbf \Delta\op_{\leq n ,s}} \StFil_{\cI,\cE}^{\cU} |_{\Delta\op_{\leq n ,s}} \ . \]
	Since $\Delta\op_{\leq n,s}$ is a finite category, \cref{finite_limit_prLRomega} reduces us to show that the transition maps in
	\[ \StFil_{\cI,\cE}^{\cU}|_{\Delta\op_{\leq n ,s}}   \colon  \mathbf \Delta\op_{\leq n ,s}  \to \CAT_{\infty} \]
	are both left and right adjoints and that $\StFil_{\cI,\cE}^{\cU}([m])$ is of finite type for every $m \leqslant n$.
	The first point follows from \cref{StFil_in_prLR}, while the second one follows from \cref{cocart_finite_type} and \cref{categorical_compactness} stating that for every relatively compact open subanalytic subset $U\subset X$, the  stratified space $(U,P)$ is categorically compact.
\end{proof}

\begin{lem}\label{topological_retraction_lemma_bis}
	Let $f \colon (Y,P,\cJ)\to (X,P,\cI)$ be a cartesian finite Galois cover in $\StStrat$ where $(X,P)$ is conically refineable with $\Pi_{\infty}(X)$ compact and where $(Y,P,\cJ)$ is universal.
	Let $Y_{\bullet} \colon \mathbf \Delta_{s}\op \to \Top_{/X}$ be the Cech complex of $f \colon Y \to X$ and put 
	\[ \cI_\bullet \coloneqq \Pi_{\infty}(Y_{\bullet},P) \times_{\Pi_{\infty}(X,P)} \cI \ . \]
	Then $(X,P,\cI)$  is universal and there exists $m\geq 1$ such that 
	$\St_{\cI,\cE}$ is a retract of 
	\[ \lim_{[n] \in \mathbf \Delta_{s,\leq m}} \St_{\cI_n,\cE} \]
	in $\PrLR$.
\end{lem}

\begin{proof}
	The Stokes stratified space $(X,P,\cI)$  is universal in virtue of \cref{topological_retraction_lemma}.
	Since $f \colon Y \to X$ is Galois, $Y_n$ is a finite coproduct of copies of $Y$ over $X$, so that $(Y_n,P,\cI_n)$ is  universal for every $[n]\in \mathbf \Delta_{s}$.
	Since the $Y_n\to X$ is a finite étale cover for every $[n]\in \mathbf \Delta_{s}$, \cite[\cref*{Abstract_Stokes-lem:finite_etale_fibrations_example}]{Abstract_Derived_Stokes} implies that 
	\[ \Pi_{\infty}(Y_n,P)  \to \Pi_{\infty}(X,P) \]
	is a finite étale fibration in the sense of \cite[\cref*{Abstract_Stokes-def:finite_etale_fibrations}]{Abstract_Derived_Stokes}.
	By \cite[\Cref*{Abstract_Stokes-cor:retraction_lemma}]{Abstract_Derived_Stokes}, there is an integer $m\geq 1$ such that there exists a retract
	\[ \St_{\cI,\Sp} \to \lim_{[n] \in \mathbf \Delta_{s,\leq m}} \St_{\cI_n,\Sp}  \to \St_{\cI,\Sp} \ . \]
	in $\PrLR$.	
	\Cref{topological_retraction_lemma_bis} follows from \cref{Peter_lemma} by tensoring  with $\cE$.
\end{proof}

We are now ready for:

\begin{proof}[Proof of \cref{finite_typeness}]
	Since $f$ is strongly proper, we can choose a categorically finite  subanalytic refinement $ R\to Q $  such that $f_\ast(\frakStokes_{\cI,\cE})$ is $R$-hyperconstructible.
	Let $F \colon \Pi_\infty(Y,R) \to \CAT_\infty$ be the functor corresponding to $f_\ast(\frakStokes_{\cI,\cE})$ via the exodromy equivalence \eqref{exodromy_equivalence}.
	By \cref{exodromy_functorialities}, we have
	\[ \St_{\cI,\cE} \simeq f_\ast( \frakStokes_{\cI,\cE} )(Y) \simeq \lim_{\Pi_\infty(Y,R)} F(y) \ . \]
	Recall from \cref{direct_image_factors_trough_PrL} that $f_\ast( \frakStokes_{\cI,\cE} )$ belongs to $\ConsRhyp(Y;\PrLR_k)$, and therefore that $F$ factors through $\PrLR_k$ as well.
	By \cref{Peter_lemma_1}, the above limit can thus equally be computed in $\PrL$.
	Since $(Y,R)$ is categorically finite, \cref{finite_limit_prLRomega} reduces us to check that for each $y \in Y$, $F(y)$ is compactly generated and of finite type relative to $k$.
	By \cref{prop:locally_contractible_strata}, we can choose an open neighborhood $U$ of $y$ such that $y$ is initial in $\Pi_\infty(U,R)$.
	Then
	\[ F(y) \simeq (f_\ast(\frakStokes_{\cI,\cE}))_y \simeq (f_\ast(\frakStokes_{\cI,\cE}))(U) \ , \]
	so compact generation of $F(y)$ follows from \cref{compactly_generated}.
	To check that $F(y)$ is of finite type relative to $k$, we first observe that the base change results Propositions \ref{pullback_Stokes_sheaf} and \ref{proper_analytic_direct_image} allow to reduce to the case where $Y$ is a point and $X$ is compact.
	In that case, there exists  a cartesian finite Galois cover $ (Y,P,\cJ)\to (X,P,\cI)$ such that $(Y,P,\cJ)$ admits a piecewise elementary level structure.	
	Recall that $(X,P)$ is conically refineable in virtue of \cref{eg:subanalytic_implies_combinatorial} and  that $\Pi_{\infty}(X)$ is finite by  \cref{categorical_compactness}.
       Hence, \cref{topological_retraction_lemma_bis} implies the existence of  an integer $m\geq 1$ such that $\St_{\cI,\cE}$ is a retract of 
\begin{equation}\label{eq:finite_typeness}
\lim_{[n] \in \mathbf \Delta_{s,\leq m}} \St_{\cI_n,\cE}
\end{equation}
in $\PrLR$, where $Y_{\bullet} \colon \mathbf \Delta_{s}\op \to \Top_{/X}$ is the Cech complex of $f \colon Y \to X$ and where 
\[
\cI_\bullet \coloneqq \Pi_{\infty}(Y_{\bullet},P) \times_{\Pi_{\infty}(X,P)} \cI   \ .
\]
Hence, it is enough to show that \eqref{eq:finite_typeness} is of finite type relative to $k$.
Since $Y\to X$ is a finite Galois cover, each $Y_n$ is a finite coproduct of copies of $Y$.
       By \cref{finite_limit_prLRomega}, it is thus enough to show that $\St_{\cJ,\cE}$ is of finite type relative to $k$.
       Hence, we can suppose  that $(X,P,\cI)$ admits a piecewise elementary level structure.
	We now argue by induction on the length $d$ of the piecewise elementary level structure of $(X,P,\cI)$.
	When $d = 0$, $\cI = \Pi_\infty(X,P)$ is a fibration in sets, so $(X,P,\cI)$ is (globally) elementary and the conclusion follows from \cref{finite_typeness_piecewise_elementary_case}.
	Otherwise, we can assume the existence of a level morphism $p \colon \cI \to \cJ$ such that:
	\begin{enumerate}\itemsep=0.2cm
		\item $\cJ$ admits a piecewise elementary level structure of length $< d$;
		\item $(X,P,\cI_p)$ is piecewise elementary.
	\end{enumerate}
	Notice that since level morphisms are surjective, the fibers of $\cJ$ are again finite posets, so the inductive hypothesis applies to the Stokes stratified space $(X,P,\cJ)$.
	Consider the pullback square
	\[ \begin{tikzcd}
			\St_{\cI,\cE} \arrow{r}{p_!}\arrow{d}{\Gr_p } &\St_{\cJ,\cE}\arrow{d}{\Gr} \\
			\St_{\cI_p,\cE}\arrow{r}{\pi_!}& \St_{\cJ^{\ens},\cE}
	\end{tikzcd} \]
	supplied by \cref{level_dévissage}.
	Both $\St_{\cI_p, \cE}$ and $\St_{\cJ^{\ens},\cE}$ are of finite type thanks to \cref{finite_typeness_piecewise_elementary_case}, while the inductive hypothesis guarantees that $\St_{\cJ,\cE}$ are of finite type.
	Finally, \cref{stability_lim_colim_ISt} implies that the assumptions of \cref{level_pullback_in_PrLR} are satisfied, so that the above square is a pullback in $\PrLR$.
	Thus, it follows from \cref{finite_limit_prLRomega} that $\St_{\cI,\cE}$ is of finite type.
\end{proof}

\section{Geometricity}\label{geometricity_section}

	We now turn to the main theorem of this paper, namely the construction of a derived Artin stack parametrizing Stokes functors.
	Similarly to Theorems~\ref{thm:St_presentable_stable} and \ref{finite_typeness} we prove this result in the analytic setting and assuming the existence of a locally elementary level structure.
	The geometricity is essentially a consequence of \cref{finite_typeness}, but we need to run more time the level induction to provide an alternative description of the functor of points.

\subsection{Description of the moduli functor}

	We fix an animated commutative ring $k$.
	For every animated commutative $k$-algebra $A$, we let $\Mod_A$ denote the associated stable $\infty$-category of $A$-modules and by $\Perf_A$ the full subcategory of perfect $A$-modules (see e.g.\ \cite[Definition 7.2.4.1]{Lurie_Higher_algebra}).


Fix  a Stokes stratified space $(X,P,\cI)$.

\begin{notation}
	Let $\cE$ be a compactly generated presentable stable $\infty$-category.
	We set
	\[ \St_{\cI,\cE,\omega} \coloneqq \St_{\cI,\cE} \times_{\Fun(\cI,\cE)} \Fun(\cI,\cE^\omega) \ .  \]
	When $\cE = \Mod_A$, we write
	\[ \St_{\cI,A} \coloneqq \St_{\cI,\Mod_A} \qquad \text{and} \qquad \St_{\cI,A,\omega} \coloneqq \St_{\cI,\Mod_A,\omega} \ . \]
\end{notation}

	Let $f \colon \cE \to \cE'$ be a functor of stable presentable $\infty$-categories.
	Via \cref{prop:change_of_coefficients_stokes_IHES} we see that $f$ functorially induces a morphism
	\[ f \colon \St_{\cI,\cE} \to \St_{\cI,\cE'} \ . \]
	When in addition both $\cE$ and $\cE'$ are compactly generated and $f$ preserves compact objects, this further descends to a morphism
	\[ f \colon \St_{\cI,\cE,\omega} \to \St_{\cI,\cE', \omega} \ . \]
	This gives rise to a well defined functor
	\[ \bfSt_{\cI,k}^{\mathrm{cat}} \colon \dAff_k\op \to \CAT_\infty \]
	that sends the spectrum of an animated commutative $k$-algebra $\Spec(A)$ to $\St_{\cI,A,\omega}$.
	Passing to the maximal $\infty$-groupoid, we obtain a presheaf
	\[ \bfSt_{\cI,k} \colon \dAff_k\op \to \Spc \]
	that sends $\Spec(A)$ to
	\[ \bfSt_{\cI,k}(\Spec(A)) \coloneqq ( \St_{\cI,A,\omega} )^\simeq \in \Spc \ . \]
	When $k$ is clear out of the context, we write $\bfSt_\cI$ instead of $\bfSt_{\cI,k}$.

\begin{eg}
	When $\cI$ is the trivial fibration, \cref{Stokes_sheaf_trivial_fibration} shows that $\bfSt_{\cI}$ coincide with the derived stack of perfect local systems.
\end{eg}

	With these notations, we can state the main theorem of this section:

\begin{thm}\label{Representability_via_toen_vaquie}
	Let $f \colon (X,P,\cI) \to (Y,Q)$ be a strongly proper family of Stokes analytic stratified spaces in finite posets locally admitting a ramified vertically piecewise elementary level structure.
	Let $k$ be an animated commutative ring.
	Then, $\bfSt_{\cI}$ is locally geometric locally of finite presentation over $k$.
	Moreover, for every animated commutative $k$-algebra $A$ and every morphism
	\[ x \colon \Spec(A) \to \bfSt_{\cI} \]
	classifying a Stokes functor $F \colon \cI \to \Perf_A$, there is a canonical equivalence
	\[ x^\ast \mathbb T_{\bfSt_{\cI}} \simeq \Hom_{\Fun(\cI,\Mod_A)}( F, F )[1] \ , \]
	where $\mathbb T_{\bfSt_\cI}$ denotes the tangent complex of $\bfSt_\cI$ and the right hand side denotes the $\Mod_A$-enriched $\Hom$ of $\Fun(\cI,\Mod_A)$.
\end{thm}



	We will deduce \cref{Representability_via_toen_vaquie} from \cref{finite_typeness} and of the work of Toën-Vaquié on the moduli of objects of a stable $k$-linear $\infty$-category \cite{Toen_Moduli}.
	To do so, we need a brief digression on the behavior of Stokes functors and the tensor product of presentable $\infty$-categories.

\subsection{Stokes moduli functor as a moduli of objects}

Throughout this section we fix an animated commutative ring $k$.

\begin{recollection}\label{recollection_Toen_Vaquie}
	Let $\cC$ be a compactly generated presentable stable $k$-linear category.
	Its \emph{moduli of objects} is the derived stack
	\[ \cM_{\cC}\colon \dAff_k\op \to  \Spc \]
	given by the rule
	\[ \cM_{\cC}(\Spec(A)) \coloneqq \Fun^{\st}_k((\cC^{\omega})\op, \Perf_A)^\simeq \]
	where 
	$$
	\Fun^{\st}_k((\cC^{\omega})\op, \Perf(A)) \subset \Fun((\cC^{\omega})\op, \Perf(A))
	$$ 
	denotes the full subcategory spanned by exact $k$-linear functors.
	When $\cC$ is of finite type relative to $k$ in the sense of \cref{finite_type_defin},
	 \cite[Theorem 0.2]{Toen_Moduli} states that $\cM_{\cC}$ is a locally geometric derived stack locally of finite presentation.
\end{recollection}

	Let $(X,P,\cI)$ be a bireflexive Stokes stratified space.
	Then the proof of  \cref{compactly_generated} implies that the $\infty$-category $\St_{\cI,k}$ is stable presentable and compactly generated.
	In particular, its moduli of objects is well defined.
	We have:

\begin{prop}\label{identification_prestacks}
	Let $(X,P,\cI)$ be a universal Stokes stratified space.
	Then the derived prestacks $\bfSt_{\cI}$ and $\cM_{\St_{\cI,k}}$ are canonically equivalent.
\end{prop}

\begin{proof}
	Fix a derived affine $\Spec(A) \in \dAff_k$ and consider the following chain of canonical equivalences:
    \begin{align*}
		 \Fun^{\st}_k\big( (\St_{\cI,k})^\omega)\op, \Mod_A \big) & \simeq \Fun^{\st}_k\big( (\St_{\cI,k})^\omega, \Mod_A\op \big) \op &         \\
		 & \simeq \Fun^{\mathrm L}_k( \St_{\cI,k}, \Mod_A\op ) \op     &    \text{By \cite[§3.1]{Antieau_Gepner} } \\
		 & \simeq\Fun^{\mathrm R}_k( \St_{\cI,k}\op , \Mod_A)  &   \\
		& \simeq \St_{\cI,k} \otimes_k \Mod_A &    \text{By \cite[4.8.1.7]{Lurie_Higher_algebra} } \\
		& \simeq \St_{\cI,A}    &  
	\end{align*}
	Let $\LSt_{\cI,\cE} \colon \Fun(\cI,\cE) \to \St_{\cI,\cE}$ be the left adjoint to the canonical inclusion $\St_{\cI,\cE}\hookrightarrow \Fun(\cI,\cE)$.
	By \cref{compactly_generated} a system of compact generators of $ \St_{\cI,\Mod_A} $ is given by $\{ \LSt_{\cI,\cE}(\ev_{a,!}(A)) \}_{a \in \cI}$, where the $\ev_a \colon \{a\}\to \cI$ are the canonical inclusions.
	Then via the embedding
	\[ \Fun^{\st}_k\big( (\St_{\cI,k})^\omega)\op, \Perf(A)\big) \hookrightarrow \Fun^{\st}_k\big( (\St_{\cI,k})^\omega)\op, \Mod_A \big) \]
	induced by $\Perf(A) \hookrightarrow \Mod_A$, the above chain of equivalences exhibits 
	\[ \Fun^{\st}_k\big( (\St_{\cI,k})^\omega)\op, \Perf(A) \big) \]
	as the full-subcategory of $ \St_{\cI,A} $ spanned by Stokes functors $F \colon \cI \to \Mod_A$ such that
	\[ \Hom_{\St_{\cI,A}}( \LSt_{\cI,\cE}(\ev_{a,!}(A)) , F) \in \Perf(A) \]
	for every $a \in \cI$.	
    Hence for every $ F\in \St_{\cI,A} $, we have
    \begin{align*}
	    F \in \cM_{\St_{\cI,k}}(\Spec A) & \Leftrightarrow \Hom_{\St_{\cI,A}}(\LSt_{\cI,\cE}(\ev_{a,!}(A)), F) \in \Perf(A) & \forall a \in \cI \\
		& \Leftrightarrow \Hom_{\St_{\cI,A}}(\ev_{a,!}(A), F) \in \Perf(A) & \forall a \in \cI  \\ 
		& \Leftrightarrow F(a) \in \Perf(A) & \forall a \in \cI \\
		& \Leftrightarrow F \in \bfSt_{\cI,k}(\Spec(A))
    \end{align*}
This concludes the proof of \cref{identification_prestacks}.
\end{proof}

We are now ready for:

\begin{proof}[Proof of \cref{Representability_via_toen_vaquie}]
	By \cref{compactly_generated} and \cref{identification_prestacks}, the prestack $\cM_{\St_{\cI,k}}$ and $\bfSt_{\cI,k}$  are canonically equivalent.
	By \cref{finite_typeness}, the $\infty$-category $\St_{\cI,k}$ is stable presentable and of finite type relative to $k$.
	The conclusion thus follows from \cite[Theorem 0.2]{Toen_Moduli}.
\end{proof}

\subsection{The moduli of Stokes vector bundles}

We fix  an animated commutative ring $k$.
A $k$-point of $\bfSt_{\cI,k}$ is a Stokes functor $F \colon \cI \to \Perf_k$.
In particular, even when $k$ is a field the stack $\bfSt_{\cI,k}$ classifies $\cI$-Stokes structures on perfect complexes, rather than vector bundles.
Thus, when the Stokes stratified space is of dimension $1$, $\bfSt_{\cI,k}$ provides an extension of \cite{BV}.
We are going to see how to extract from $\bfSt_{\cI,k}$ a more classical substack.

\medskip

Let $(X,P,\cI)$ be a Stokes stratified space.
For every animated commutative $k$-algebra $A$, consider the standard $t$-structure $\tau = ((\Mod_A)_{\geqslant 0}, (\Mod_A)_{\leqslant 0})$ on the stable derived $\infty$-category $\Mod_A$.
It is accessible and compatible with filtered colimits, and $\Fun(\cI, \Mod_A)$ inherits an induced $t$-structure defined objectwise and satisfying the same properties.
Besides, $\Fun(\cI,\Mod_A)$ has a canonical $A$-linear structure, with underlying tensor product
\[ (-) \otimes_A (-) \colon \Mod_A \otimes \Fun(\cI, \Mod_A) \to \Fun(\cI, \Mod_A) \ , \]
that sends $(M,F)$ to the functor $M \otimes_A F(-) \colon \cI \to \Mod_A$.
Using \cref{prop:change_of_coefficients_stokes_IHES}, we deduce that if $F$ is a Stokes functor, then the same goes for $M \otimes_A F$.
Following \cite{DPS}, we introduce the following:

\begin{defin}\label{def:Stokes_flat}
	Let $A$ be an animated commutative $k$-algebra and let $F \colon \cI \to \Mod_A$ be a filtered functor.
	We say that $F$ is \emph{flat relative to $A$} (or \emph{$A$-flat}) if for every $M \in \Mod_A^\heartsuit$, the functor $M \otimes_k F \colon \cI \to \Mod_A$ belongs to $\Fun(\cI,\Mod_A)^\heartsuit$.
\end{defin}

\begin{rem}
	Since $\Fun(\cI,\Mod_A)^\heartsuit \simeq \Fun(\cI,\Mod_A^\heartsuit)$, a filtered functor $F$ is $A$-flat if and only if it takes values in flat $A$-modules.
	In particular flatness relative to $A$ is a local property on $X$.
\end{rem}

\begin{eg}\label{eg:flat_over_discrete_ring}
	Assume that $A$ is an discrete commutative algebra.
	If a Stokes functor $F \colon \cI \to \Mod_A$ is flat relative to $A$, then automatically $F \in \St_{\cI,A}^\heartsuit$.
	The vice-versa holds provided that $A$ is a field.
\end{eg}

\begin{rem}
	It can be shown that in the setting of \cref{thm:St_presentable_stable}, $\St_{\cI,\Mod_A}$ inherits a $t$-structure by declaring that a Stokes functor is connective (resp.\ coconnective) if and only if it is so as a filtered functor.
	See \cite[\cref*{Abstract_Stokes-prop:t_structure_Stokes}]{Abstract_Derived_Stokes}.
	Furthermore, the pointwise split condition allows to prove that both induction and relative graduation are $t$-exact at the level of Stokes functors (see \cite[\cref*{Abstract_Stokes-cor:t_exactness_induction_on_Stokes} \& \cref*{Abstract_Stokes-prop:Gr_t_exact}]{Abstract_Derived_Stokes}), whereas these statements fail for filtered or cocartesian functors.
\end{rem}

Sending $\Spec(A) \in \dAff_k\op$ to the full subgroupoid of $\bfSt_{\cI,k}(\Spec(A))$ spanned by flat Stokes functors defines a full sub-prestack $\bfStflat_{\cI,k}$ of $\bfSt_{\cI,k}$.
The goal is to prove the following:

\begin{thm}\label{thm:moduli_Stokes_vector_bundles}
	Let $k$ be an animated commutative ring and let $f \colon (X,P,\cI) \to (Y,Q)$ be a strongly proper family of Stokes analytic stratified spaces in finite posets admitting a ramified vertically piecewise elementary level structure.
	Then the morphism
	\[ \bfStflat_{\cI,k} \to \bfSt_{\cI,k} \]
	is representable by open immersions.
	In particular, $\bfStflat_{\cI,k}$ is a derived $1$-Artin stack locally of finite type.
\end{thm}

	We start discussing some preliminaries.

\begin{lem}\label{lem:Stokes_flat_initial_object}
	Let $(X,P,\cI)$ be a Stokes stratified space and let $A$ be an animated commutative $k$-algebra.
	Assume that $\Pi_\infty(X,P)$ has an initial object $x$.
	Then a Stokes functor $F \colon \cI \to \Mod_A$ is $A$-flat if and only if $j_x^\ast(F)$ is $A$-flat.
\end{lem}

\begin{proof}
	Notice that for every $M \in \Mod_A$, the canonical comparison map
	\[ M \otimes_A j_x^\ast(F) \to j_x^\ast(M \otimes_A F) \]
	is an equivalence.
	Then the lemma follows directly from \cite[\cref*{Abstract_Stokes-cor:t_structures_Stokes_initial_object}]{Abstract_Derived_Stokes}.
\end{proof}

\begin{notation}\label{notation:stacky_evaluations}
	Let $(X,P,\cI)$ be a Stokes stratified space.
	For every $a \in \cI$, the functor $\ev_a \colon \{a\} \to \cI$ induces a morphism of derived prestacks
	\[ \mathbf{ev}_a \colon \bfSt_{\cI,k} \to \bfPerf_k \ . \]
\end{notation}

\begin{prop}\label{prop:Stokes_vector_bundles_compact_case}
	Let $(X,P,\cI)$ be a compact Stokes analytic stratified space.
	Then 
	\[ \bfStflat_{\cI,k} \to \bfSt_{\cI,k} \]
	is representable by an open immersion.
\end{prop}

\begin{proof}
	By \cref{prop:locally_contractible_strata} and since $X$ is compact we can find an open cover of $X$ by finitely many open subsets $U_1, U_2, \ldots, U_n$ such that each $\Pi_\infty(U_i,P)$ has an initial object $x_i$.
	Let
	\[ e \colon \bfSt_{\cI,k} \to \prod_{i = 1}^n \prod_{a \in \cI_{x_i}} \bfPerf_k \]
	be the product of the evaluation maps of \cref{notation:stacky_evaluations}.
	Notice that both products are finite, so the map
	\[ \prod_{i = 1}^n \prod_{a \in \cI_{x_i}} \rB\mathrm{GL} \to \prod_{i = 1}^n \prod_{a \in \cI_{x_i}} \bfPerf_k \]
	is representable by open an immersion (see e.g.\ \cite[Proposition 6.1.4.5]{Lurie_SAG}).
	Besides, \cref{lem:Stokes_flat_initial_object} implies that the square
	\[ \begin{tikzcd}
		\bfStflat_{\cI,k} \arrow{r} \arrow{d} & \bfSt_{\cI,k} \arrow{d} \\
		\prod_{i = 1}^n \prod_{a \in \cI_{x_i}} \rB\mathrm{GL} \arrow{r} & \prod_{i = 1}^n \prod_{a \in \cI_{x_i}} \bfPerf_k
	\end{tikzcd} \]
	is a fiber product.
	The conclusion follows.
\end{proof}

	This proves \cref{thm:moduli_Stokes_vector_bundles} when the base is reduced to a single point.
	To prove the general case, we need a couple of extra preliminaries.

\begin{lem}\label{lem:Stcoh_1_truncated}
	Let $(X,P,\cI)$ be a Stokes stratified space.
	Then the derived prestack $\bfStflat_{\cI,k}$ is $1$-truncated.
\end{lem}

\begin{proof}
	We have to prove that for a discrete commutative $k$-algebra $A$, $\bfStflat_{\cI,k}(\Spec(A))$ is a $1$-groupoid.
	Since $\St_{\cI,A}$ is fully faithful inside $\Fun(\cI,A)$, using \cite[Proposition 2.3.4.18]{HTT} we see that it is enough to show that for every pair of $A$-flat Stokes functors $F, G \colon \cI \to \Mod_A$, the mapping space $\Map_{\Fun(\cI,\Mod_A)}(F,G)$ is discrete.
	Since $A$ is discrete, both $F$ and $G$ belongs to $\St_{\cI,A}^\heartsuit$ by \cref{eg:flat_over_discrete_ring}.
	Thus, \cite[\cref*{Abstract_Stokes-cor:t_structure_Stokes_heart}]{Abstract_Derived_Stokes} implies that both $F$ and $G$ take values in the $1$-category $\Mod_A^\heartsuit$.
	Then the conclusion follows from \cite[Corollary 2.3.4.8]{HTT}.
\end{proof}

%

\begin{lem}\label{lem:open_immersions_closed_under_compact_limits}
	Let $I$ be a finite $\infty$-category and let
	\[ f_\bullet \colon F_\bullet \to G_\bullet \]
	be a natural transformation between diagrams $I \to \mathrm{dSt}_k$.
	Let
	\[ F \coloneqq \lim_{i \in I} F_i \qquad \text{and} \qquad G \coloneqq \lim_{i \in I} G_i \]
	be the limits computed in $\mathrm{dSt}_k$.
	Assume that:
	\begin{enumerate}\itemsep=0.2cm
		\item for every $i \in I$, $F_i$ is geometric and locally of finite type and $G_i$ is locally geometric and locally of finite type;
		\item for every $i \in I$, $f_i \colon F_i \to G_i$ is representable by open immersions;
		\item $G$ is locally geometric and locally of finite presentation.
	\end{enumerate}
	Then $F$ is a geometric derived stack and the induced morphism $f \colon F \to G$ is an open immersion.
\end{lem}

\begin{proof}
	It follows from \cite[Proposition 1.3.3.3 and Lemma 1.4.1.12]{HAG-II} that geometric stacks locally of finite type are closed under finite limits.
	Thus, $F$ is geometric and locally of finite type.
	We are left to check that $f$ is an open immersion.
	Since both $F$ and $G$ are locally geometric and locally of finite type, it follows that $f$ is an open immersion if and only if it is \'etale and the diagonal
	\[ \delta_f \colon F \to F \times_G F \]
	is an equivalence.
	Besides, since $f$ is automatically locally of finite presentation, \cite[Corollary 2.2.5.6]{HAG-II} shows that $f$ is \'etale if and only if it is formally \'etale, i.e.\ the relative cotangent complex $\mathbb L_f$ vanishes.
	Since limits commutes with limits, we see that $\delta_f$ is the limit of the diagonal maps
	\[ \delta_{f_i} \colon F_i \to F_i \times_{G_i} F_i \ , \]
	and since each $f_i$ is an open immersion, it automatically follows that each $\delta_{f_i}$ is an equivalence.
	Therefore, the same goes for $f$.
	Similarly, the property of being formally \'etale is clearly closed under retracts.
	On the other hand, \cite[Lemma 1.4.1.12]{HAG-II} implies that formally \'etale maps are closed under pullbacks and hence under finite limits.
	The conclusion follows.
\end{proof}

We are now ready for:

\begin{proof}[Proof of \cref{thm:moduli_Stokes_vector_bundles}]
	Since $f$ is strongly proper, we can choose a categorically finite subanalytic refinement $ R\to Q $  such that $f_\ast(\frakStokes_{\cI,k})$ is $R$-hyperconstructible.
	Let $F \colon \Pi_\infty(Y,R) \to \CAT_\infty$ be the functor corresponding to $f_\ast(\frakStokes_{\cI,k})$ via the exodromy equivalence.
	As we argued in \cref{finite_typeness}, we obtain a canonical equivalence
	\[ \St_{\cI, k} \simeq \lim_{y \in \Pi_\infty(Y,R)} F_y \ , \]
	the limit being computed in $\PrLR$.
	Besides, the base-change results of Propositions \ref{pullback_Stokes_sheaf} and \ref{proper_analytic_direct_image} and \cref{elementarity_pullback} provide a canonical identification $F_y \simeq \St_{\cI_y, k}$.
	Passing to the moduli of objects and applying \cref{identification_prestacks}, we deduce that
	\[ \bfSt_{\cI,k} \simeq \lim_{y \in \Pi_\infty(Y,R)} \bfSt_{\cI_y,k} \ . \]
	Using \cite[\cref*{Abstract_Stokes-prop:t_structure_Stokes}]{Abstract_Derived_Stokes}, we deduce from here that the induced morphism
	\[ \bfStflat_{\cI,k} \to \lim_{y \in \Pi_\infty(Y,R)} \bfStflat_{\cI_y,k} \]
	is an equivalence as well.
	Besides, $\bfSt_{\cI,k}^{\mathrm{flat}}$ and $\bfSt_{\cI_y,k}^{\mathrm{flat}}$ are $1$-truncated for every $y \in \Pi_\infty(Y,R)$ thanks to \cref{lem:Stcoh_1_truncated}.
	Thus, \cref{lem:open_immersions_closed_under_compact_limits} reduces us to the case where $Y$ is reduced to a single point.
	Since in this case $X$ is compact, the conclusion follows from \cref{prop:Stokes_vector_bundles_compact_case}.
\end{proof}

\section{Elementarity and polyhedral Stokes stratified spaces}\label{sec:polyhedral}

	The goal of this section is to prove an elementarity criterion for a specific class of Stokes stratified spaces that we now introduce.

\subsection{Polyhedral Stokes stratified spaces}

\begin{recollection}
	For $n\geq 0$, recall that a \textit{polyhedron} of  $\mathbb{R}^n$ is a non empty subset obtained as the intersection of a finite number of closed half spaces.
\end{recollection}
	In what follows, $\{-,0,+\}$ will denote the span poset where $0$ is declared to be the initial object. 
	Let $n\geq 0$.
	For a non zero affine form $\varphi \colon  \mathbb{R}^n\to \mathbb{R}$, we denote by $H_{\varphi}$ the zero locus of $\varphi$.

\begin{defin}\label{polyhedral_stratified_space}
	Let $n\geq 0$ and let $C\subset \mathbb{R}^n$ be a polyhedron.
	Let $\Phi$ be a finite set of non zero affine forms on  $\mathbb{R}^n$.
	Let $(\mathbb{R}^n,\Phi)$ be the stratified space given by the continuous function $C\to \{-,0,+\}^{\Phi}$ sending $x\in C$ to the function sending $\varphi$  to the sign of $\varphi(x)$ if $x\notin  H_{\varphi}$, and to 0 otherwise.
\end{defin}

\begin{rem}\label{homotopy_type_poset}
	The stratified space $(C,\Phi)$ is conical and the induced functor $\Pi_{\infty}(C,\Phi )\to \{-,0,+\}^{\Phi}$ is an equivalence of $\infty$-categories.
\end{rem}

\begin{defin}
	A \textit{polyhedral Stokes stratified space} is a Stokes stratified space in finite posets of the form $(C,\Phi,\cI)$ where $(C,\Phi)$ is as in  \cref{polyhedral_stratified_space} and  such that  $\cI^{\ens}\to \Pi_{\infty}(C,\Phi)$ is locally constant (\cref{def:locally_constant}).
\end{defin}

\begin{rem}
\label{elementarity_constraint}
	If $(C,\Phi,\cI)$ is an elementary polyhedral Stokes stratified space of $\mathbb{R}$, one can show that the Stokes locus of every distinct $a,b\in \mathscr{I}(C)$ is reduced to a point.
	In particular, polyhedral Stokes stratified spaces are rarely elementary.
\end{rem}

\subsection{Elementarity criterion}

	The main result of this section is the following theorem whose statement is inspired from  \cite[Proposition 3.16]{MochStokes}.

\begin{thm}\label{induction_for_adm_Stokes_triple}
	Let $(C,\Phi,\cI)$ be a polyhedral Stokes stratified space.
	Suppose that for every distinct $a,b\in \mathscr{I}(C)$, there exists $\varphi\in  \Phi$ such that 
\begin{enumerate}\itemsep=0.2cm
\item The Stokes locus of $\{a,b\}$ is $C\cap H_{\varphi}$ (\cref{Stokes_locus}).

\item $C\setminus H_{\varphi}$  admits exactly two connected components $C_1$ and $C_2$.

\item $a <_{x} b$ for every $x\in C_1$ and $b <_{x} a$ for every $x\in C_2$.
\end{enumerate}
Then $(C,\Phi,\cI)$ is elementary (\cref{defin_elementary}).
\end{thm}

\begin{rem}\label{total_order_open_stratum}
	In the setting of \cref{induction_for_adm_Stokes_triple}, the order of $\cI_x$ is total for every $x$ lying in an open stratum of $(C,\Phi)$.
	
\end{rem}

\begin{rem}

	Fully-faithfulness in \cref{induction_for_adm_Stokes_triple} will not require any extra technology that the one developed so far and will be proved in \cref{polyhedral_fully_faithful}.
	On the other hand, essential surjectivity will require more work and will ultimately be proved  in  \cref{polyhedral_essentially_surjective}.

\end{rem}

\cref{induction_for_adm_Stokes_triple} will be used via the following:

\begin{thm}\label{cor_induction_for_adm_Stokes_triple}

	Let $(C,P,\cI)$ be a Stokes analytic stratified space in finite posets where $C\subset \mathbb{R}^n$ is a polyhedron and $\cI^{\ens}\to \Pi_{\infty}(C,P)$ is locally constant.
	Assume that for every distinct $a,b\in \mathscr{I}(C)$, there exists a non zero affine form $\varphi \colon  \mathbb{R}^n \to \mathbb{R}$ such that 
\begin{enumerate}\itemsep=0.2cm
\item The Stokes locus of $\{a,b\}$ is $C\cap H_{\varphi}$.

\item $C\setminus H_{\varphi}$  admits exactly two connected components $C_1$ and $C_2$.

\item $a <_{x} b$ for every $x\in C_1$ and $b <_{x} a$ for every $x\in C_2$.
\end{enumerate}
Then $(C,P,\cI)$ is elementary.

\end{thm}

\begin{proof}
	Let $\Phi$ be a finite set of non zero affine forms such that for every distinct $a,b\in \mathscr{I}(C)$, there is  $\varphi\in \Phi$ satisfying (1),(2),(3) for $a,b$.
	By \cref{elementarity_pullback}, the conclusion of \cref{cor_induction_for_adm_Stokes_triple} is insensitive to subanalytic refinements.
	Hence, at the cost of refining $(C,P)$, we can suppose that there exists a refinement $(C,P)\to (C,\Phi)$.
	By \cref{refinement_localization}, the induced functor 
\[
\Pi_{\infty}(C,P)\to \Pi_{\infty}(C,\Phi)
\]  
	exhibits $ \Pi_{\infty}(C,\Phi)$ 	as the  localization of $\Pi_{\infty}(C,P)$ at the set of arrows sent to equivalences by $\Pi_{\infty}(C,P)\to P\to \Phi$.
	On the other hands, conditions (1) and (3) say that for every morphism $\gamma \colon  x\to y$ in  $\Pi_{\infty}(C,P)$ sent to an equivalence by $\Pi_{\infty}(C,P)\to P\to \Phi$, the induced morphism of posets $\cI_x \to \cI_y$ is an isomorphism.
	Hence, there is a cocartesian fibration in finite posets $\cJ \to \Pi_{\infty}(C,\Phi)$ and a cartesian morphism 
\[
(C,P,\cI)\to (C,\Phi,\cJ)    \ .
\]
	By \cref{elementarity_pullback}, we thus have to show that $(C,\Phi,\cJ)$ is elementary, which follows from \cref{induction_for_adm_Stokes_triple}.
\end{proof}

\subsection{Distance on the set of open strata}

\begin{defin}
	Let $(C,\Phi)$ be a stratified polyhedron.
	For $A, B\subset C$, we say that \textit{$A$ and $B$ are separated by $\varphi \in \Phi$} if they lie in distinct connected components of $C\setminus H_{\varphi}$.
	We let $\Phi(A,B)\subset \Phi$ be the set of  forms  separating $A$ and $B$ and denote by $d(A,B)$ its cardinality.
\end{defin}

\begin{rem}
	If $U,V,W$ are open strata of $(C,\Phi)$, then 
\[
\Phi(U,V)\subset \Phi(U,W) \cup \Phi(W,V) \ .
\]
	In particular, $d$ induces a distance on the set of open strata of  $(C,\Phi)$. 
\end{rem}

\begin{lem}\label{adjacent_faces}
	Let $(C,\Phi)$ be a stratified polyhedron and let  $U,V,W$ be open strata of  $(C,\Phi)$. 
	Suppose that $V$ and $W$ are distinct and adjacent along a face lying in $H_{\varphi}$ for some $\varphi\in \Phi$. 
	Then $\Phi(U,V)$ and $\Phi(U,W)$ differ exactly by $\varphi$.
\end{lem}
\begin{proof}
	Let $\psi \in \Phi(U,V)$.
	If $\psi$ does not appear in $\Phi(U,W)$, then $\psi$ separates $V$ and $W$.
	Hence, $\overline{V}\cap \overline{W}\subset H_\varphi\cap H_{\psi}$.
	Since $V$ and $W$ are assumed to be adjacent, $\overline{V}\cap \overline{W}$ has codimension 1.
	Hence, so does $H_\varphi\cap H_{\psi}$.
	Thus $\psi =\varphi$.
\end{proof}

\begin{defin}\label{U_smaller_than_k}
	Let $(C,\Phi)$ be a stratified polyhedron and let  $U$ be an open stratum. 
	For $k=-1$, put  $U_{\leq -1}=U$.
	For $k\geq 0$, put
\[
U_{\leq k} \coloneqq \bigcup_{V, d(U,V)\leq k}  \overline{V}
\]
where the union runs over the open strata  $V$ of $(C,\Phi)$ satisfying $d(U,V)\leq k$. 
\end{defin}

\begin{rem}\label{Uk_fully_faithful}
	Let $V$ be an open stratum of $(C,\Phi)$ mapping to $f \in \{-,+\}^{\Phi}$.
	Then, $\overline{V}$ is the set of points of $C$ lying above the closed subset $S(V)\coloneqq (\{-,0,+\}^{\Phi})_{\leq f}$.
	In particular $U_{\leq k}$ is the set of points of $C$ lying above the closed subset 
\[
S(U,k) \coloneqq \bigcup_{V, d(U,V)\leq k}  S(V) \ .
\]
\end{rem}

\begin{lem}\label{A_and_B}
	Let $(C,\Phi)$ be a stratified polyhedron and let  $U,V$ be distinct open strata. 
	Put $k\coloneqq d(U, V)-1$. 
	Let $F$ be a face of $\overline{V}$. 
	Let $\varphi \in \Phi$ be the unique form such that $F=\overline{V}\cap H_{\varphi}$. 
	Then, $F\subset U_{\leq k}$ if and only if $\varphi$ separates $U$ and $V$.
	In particular, 
\[
\overline{V}\cap U_{\leq k}=\bigcup_{\varphi\in \Phi(U,V)} \overline{V}\cap H_{\varphi} \ .
\]
\end{lem}
\begin{proof}
	Suppose that $\varphi$ separates $U$ and $V$.
	Hence, there is an open stratum $W\neq V$ adjacent to $V$ along $F$. 
	From \cref{adjacent_faces},  we have $d(U, W)=k$. 
	Hence, $F\subset \overline{W}\subset  U_{\leq k}$. 
	On the other hand, suppose that $F\subset U_{\leq k}$. 
	By definition, there is an open stratum $W$ with $d(U,W)\leq k$ such that  $F$ is a face of $\overline{W}$.
	In particular, $W\neq V$.
	Thus, \cref{adjacent_faces} ensures that $\Phi(U,V)$ and $\Phi(U,W)$ differ exactly by $\varphi$.
	Since $d(U,V)>d(U,W)$, we necessarily have $\varphi \in \Phi(U,V)$  and \cref{A_and_B} is proved.
\end{proof}

\begin{lem}\label{contractile}
	Let $(C,\Phi)$ be a stratified polyhedron and let  $U,V$ be distinct open strata. 
	Put $k\coloneqq d(U, V)-1$. 
	Then, $\overline{V}\cap U_{\leq k } \to \overline{V}$ admits a deformation retract. 
	In particular, $\overline{V}\cap U_{\leq k }$ is contractible.
\end{lem}
\begin{proof}
	Fix $x\in U$.
	At the cost of replacing some forms in $\Phi$ by their opposite, we can suppose that $\overline{V}$ is the set of points $x\in C$ such that $\varphi(x)\geq 0$ for every $\varphi \in \Phi$.
	For $y\in \overline{V}$, define the following degree $k+1$ polynomial
$$
P_V(y) \colon  t \mapsto \prod_{\varphi\in \Phi(U,V)}  \varphi((1-t)\cdot x+t\cdot y)   \ .
$$
	Then, $P_V(y)$ has exactly $k+1$ roots in $(0,1]$ counted with multiplicities.
	Let $t_V(y)\in (0,1]$ be the biggest root of $P_V(y)$ and put
$$
p_V(y)\coloneqq (1-t_V(y))\cdot x+ t_V(y)\cdot y   \ .
$$ 
	Since the coefficients of $P_V(y)$ depend continuously on $y$, so does $t_V(y)$.
	Hence, $p_V $ varies continuously in $y$.
	Let $y \in \overline{V}$.
	We want to show that $[y,p_V(y)]\subset \overline{V}$.
	If $y=p_V(y)$, there is nothing to prove.
	Suppose that  $y\neq p_V(y)$ and pick $z\in (y,p_V(y))$.
	If $\varphi$ separates $U$ and $V$, the non zero real numbers $\varphi(y)$ and $\varphi(z)$ have the same sign by construction. 
	Hence $\varphi(z)> 0$.
	If $\varphi$ does not separate $U$ and $V$, we have $\varphi(x)>0$.
	Since $\varphi(y)\geq 0$, we deduce $\varphi(z)\geq 0$. 
	Hence, $(y,p_V(y))\subset \overline{V}$, so that $[y,p_V(y)]\subset \overline{V}$.
	By \cref{A_and_B}, we deduce that $p_V(y)\in  \overline{V} \cap U_{\leq k }$.
	Note that if $y\in \overline{V}\cap U_{\leq k }$, then $y$ lies on a face of $\overline{V}$ separating $U$ and $V$ by 
\cref{A_and_B}.
	Hence, $P_V(y)$ vanishes at $t=1$, so that $p_V(y)=y$.
	Thus, the continuous function $ [0,1]\times \overline{V}  \to \overline{V}  $ 
defined as 
$$
(u,y)\mapsto u\cdot p_V(y)+(1-u) \cdot y
$$
provides the sought-after deformation retract. 
\end{proof}

\begin{construction}\label{deleting_strata}
	Let $(C,\Phi)$ be a stratified polyhedron and let  $U$ be an open strata. 
	Let $k\geq 0$ and put $S(U,k+1)^\circ \coloneqq  S(U,k+1)\setminus S(U,k)$.
	Observe that $S(U,k+1)^\circ $ is open in $S(U,k+1)$.
	Consider the following pushout of posets
$$
	\begin{tikzcd}
	S(U,k+1)^\circ\arrow{r}\arrow{d}&S(U,k+1)\arrow{d}\\
	\ast\arrow{r}& P(U,k+1)   \ .
			\end{tikzcd}
$$
	Since $S(U,k+1)^\circ $ is open in $S(U,k+1)$, the stratified space $(U_{k+1},P(U,k+1))$ is conically stratified and admits  $U_{\leq k+1}\setminus U_{\leq k}$ as  open stratum.
\end{construction}

\begin{lem}\label{finality_closed_cell}
	Let $(C,\Phi)$ be a stratified polyhedron and let  $U$ be an open stratum. 
	Let $k\geq 0$.
	Then, the induced functor
\begin{equation}\label{finality_closed_cell_functor}
\Pi_{\infty}(U_{\leq k},S(U,k)) \to \Pi_{\infty}(U_{\leq k+1},P(U,k+1))
\end{equation}
is final.
\end{lem}
\begin{proof}
	To prove \cref{finality_closed_cell}, it is enough to prove that for $x\in 
U_{\leq k+1}$, the $\infty$-category
$$
\cX\coloneqq \Pi_{\infty}(U_{\leq k},S(U,k)) \times_{\Pi_{\infty}(U_{\leq k+1},P(U,k+1))} \Pi_{\infty}(U_{\leq k+1},P(U,k+1))_{/x}
$$ 
is weakly contractible.
	Since $S(U,k)$ is closed in $P(U,k+1)$, the functor \eqref{finality_closed_cell_functor} is fully-faithful.
	Hence, we can suppose that $x\in U_{\leq k+1} \setminus U_{\leq k}$.
	In that case, let $V\subset U_{\leq k+1}$ be an open stratum  at distance $k+1$ from $U$ such that $x\in \overline{V}$.
	By \cref{homotopy_type_poset}, the $\infty$-category $\cX$ is equivalent to the full subcategory of $\Pi_{\infty}(U_{\leq k},S(U,k))$ spanned by points $y$ at the source of some exit-path $\gamma \colon  y\to x$ in $\Pi_{\infty}(U_{\leq k+1},P(U,k+1))$.
	In particular $\gamma((0,1])\subset U_{\leq k+1} \setminus U_{\leq k}$.
	Note that $\gamma((0,1])\subset \overline{V}$.
	Indeed if this was not the case, there would exist an open stratum $W\neq V$ adjacent to $V$ with $d(U,W)=k+1$.
	This is impossible by \cref{adjacent_faces}.
	Hence $y\in \overline{V}\cap U_{\leq k}$.
	On the other hand, for $y\in \overline{V}\cap U_{\leq k}$, the line joining $y$ to $x$ is a morphism in $\Pi_{\infty}(U_{\leq k+1},P(U,k+1))$.
	Hence, $\cX$ is equivalent to the full subcategory of 
$\Pi_{\infty}(U_{\leq k},S(U,k))$ spanned by points $y\in \overline{V}\cap U_{\leq k}$, that is 
$$
\cX\simeq \Pi_{\infty}(\overline{V}\cap U_{\leq k}, S(V)\cap S(U,k)) \ .
$$
	Hence, 
$$
\Env(\cX)\simeq \Pi_{\infty}(\overline{V}\cap U_{\leq k})\simeq \ast
$$
where the last equivalence follows from \cref{contractile}.
\end{proof}

\subsection{Splitting propagation}

\begin{defin}\label{the_class_W(U)}
	Let $(C,\Phi)$ be a stratified polyhedron and let  $U$ be an open stratum. 
	Let $W(U)$ be the class of morphisms $\gamma \colon  x\rightarrow y$ in $\Pi_{\infty}(C,\Phi)$ such that for every $\varphi \in \Phi$  with  $x\in H_{\varphi}$, one of the following condition is satisfied:
\begin{enumerate}\itemsep=0.2cm
\item[$(i)$] We have $y\in H_\varphi$.
\item[$(ii)$] The point $y$ and  $U$ are not separated by $H_{\varphi}$. 
\end{enumerate}
	In particular, $W(U)$ contains every equivalence of $\Pi_{\infty}(C,\Phi)$.
\end{defin}

	Here are some examples of arrows in the class $W(U)$.

\begin{lem}\label{belong_to_S}
	Let $(C,\Phi)$ be a stratified polyhedron and let  $U$ be an open stratum. 
	Let $k\geq 0$. 
	Then, every exit path of $(U_{\leq k+1}\setminus U_{\leq k}, \Phi)$ lies in $W(U)$.
\end{lem}
\begin{proof}
	Let $\gamma \colon  x\rightarrow y$ be an exit path of $(U_{\leq k+1}\setminus U_{\leq k}, \Phi)$.
	Let $V$ be a stratum at distance $k+1$ from $U$ with  $x\in \overline{V}$.
	Let $\varphi \in \Phi$ with $x\in H_{\varphi}$ and assume that $y\notin H_{\varphi}$.
	Since $x\notin U_{\leq k}$, \cref{A_and_B} ensures that $\varphi$ does not separate $U$ and $V$.
	Since  $\gamma \colon  x\rightarrow y$ lies in $U_{\leq k+1}$ we deduce that $\varphi$ does not separate $y$ and $U$.
\end{proof}

	The class of maps from \cref{the_class_W(U)} is useful because of the following 

\begin{lem}\label{use_of_W(U)}
	Let $(C,\Phi)$ be a stratified polyhedron and let  $U$ be an open stratum. 
	Let $F \colon  \Pi_{\infty}(C,\Phi) \to  \cE$ be a functor inverting every arrow in $W(U)$.
	Then, the canonical morphism
\[
\lim_{\Pi_{\infty}(C,\Phi)} F \to \lim_{\Pi_{\infty}(U,\Phi)} F|_U
\]
is an equivalence.
\end{lem}

\begin{proof}
	To prove  \cref{use_of_W(U)},  it is enough to prove that 
\begin{equation}\label{eq_image_of_S_equivalence_then_splitting}
\lim_{\Pi_{\infty}(U_{\leq k},S(U,k)) } F|_{U_{\leq k}}  \to \lim_{\Pi_{\infty}(U_{\leq k-1},S(U,k-1)) } F|_{U_{\leq k-1}}  
\end{equation}
is an equivalence for every $k\geq 0$,  where we used the notations of \cref{deleting_strata}.
	Assume that $k\geq 1$.
Since 
\[
(U_{\leq k},S(U,k))\to (U_{\leq k},P(U,k))
\] 
is a refinement, we know by 
\cref{refinement_localization} that the functor
\begin{equation}\label{use_of_W(U)_localisation}
\Pi_{\infty}(U_{\leq k},S(U,k)) \to \Pi_{\infty}(U_{\leq k},P(U,k))
\end{equation}
 exhibits the target as the localization of the source at the exit paths in $U_{\leq k}\setminus U_{\leq k-1}$.
	By \cref{belong_to_S},  the functor  \eqref{use_of_W(U)_localisation} is thus a localization functor at some arrows in $W(U)$.
	Hence, the functor 
\[
F|_{U_{\leq k}} \colon  \Pi_{\infty}(U_{\leq k},S(U,k))\to \cE\]  
factors uniquely through $\Pi_{\infty}(U_{\leq k},P(U,k)) $.
	Since a localization functor is final, to prove that \eqref{eq_image_of_S_equivalence_then_splitting} is an equivalence thus amounts to prove that the functor 
\[
\Pi_{\infty}(U_{\leq k-1},S(U,k-1)) \to \Pi_{\infty}(U_{\leq k},P(U,k))
\]
is final, which follows from \cref{finality_closed_cell}.
	The case  $k=0$ is treated similarly.
\end{proof}

	The lemma below provides examples of functors where \cref{use_of_W(U)} applies.
	Before this, let us recall the following 
	
	\begin{lem}[{\cite[\cref*{Abstract_Stokes-Stokes_when_locally_constant}]{Abstract_Derived_Stokes}}]\label{Stokes_when_locally_constant_IHES}
	Let $(X,P,\cI)$ be a Stokes stratified space such that $\cI\to \Pi_{\infty}(X,P)$ is locally constant (\cref{def:locally_constant}).
	Let $F \colon \cI \to \cE$ be a cocartesian functor.
	Let $\sigma \colon \Pi_{\infty}(X,P) \to \cI$ be a cocartesian section.
	Then, $\sigma^*(F) \colon \Pi_{\infty}(X,P)\to \cE$ inverts every arrow of $\Pi_{\infty}(X,P)$.
\end{lem}

\begin{lem}\label{image_of_S_equivalence}
	Let $(C,\Phi,\cI)$ be a polyhedral Stokes stratified space satisfying the conditions of \cref{induction_for_adm_Stokes_triple}.
	Let $U$ be an open stratum.
	Let $a\in \mathscr{I}(C)$ minimal  on $U$.
	Let $F \colon \cI \to \cE$ be a Stokes functor.
	Then $F_{<a}$ and $F_a$ invert arrows in $W(U)$.
\end{lem}
\begin{proof}
	Consider the fibre sequence
\[
  F_{<a}    \to     F_{a}    \to    \Gr_a F    \ .
\]
	By \cref{Stokes_when_locally_constant_IHES}, the functor  $\Gr_a F$ inverts every arrow of $\Pi_{\infty}(C,\Phi)$.
Hence, we are left to show that $F_a$ invert arrows in $W(U)$.
	Let $\gamma \in W(U)$.
	At the cost of writing $\gamma$ as the composition of a smaller path followed by an equivalence, we can suppose that $\gamma$ lies in an open subset $V$ such that $x$ is  initial  in $\Pi_{\infty}(V,\Phi)$.
	From \cref{eg:Stokes_structures_at_a_point}, we have  $F|_V=i_{\cI ! }(V)$ where $V \colon \cI^{\ens} \to \cE$.
	Then, $F_a(\gamma)$ reads as 
\[
 \bigoplus_{\substack{b\in \mathscr{I}(C) \\ b\leq_x a}} V_b \to \bigoplus_{\substack{b\in \mathscr{I}(C) \\ b\leq_y a}}  V_b   \ .
\]
	Let $b\in \mathscr{I}(C)$ with $b \neq a$.
	To prove \cref{image_of_S_equivalence}, we are left to show that $b <_x a$ if and only if $b<_y a$.
	The direct implication is obvious.
	We thus suppose that $b<_y a$.
	Let $\varphi \in \Phi$ such that the Stokes locus of $\{a,b\}$ is  $C\cap H_{\varphi}$.
	Since $a$ is minimal on $U$, the assumption $(1)$ from \cref{induction_for_adm_Stokes_triple} implies that $\varphi$ separates $y$ and $U$.
	If $x\in H_{\varphi}$, then the definition of $W(U)$ yields $y \in H_{\varphi}$, which contradicts $b<_y a$.
	Hence, $x\notin H_{\varphi}$. 
	In particular, $b <_x a$ or $a <_x b$.
	Note that the inequality $a <_x b$ contradicts  $b <_y a$.
	Hence, $b<_x a$ and the proof of \cref{image_of_S_equivalence} is complete.
 \end{proof}

\begin{lem}\label{image_of_S_equivalence_then_splitting}
Let $(C,\Phi,\cI)$ be a polyhedral Stokes stratified space satisfying the conditions of \cref{induction_for_adm_Stokes_triple}.
Let $U$ be an open stratum.
Let $a\in \mathscr{I}(C)$ minimal element on $U$.
Let $F \colon \cI \to \cE$ be a Stokes functor.
Then, the fiber sequence
\begin{equation}\label{split_on_big_open}
 F_{<a} \to   F_{a}  \to   \Gr_a  F
\end{equation}
admits a splitting.
\end{lem}
\begin{proof}
Since $a$ is minimal on $U$, the restriction of $F_{<a}$ to $U$ is the zero functor.
Hence, \eqref{split_on_big_open} admits a canonical splitting on $U$.
By \cref{Gr_of_Stokes_IHES}, the functor $\Gr  F \colon \cI^{\ens}\to \cE$ is cocartesian.
By \cref{Stokes_when_locally_constant_IHES}, we deduce that $\Gr_a  F \colon  \Pi_{\infty}(C,\Phi) \to \cE$ inverts every arrows.
Since 
\[
\Env(\Pi_{\infty}(C,\Phi) )\simeq \Pi_{\infty}(C)\simeq \ast \ , 
\]
we deduce that $\Gr_a  F \colon  \Pi_{\infty}(C,\Phi) \to \cE$  is a constant functor. 
Hence, it is enough to show that 
$$
\Map(\Gr_a  F, F_a) \to \Map(\Gr_a  F|_{U},  F_{a}|_{U})
$$
is an equivalence
This amounts to show that 
\[
\lim_{\Pi_{\infty}(C,\Phi)} F_a \to \lim_{\Pi_{\infty}(U,\Phi)} F_a|_U
\]
is an equivalence.
By \cref{use_of_W(U)},  we are thus left to show that $F_a$ inverts every arrow in $W(U)$.
This in turn holds by \cref{image_of_S_equivalence}.
\end{proof}

\subsection{Proof of \cref{induction_for_adm_Stokes_triple}}

The proof will be the consequence of the following

\begin{prop}\label{polyhedral_fully_faithful}
Let $(C,\Phi,\cI)$ be a polyhedral Stokes stratified space satisfying the conditions of \cref{induction_for_adm_Stokes_triple}.
Then, the induction functor 
\[
i_{\cI, !} \colon \St_{\cI^{\ens},\cE} \to \St_{\cI,\cE}
\]
is fully faithful.
\end{prop}

\begin{proof}
Let $V,W \colon  \cI^{\ens} \to \cE$ be Stokes functors and let us show that 
\[
\Map(V,W)\to  \Map(i_{\cI, !}(V),i_{\cI, !}(W))\simeq  \Map(V,i_{\cI}^* i_{\cI, !}(W))
\]
is an equivalence. 
This is equivalent to show that for every $a\in \mathscr{I}(C)$, the map 
\[
\Map(V_a,W_a)\to  \Map(V_a,(i_{\cI, !}(W))_a)
\]
is an equivalence. 
By \cref{Stokes_when_locally_constant_IHES}, the cocartesian functor $V_a \colon  \Pi_{\infty}(C,\Phi) \to \cE$  inverts every arrow in 
$\Pi_{\infty}(C,\Phi)$.
Since $C$ is contractible, we deduce that $V_a \colon  \Pi_{\infty}(C,\Phi) \to \cE$ is a constant functor.
Thus, we are left to show that for every $a\in \mathscr{I}(C)$, the map 
\begin{equation}\label{eq_polyhedral_fully_faithful}
\lim_{\Pi_{\infty}(C,\Phi)} W_a \to \lim_{\Pi_{\infty}(C,\Phi)} (i_{\cI, !}(W))_a
\end{equation}
is an equivalence.
At the cost of writing $W \colon \cI^{\ens} \to \cE$ as  a finite direct sum over $\mathscr{I}(C)$, we can suppose the existence of $b\in   \mathscr{I}(C)$ such that $W_a \simeq 0$ for $a\neq b$.
In that case, let $i_b \colon   \cI_b \hookrightarrow \cI$ be the cocartesian fibration constant to $b$, so that $W\simeq i_{b,!}^{\ens}(W_b)$ with $W_b \colon  \Pi_{\infty}(C,\Phi) \to \cE$ constant to an object $e\in \cE$.
Thus,  $i_{\cI, !}(W)  \simeq   i_{b,!}(W_b)$.
In particular, 
\[
(i_{\cI, !}(W) )_b \simeq i_{b}^* i_{b,!}(W_b)\simeq W_b \ .
\]
Hence, we are left to prove that  \eqref{eq_polyhedral_fully_faithful} is an equivalence for  $a \in \mathscr{I}(C)$ with $a\neq b$.
Let $\varphi\in \Phi$ such that the Stokes locus of $\{a,b\}$ is $H_{\varphi}$.
Let $C_1$ and $C_2$ be the two connected components of  $C\setminus H_{\varphi}$ such that $a <_{x} b$ for every $x\in C_1$ and $b <_{x} a$ for every $x\in C_2$.
Then 
\begin{eqnarray*}
(i_{\cI, !}(W) )_a(x)\simeq (i_{b,!}(W_b))_a(x)  & \simeq   0  & \text{ if $x \in H_{\varphi}$ or $x\in C_1$, }  \\
&  \simeq    e    & \text{  if $x\in C_2$ .  }
\end{eqnarray*}
Hence both functors in \eqref{eq_polyhedral_fully_faithful}  invert every exit-path in $C_1$,  in $C_2$ and in $H_{\varphi}$.
Consider the map
\[
\ev_\varphi \colon  \{-,0,+\}^{\Phi}\to \{-,0,+\} 
\]
given by evaluation at $\varphi$.
By \cref{refinement_localization}, the refinement
\[
(C,\Phi)\to (C,\{-,0,+\} )
\] 
induces a functor
\[
\Pi_{\infty}(C,\Phi)\to \Pi_{\infty}(C,\{\varphi\})
\]
 exhibiting the target as the localization of the source at the exit paths in $C_1$,  in $C_2$ and in $H_{\varphi}$.
Since localization functors are final, we are left to prove that \eqref{eq_polyhedral_fully_faithful} is an equivalence when $\Phi=\{\varphi\}$ and $W=W\simeq i_{b,!}^{\ens}(W_b)$.
In that case,  $\Pi_{\infty}(C,\Phi) \to \{-,0,+\}$ is an equivalence. 
Thus,  any point $x$ of $H_{\varphi}$ is initial in  $\Pi_{\infty}(C,\Phi)$.
Hence, the map \eqref{eq_polyhedral_fully_faithful} identifies canonically with 
\[
(i_{b,!}^{\ens}(W_b))_a(x) \to (i_{\cI, !}(W) )_a(x)  \ .
\]
Since both terms are $0$, \cref{polyhedral_fully_faithful} follows.
\end{proof}

\begin{prop}\label{polyhedral_essentially_surjective}
Let $(C,\Phi,\cI)$ be a polyhedral Stokes stratified space satisfying the conditions of \cref{induction_for_adm_Stokes_triple}.
Then, the induction functor 
\[
i_{\cI, !} \colon \St_{\cI^{\ens},\cE} \to \St_{\cI,\cE}
\]
is essentially surjective.
\end{prop}

\begin{proof}
The proof follows the method from \cite[Proposition 3.16]{MochStokes}. 
Let $F \colon \cI \to \cE$ be a Stokes functor.
By \cref{cocartesian_splitting}, it is enough to show that $F$ splits.
We argue by recursion on the cardinality of $\mathscr{I}(C)$.
If $\mathscr{I}(C)$ has one element, there is nothing to prove.
Suppose that $\mathscr{I}(C)$ has at least two elements.
Then, there exist open strata $U$ and $V$ and $a,b\in \mathscr{I}(C)$ distinct such that $a$ is minimal on $U$ and $b$ is minimal on $V$.
Let $i_a \colon   \cI_a \hookrightarrow \cI$ (resp. $i_b \colon   \cI_b \hookrightarrow \cI$) be the cocartesian fibration constant to $a$ (resp. $b$) and let $i \colon \cM \hookrightarrow \cI$ be the full subcategory spanned by objects not in $\cI_a$ nor $\cI_b$.
In particular, we have $\cI^{\ens} = \cI_a^{\ens} \sqcup \cI_b^{\ens}  \sqcup \cM^{\ens}$.
By \cref{image_of_S_equivalence_then_splitting}, the fiber sequences 
\[
 F_{<a}  \to  F_{a}    \to  \Gr_a  F \quad  \text{ and }  \quad  F_{<b}  \to  F_{b}    \to  \Gr_b  F 
\]
admit some splittings.
Let us choose some and let $F^{\backslash \cI_a} \colon \cI \to \cE$ and $F^{\backslash \cI_b} \colon \cI \to \cE$ be the corresponding functors as constructed in \cref{constr:F_minus_I}.
By \cref{fiber_product_is_split_IHES}, we have to show that $F^{\backslash \cI_a}$ and $F^{\backslash \cI_b}$ split.
We are going to show that $F^{\backslash \cI_a}$ splits as the argument is the same for $F^{\backslash \cI_b}$.
Let $i \colon \cI_b \cup \cM \hookrightarrow \cI$ be the subcategory spanned by the objects of $\cI$ not in $\cI_a$.
Since $F$ is a Stokes functor,  \cref{morphism_comes_from_graded_IHES} implies that $F^{\backslash \cI_a}$ is a Stokes functor as well.
Now an explicit computations yields $(\Gr F^{\backslash \cI_a})(c)\simeq 0$ for every  $c$ not in $\cI_b \cup \cM$.
By \cref{image_fully_faithful_induction_Stokes_IHES}, we deduce that $F^{\backslash \cI_a}$ lies in the essential image of 
$i_{ !} \colon \St_{\cI_b \cup \cM ,\cE} \to \St_{\cI,\cE}$.
By recursion assumption applied to $(C,\Phi, \cI_b \cup \cM)$, we deduce that $F^{\backslash \cI_a}$ splits.
\end{proof}

\section{Stokes structures and flat bundles}\label{sec:flat_bundles}

\subsection{Real blow-up}

\begin{defin}
A \textit{strict normal crossing pair} is the data  of $(X,D)$ where $X$ is a complex manifold and $D$ is a strict normal crossing divisor in $X$.
\end{defin}

\begin{notation}
Let $(X,D)$ be a strict normal crossing pair and put $U\coloneqq X\setminus D$.
Let $D_1,\dots, D_l$ be the irreducible components of $D$.
For $I\subset \{1,\dots, l\}$, we put
\[
D_I  \coloneqq  \bigcap_{i\in I} D_i  \quad \text{ and } \quad  D_I^{\circ}  \coloneqq  \bigcap_{I\subsetneq J} D_I \setminus D_J   \ .
\]
We denote by $i_I \colon  D_I \hookrightarrow  X$ and $i_I^{\circ} \colon  D_I \hookrightarrow X$ the canonical inclusions.
We note $(X,D)$ for the stratification 
$X\to \Fun( \{D_1,\dots, D_l\},\Delta^1)$  induced by the irreducible components of $D$.
\end{notation}

\begin{rem}\label{exit_path_sncd}
The canonical functor  $\Pi_{\infty}(X,D) \to \Fun( \{1,\dots, l\},\Delta^1)$  is an equivalence of $\infty$-categories.
\end{rem}

\begin{construction}[{\cite[§8.b]{Stokes_Lisbon}}]\label{real_blow_up}
Let $(X,D)$ be a strict normal crossing pair.
Let $D_1,\dots, D_l$ be the irreducible components of $D$.
For $i=1,\dots, l$,  let $L(D_i)$ be the line bundle over $X$ corresponding to the sheaf $\cO_X(D_i)$ and let $S^1 L(D_i)$ be the associated circle bundle. 
Put
\[
S^1 L(D)  \coloneqq  \bigoplus_{i=1}^l S^1 L(D_i)   \ .
\]
Let $U\subset X$ be an open polydisc with coordinates $(z_1,\dots, z_n)$ and let $z_i=0$ be an equation of $D_i$ in $U$.
Let $\widetilde{X}_U \subset S^1 L(D)|_U$ be the closure of the image of $(z_i/|z_i|)_{1\leq i \leq l} \colon  U\setminus D \to S^1 L(D)$.
Then, the $\widetilde{X}_U$ are independent of the choices made and thus glue as a closed subspace $\widetilde{X} \subset S^1 L(D)$ called the \textit{real-blow up of $X$ along $D$}.
We denote by $\pi \colon  \widetilde{X}\to X$  the induced proper morphism and by $j \colon  X\setminus D \to \widetilde{X}$ the canonical open immersion.
For $I\subset \{1,\dots, l\}$ of cardinal $1\leq k \leq l$, we put $\widetilde{D}_I \coloneqq  \pi^{-1}(D_I )$ and $\widetilde{D}_I^{\circ} \coloneqq  \pi^{-1}(D_I^{\circ} )$
and observe that the restriction 
\[
\pi|_{D_I^{\circ} }  \colon   \widetilde{D}_I^{\circ} \to D_I^{\circ}
\]
 is a $S^k$-bundle.
\end{construction}

\begin{eg}\label{local_real_blow_up}
Let $\Delta\subset \mathbb{C}^l$ be a polydisc  with coordinates $(z_1,\dots, z_l)$, let $Y$ be a complex manifold and put $X=\Delta \times Y$.
Let $D$ be the divisor defined by $z_1 \cdots z_l = 0$.
Then,  $S^1 L(D) = \Delta \times (S^1)^l\times Y$ and
\[
\widetilde{X} = \{(z,y,u)\in S^1 L(D) \text{ such that } z_k=|z_k| u_k, 1\leq k \leq l  \}  \ .
\]
In particular, 
\[
\widetilde{X} \simeq  ( \mathbb{R}_{\geq 0} \times S^1)^l   \times    \mathbb{C}^{n-l}
\]
and via the above identification, the inclusion $\widetilde{X} \hookrightarrow S^1 L(D)$ reads
\[
(r,u,y) \to (r_1 u_1,\dots, r_l u_l ,y,u)    \ .
\]
\end{eg}

\begin{rem}\label{extension_z_over_zmodule}
In  \cref{local_real_blow_up}, let $1\leq k\leq l$.
Then,  the map 
\[
z_k/|z_k | \colon  X\setminus D_k \to S^1
\] 
extends as a map $S^1 L(D) \to S^1$ given by $(z,y,u) \to u_k$.
\end{rem}

\cref{local_real_blow_up} implies the following

\begin{lem}
Let $(X,D)$ be a strict normal crossing pair.
Then,  $\widetilde{X}$ is a closed subanalytic subset of $S^1 L(D)$ and $\pi \colon  \widetilde{X}\to X$ is a subanalytic map.
\end{lem}

\begin{lem}\label{blow_up_strongly_proper}
Let $(X,D)$ be a strict normal crossing pair such that $X$ admits a smooth compactification.
Then, $\pi \colon  \widetilde{X}\to X$ is strongly proper (\cref{strongly_proper}).
\end{lem}

\begin{proof}
Let $X\hookrightarrow Y$ be a smooth compactification of $X$.
By the resolution of singularities,  we can suppose that $Z\coloneqq Y\setminus X$ is a divisor such that $E\coloneqq Z+D$ is a strict normal crossing divisor.
In particular, there is a pull-back square
\[ 
\begin{tikzcd}
		(S^1 L(D),\widetilde{X}) \arrow[hook]{r}\arrow{d} & (S^1 L(E),\widetilde{Y})  \arrow{d}\\
		 X \arrow[hook]{r} & Y  \ .
	\end{tikzcd} 
\]
Then \cref{blow_up_strongly_proper} follows from \cref{compactifiable_is_stronly_proper}.
\end{proof}

\begin{recollection}[{\cite[§8.c]{Stokes_Lisbon}}]\label{moderate_growth_sheaf}
Let $(X,D)$ be a strict normal crossing pair and put $U \coloneqq X\setminus D$.
Let $\pi \colon  \widetilde{X}\to X$  be the real blow-up along $D$ and let $j \colon U \hookrightarrow \widetilde{X}$ be the canonical inclusion.
We denote by $ \cA^{\mode}_{\widetilde{X}}\subset  j_{\ast}\cO_{U}$ the sheaf of analytic functions with moderate growth along $D$.
By definition for every open subset $V\subset  \widetilde{X}$, a section of $\cA^{\mode}_{\widetilde{X}}$ on $V$ is an analytic function $f \colon  V\cap U \to \mathbb{C}$ such that for every open subset $W\subset V$ with $D$ defined by $h=0$ in a neighbourhood of $\pi(W)$, for every compact subset $K \subset W$, there exist  $C_K> 0$ and $N_K \in \mathbb{N}$ such that for every $z\in K\cap U$, we have 
\[
|f(z)|\leq C_K \cdot |h(z)|^{-N_K}  \ .
\]
\end{recollection}

The following lemma is obvious:

\begin{lem}\label{moderate_invertible_bounded}
In the setting of \cref{moderate_growth_sheaf}, let $(j_{\ast}\cO_{U})^{\lb}\subset j_{\ast}\cO_{U}$ be the subsheaf of locally bounded functions.
Then $\cA^{\mode}_{\widetilde{X}}$ is a unitary sub  $(j_{\ast}\cO_{U})^{\lb}$-algebra of  $j_{\ast}\cO_{U}$ such that 
\[
\cA^{\mode,\times }_{\widetilde{X}} \subset (j_{\ast}\cO_{U})^{\lb} \ .
\]
\end{lem}

\begin{recollection}[{\cite[Definition 9.2]{Stokes_Lisbon}}]\label{order_general}
Let $(X,D)$ be a strict normal crossing pair and put $U \coloneqq X\setminus D$.
Let $\pi \colon  \widetilde{X}\to X$  be the real blow-up along $D$ and let $j \colon U \hookrightarrow \widetilde{X}$ be the canonical inclusion.
For $f,g\in j_{\ast}\cO_{U}$, we write 
\[
f\leq g \text{ if and only if } e^{f-g}\in \cA^{\mode}_{\widetilde{X}}  \ .
\]
By \cref{moderate_invertible_bounded}, the relation $\leq$ induces an order on $(j_{\ast}\cO_{U})/(j_{\ast}\cO_{U})^{\lb}$.
From now on, we view  $(j_{\ast}\cO_{U})/(j_{\ast}\cO_{U})^{\lb}$ as an object of 
$ \Sh^{\hyp}(\widetilde{X},\Poset)$.
\end{recollection}

\begin{rem}\label{meromorphic_order}
Viewing $\pi^{\ast}\cO_{X}(\ast D)$ inside $j_{\ast}\cO_{U}$, we have 
\[\pi^{\ast}\cO_{X}(\ast D)\cap  (j_{\ast}\cO_{U})^{\lb} = \pi^{\ast}\cO_{X}  \ .
\]
Hence, $\pi^{\ast}(\cO_{X}(\ast D)/\cO_{X})$ can be seen as a subsheaf of $(j_{\ast}\cO_{U})/(j_{\ast}\cO_{U})^{\lb}$.
From now on, we view it as an object of $ \Sh^{\hyp}(\widetilde{X},\Poset)$.
\end{rem}

\subsection{Sheaf of unramified irregular values}\label{secSheaf_of_unramifie_irregular_values}

\begin{defin}\label{locally_generated_hypersheaf}
Let $X$ be a topological space.
Let $\cF \in \Shhyp(X,\Cat_{\infty})$.
We say that $\cF $ is \textit{locally generated} if there is a cover by open subsets $U\subset X$ such that for every $x\in U$, the functor $\cF(U)\to \cF_x$ is essentially surjective.
We say that $\cF $ is \textit{globally generated} if for every $x\in X$, the functor $\cF(X)\to \cF_x$ is essentially surjective.
\end{defin}


\begin{lem}\label{pullback_locally_generated_lem}
Let $f\colon  Y\to X$ be a morphism of topological spaces.
Let $\cF \in \Shhyp(X,\Cat_{\infty})$.
If $\cF$ is locally (resp.  globally) generated, then so is $f^{\ast, \hyp}(\cF)$.
\end{lem}

\begin{proof}
We argue in the locally generated situation,  the globally generated situation being similar.
Let $y\in Y$ and put $x=f(y)$.
Let $U\subset X$ be an open neighbourhood of $x$ as in \cref{locally_generated_hypersheaf}.
Let $V\subset Y$ be an open neighbourhood of $y$ such that $f(V)\subset U$.
For $z\in V$,  there is a factorization 
\[
\cF(U) \to (f^{\ast, \hyp}(\cF))(V)\to (f^{\ast, \hyp}(\cF))_z \simeq \cF_{f(z)} \ .
\]
Since the composition is essentially surjective, so is the second functor.
\end{proof}

\begin{recollection}[{\cite[Definition 2.4.2]{Mochizuki1}}]\label{good sheaf}
Let $(X,D)$ be a strict normal crossing pair.
A \textit{sheaf of unramified irregular values} is a locally generated subsheaf of finite sets $\mathscr{I}\subset \cO_{X}(\ast D)/\cO_{X}$ in the sense of \cref{locally_generated_hypersheaf}.
\end{recollection}

The goal of what follows is to show that a sheaf of unramified irregular values is automatically constructible on $(X,D)$.

\begin{lem}\label{presheaf_pullback_Qcoh}
Let $X\subset \mathbb{C}^n$ be a polydisc  with coordinates $(z_1,\dots, z_n)$ and $D$ defined by $z_1\cdots z_l = 0$ for some $1\leq l\leq n$ with $ I=\{1,\dots, l\}$ and put $\cE\coloneqq\cO_{X}(\ast D)/\cO_{X}$.
Then, the map 
$$
(i_I^{\circ,\ast}\cE)(X\cap D_I^{\circ})\to \cE_0
$$
is injective.
\end{lem}

\begin{proof}
A section of $i_I^{\circ,\ast}\cE$ above $X\cap D_I^{\circ}$ is a function
$$
s \colon X\cap D_I^{\circ} \to \bigsqcup_{x\in X\cap D_I^{\circ}}  \cE_x
$$
such that there exists a collection of polydiscs $(U_j)_{j\in J}$ such that the $U_j \cap D_I^{\circ}$ cover $X\cap D_I^{\circ}$ and for every $j\in J$, there exists $s_j\in \cE(U_j)$ such that $s_j$ and $s$ coincide on $U_j \cap D_I^{\circ}$.
Since a polydisc is a Stein manifold, $s_j$ can be represented by a meromorphic function $f_j\in \cO_{X}(\ast D)(U_j)$ modulo $\cO_{X}(U_j)$.
Assume now that $s_0=0$.
Since $i_I^{\circ,\ast}\cE$ is a sheaf, the set $S$  of points $x$ where $s_x=0$ is thus a non empty open subset of $X\cap D_I^{\circ}$.
On the other hands, for every $j\in J$, the meromorphic function $f_j$ is holomorphic if and only if it is holomorphic in a neighbourhood of a point in $U_j \cap D_I^{\circ}$.
Thus, $S$ is closed, which proves \cref{presheaf_pullback_Qcoh}.
\end{proof}

\begin{prop}\label{locally_generated_implies_constructibility}
Let $(X,D)$ be a strict normal crossing pair.
Let  $\mathscr{I}\subset \cO_{X}(\ast D)/\cO_{X}$ be a sheaf of unramified irregular values.
Then,  $\mathscr{I} \in \Cons_D^{\hyp}(X,\Setcat)$.
\end{prop}

\begin{proof}
Let $D_1,\dots, D_l$ be the irreducible components of $D$.
Let $I\subset \{1,\dots, l\}$ be a subset.
We have to show that $i_I^{\circ,\ast}(\mathscr{I})$ is locally constant.
The question is local.
Hence,  we can suppose that  $X\subset \mathbb{C}^n$ is a polydisc  with coordinates $(z_1,\dots, z_n)$ and $D$ defined by $z_1\cdots z_l = 0$ for some $1\leq l\leq n$ with $ I=\{1,\dots, l\}$.
Let $x\in D_I^{\circ}$.
Let $B \subset X$ be a polydisc centred at $x$.
We can suppose that $x=0$.
We have to show that at the cost of shrinking $B$ further, the restriction $(i_I^{\circ,\ast}\mathscr{I})|_{B\cap D_I^{\circ}}$ is a constant sheaf, that is the map
\begin{equation}\label{eq_locally_generated_implies_constructibility}
(i_I^{\circ,\ast}\mathscr{I})(B\cap D_I^{\circ})\to \mathscr{I}_y
\end{equation}
is bijective for every $y\in B\cap D_I^{\circ}$.
Since $\mathscr{I}$ is locally generated, we can suppose that \eqref{eq_locally_generated_implies_constructibility} is surjective for every $y\in B\cap D_I^{\circ}$.
The injectivity follows from \cref{presheaf_pullback_Qcoh}.
\end{proof}

\begin{rem}\label{Iens_constructible}
In the setting of \cref{locally_generated_implies_constructibility}, let us denote by $(\widetilde{X},\widetilde{D})$ the space $\widetilde{X}$ endowed with the stratification induced by that of $D$ on $X$.
Then, \cref{locally_generated_implies_constructibility} yields  $\pi^{\ast}\mathscr{I}\in \Cons_{\widetilde{D}}^{\hyp}(\widetilde{X},\Setcat)$.
\end{rem}

Under constructibility assumption, local generation can sometimes be upgraded into global generation, due to the following

\begin{lem}\label{semi_globally_generated}
Let $(M,X,P)$ be a subanalytic stratified space where $\Pi_{\infty}(X,P)$ admits an initial object.
Then, every locally generated constructible sheaf $\cF \in \Cons_P^{\hyp}(X,\Cat_{\infty})$ is globally generated.
\end{lem}

\begin{proof}
Let $y \in X$.
We want to show that $\cF(X)\to \cF_y$ is essentially surjective.
Let $x \in X$ initial in $\Pi_{\infty}(X,P)$ and let $U \subset X$ be an open neighbourhood of $x$ on which $\cF$ is globally generated. 
At the cost of shrinking $U$, we can further suppose by \cref{prop:locally_contractible_strata} that $x$ is initial in $\Pi_{\infty}(U,P)$. 
Choose a morphism $\gamma\colon x\to y$ in  $\Pi_{\infty}(X,P)$.
At the cost of replacing $y$ by a point of $\gamma$ distinct from $x$ and sufficiently close to $x$, we can suppose that $y\in U$.
Let $F \colon \Pi_{\infty}(X,D) \to \Cat_{\infty}$ be the functor corresponding to $\cF$ via the exodromy equivalence \eqref{exodromy_equivalence}.
By assumption, the second arrow of
\[
\lim_{\Pi_{\infty}(X,P)}	F\to   \lim_{\Pi_{\infty}(U,P)}  F \to F(y)
\]
is essentially surjective, while the first one  is an equivalence since $x$ is initial in both  $\Pi_{\infty}(X,P)$ and 
$\Pi_{\infty}(U,P)$.
\cref{semi_globally_generated} thus follows.
\end{proof}

\begin{eg}\label{global_generation_X}
Let $\Delta\subset \mathbb{C}^l$ be a polydisc  with coordinates $(z_1,\dots, z_l)$, let $Y$ be a weakly contractible complex manifold and put $X=\Delta \times Y$.
Let $D$ be the divisor defined by $z_1 \cdots z_l = 0$.
Then $0$ is initial in $\Pi_{\infty}(X,D)$.
\end{eg}

\begin{eg}\label{global_generation_Xtilde}
Let $\Delta\subset \mathbb{C}^l$ be a polydisc  of radius $r>0$ with coordinates $(z_1,\dots, z_l)$, let $Y$ be a weakly contractible complex manifold and put $X=\Delta \times Y$.
Let $D$ be the divisor defined by $z_1 \cdots z_l = 0$.
Let $\pi \colon  \widetilde{X}\to X$  be the real blow-up of $X$ along $D$.
Let  $I_1,\dots, I_l \subset S^1$ be strict open intervals.
Then, any point of $[0,r)^l\times I_1\times \cdots \times I_l \times Y\subset \widetilde{X}$ above the origin is initial in 
\[
\Pi_{\infty}([0,r)^l\times I_1\times \cdots \times I_l \times Y ,\widetilde{D}) \ . 
\]
\end{eg}

\begin{cor}\label{semi_globally_generated_irregular_values}
Let $Y$ be a weakly contractible complex manifold.
Let $\Delta\subset \mathbb{C}^l$ be a polydisc with coordinates $(z_1,\dots, z_l)$ and put $X=\Delta\times Y$.
Let $D$ be the divisor defined by $z_1 \cdots z_l = 0$.
Let $\mathscr{I}\subset \cO_{X}(\ast D)/\cO_{X}$ be a sheaf of unramified irregular values.
Then,  $\mathscr{I}$ is globally generated.
\end{cor}

\begin{proof}
Combine \cref{locally_generated_implies_constructibility} with \cref{semi_globally_generated} applied to \cref{global_generation_X}.
\end{proof}

%

\begin{cor}\label{global_section_real_blow_up_I}
      Let $\Delta\subset \mathbb{C}^l$ be a polydisc  with coordinates $(z_1,\dots, z_l)$, let $Y$ be a weakly contractible complex manifold and put $X=\Delta \times Y$.
      Let $D$ be the divisor defined by $z_1 \cdots z_l = 0$.
      Let $\pi \colon  \widetilde{X}\to X$  be the real blow-up of $X$ along $D$.
      Let $\mathscr{I}\subset \cO_{X}(\ast D)/\cO_{X}$ be a sheaf of unramified irregular values.
      Then, the canonical restriction map
\[
\mathscr{I}(X) \to (\pi^{\ast}\mathscr{I})(\widetilde{X})
\]
is bijective.
\end{cor}

\begin{proof}
      By \cref{locally_generated_implies_constructibility}, the sheaf 
      $\mathscr{I}$ is constructible on $(X,D)$, so that $\pi^{\ast}\mathscr{I}$ is constructible on $(\widetilde{X},\widetilde{D})$ (see \cref{Iens_constructible}).
      Let $F \colon \Pi_\infty(X,D) \to \Setcat$ be the functor corresponding to $\mathscr{I}$ via the exodromy equivalence \eqref{exodromy_equivalence}.
	By \cref{exodromy_functorialities}, we have to show that 
\[
\lim_{\Pi_{\infty}(X,D)}	F\to   \lim_{\Pi_{\infty}(\widetilde{X},\widetilde{D})}   F \circ \pi
\]
is an equivalence.
Since $Y$ is weakly contractible, we can suppose that $Y$ is a point.
Since $\Setcat$ is a 1-category, the functor $F \colon \Pi_\infty(X,D) \to \Setcat$ factors uniquely through the homotopy category $\ho( \Pi_\infty(X,D) )$ as a functor $G \colon \ho( \Pi_\infty(X,D) )\to \Setcat$.
Hence we are left to show that 
\[
\lim_{\ho(\Pi_{\infty}(X,D))}	G  \to   \lim_{\ho(\Pi_{\infty}(\widetilde{X},\widetilde{D}))}   G \circ \pi
\]
is an equivalence.
To do this, it is enough to show that
\begin{equation}\label{eq:global_section_real_blow_up_I}
\ho(\Pi_{\infty}(X,D)) \to \ho(\Pi_{\infty}(\widetilde{X},\widetilde{D})) 
\end{equation}
is final in the $1$-categorical sense.
If $r>0$ denotes the radius of $\Delta$, we have 
\[
\widetilde{X} = [0,r)^l\times (S^1)^l \ .
\]
Since $\ho$ commutes with finite products, we obtain
\[
\ho(\Pi_{\infty}(\widetilde{X},\widetilde{D}) ) \simeq \ho(\Pi_{\infty}( [0,r)^l,D))\times \ho(\Pi_{\infty}((S^1)^l))   \ .
\]
Via this equivalence, the functor \eqref{eq:global_section_real_blow_up_I} identifies with  the projection on the first term.
By \cite[4.1.1.13]{HTT}, we are thus left to show that $\ho(\Pi_{\infty}((S^1)^l))$ is connected, which is obvious.
\end{proof}

\cref{global_section_real_blow_up_I} implies immediately the following 

\begin{cor}\label{unit_irr_value_equivalence}
      Let $(X,D)$ be a strict normal crossing pair.
      Let $\mathscr{I}\subset \cO_{X}(\ast D)/\cO_{X}$ be a sheaf of unramified irregular values.
      Let $\pi \colon  \widetilde{X}\to X$  be the real blow-up along $D$.
      Then the unit transformation
\[
\mathscr{I}  \to \pi_* \pi^{\ast} \mathscr{I}
\]
is an equivalence.
\end{cor}

\subsection{Good sheaf of unramified irregular values}

\begin{defin}
Let $X\subset \mathbb{C}^n$ be a polydisc with coordinates $(z,y)\coloneqq (z_1,\dots, z_l,y_1,\dots y_{n-l})$.
Let $D$ be the divisor defined by $z_1 \cdots z_l = 0$.
Let $a\in \cO_{X,0}(\ast D)/\cO_{X,0}$ and consider the Laurent expansion 
\[
\sum_{m\in \mathbb{Z}^l}   a_m(y) z^m   \ .
\]
We say that \textit{$a$ admits an order} if  the set 
\[
\{m \in \mathbb{Z}^l \text{ with } a_m \neq 0\}\cup \{0\}
\]
admits a smallest element, denoted by $\order a$.
\end{defin}

\begin{rem}
The existence of an order does not depend on a choice of coordinates on $X$.
\end{rem}

\begin{recollection}[{\cite[Definition 2.1.2]{Mochizuki1}}]\label{goodness}
Let $(X,D)$ be a strict normal crossing pair.
Let $x\in X$.
A subset $I\subset \cO_{X,x}(\ast D)/\cO_{X,x}$ is \textit{good} if 
\begin{enumerate}\itemsep=0.2cm
\item  every non zero $a\in I$  admits an order with $a_{\order a}$ invertible in  $\cO_{X,x}$.

\item For every distinct $a,b\in I, a-b$  admits an order with $(a-b)_{\order (a-b)}$ invertible in  $\cO_{X,x}$.

\item The set $\{\order(a-b), a,b\in I\}\subset  \mathbb{Z}^l $ is totally ordered.
\end{enumerate}
\end{recollection}

\begin{recollection}[{\cite[Definition 2.4.2]{Mochizuki1}}]\label{good sheaf}
Let $(X,D)$ be a strict normal crossing pair.
A \textit{good sheaf of unramified irregular values} is a sheaf of unramified irregular values such that for every $x\in X$,  the set $\mathscr{I}_x \subset \cO_{X,x}(\ast D)/\cO_{X,x}$ is good.
\end{recollection}

When restricted to good sheaves of irregular values, the order from \cref{order_general} admits a handy characterisation that we now describe.

\begin{recollection}[{\cite[§3.1.2]{Mochizuki1}}]\label{order_good_irregular_values}
Let $\Delta\subset \mathbb{C}^l$ be a polydisc  with coordinates $(z_1,\dots, z_l)$, let $Y$ be a complex manifold and put $X=\Delta \times Y$ and $U\coloneqq X\setminus D$.
Let $\pi \colon  \widetilde{X}\to X$ be the real blow-up along $D$ and let $x\in \widetilde{X}$.
Let $a,b\in (\pi^{-1}(\cO_{X}(\ast D)/ \cO_{X}))_x$ and let $\mathfrak{a}$ and $\mathfrak{b}$ be lifts of $a$ and $b$ to $\cO_{X}(\ast D)$ on some open subset $V\subset X$.
By \cref{extension_z_over_zmodule}, the function
\[
\Real(\mathfrak{a}- \mathfrak{b})|z^{-\order(a-b)}|  \colon  V\setminus D\to \mathbb{R}
\]
extends as a real analytic function 
\[
F_{a,b} \colon \pi^{-1}(V)\to \mathbb{R}   \ .
\]
Then, the following are equivalent:
\begin{enumerate}\itemsep=0.2cm
\item  $a \leq_x b$ in the sense of \cref{order_general};
\item $a = b$  or $a \neq  b$ and  $F_{a,b}(x)<0$ .
\end{enumerate}
\end{recollection}

The goal of what follows is to show that for every  \textit{good} sheaf of unramified irregular values  $\mathscr{I}\subset \pi^{\ast}(\cO_{X}(\ast D)/ \cO_{X})$, there exists  a \textit{finite} subanalytic stratification $\widetilde{X}\to P$ such that $\pi^{\ast}\mathscr{I}\in \ConsPhyp(\widetilde{X},\Poset)$.
Before that, a couple of intermediate steps are needed.
To this end, we introduce the following

\begin{defin}\label{Uasmallthanb}
Let $(M,X)$ be a subanaltyic stratified space.
Let  $\cF \in \Sh\hyp(X,\Poset)$.
Note that for  $x\in X$, the stalk  
\[
\cF_x =\colim_{x\in U} \cF(U)
\]
is naturally endowed with an order $ \leq_x$ by performing the above colimit in Poset instead of Set.
For an open subset $U\subset X$ and for $a,b \in \cF(U)$, we put
\[
U_{a < b} \coloneqq  \{x\in U \text{ such that }  a_x <_x b_x \text{ in  }  \cF_x   \}   
\]
 and
\[
U_{a = b} \coloneqq  \{x\in U \text{ such that }  a_x = b_x \text{ in  }  \cF_x   \}   
\]
and 
\[
U_{a \ast b}\coloneqq  \{x\in U \text{ such that }  a_x \text{ and } b_x \text{ cannot be compared in } \cF_x   \}     \ .
\]
\end{defin}

\begin{rem}\label{Uasmallthan_rem}
Let $(M,X)$ be a subanaltyic stratified space and let  $\cF \in \Shhyp(X,\Poset)$.
For every open subset $U\subset X$ and for every $a,b \in \cF(U)$, the set $U_{a = b}$ is open and $U_{a < b}$ and $U_{a \ast b} = U\setminus (U_{a < b}\cup  U_{a > b} \cup U_{a = b} )$  are locally closed.
\end{rem}

\begin{eg}\label{stokes_loci_are_subanalytic}
Let $\Delta\subset \mathbb{C}^l$ be a polydisc  with coordinates $(z_1,\dots, z_l)$, let $Y$ be a complex manifold and put $X=\Delta \times Y$.
Let $D$ be the divisor defined by $z_1 \cdots z_l = 0$.
Let $\mathscr{I}\subset \cO_{X}(\ast D)/\cO_{X}$ be a sheaf of unramified irregular values.
Let $\pi \colon  \widetilde{X}\to X$ be the real blow-up along $D$.
Let $\alpha,\beta \in \mathscr{I}(X)$ and put $a=\pi^{\ast}\alpha \in (\pi^{\ast}\mathscr{I})(\widetilde{X})$ and $b=\pi^{\ast}\beta \in (\pi^{\ast}\mathscr{I})(\widetilde{X})$.
Let $A\subset \{1,\dots, l\}$ be the set of indices $i$ such that $\alpha-\beta$ has a pole along $D_i$. 
By \cref{order_good_irregular_values}, we have
\[
\widetilde{X}_{a = b}= \bigsqcup_{I \subset \{1,\dots, l\}\setminus A}\widetilde{D}_I^{\circ}  
\]
and
\[
\widetilde{X}_{a < b} = \bigsqcup_{\substack{I \subset \{1,\dots, l\} \\ I \cap A \neq \emptyset}}  \widetilde{D}_I^{\circ} \cap \{F_{a,b}<0\}     \ .
\]
Furthermore, 
\[
\widetilde{X}_{a \ast b} = \widetilde{X} \setminus (\widetilde{X}_{a < b}\cup  \widetilde{X}_{a > b} \cup \widetilde{X}_{a = b} )  \ .
\]
In particular the three sets above are  subanalytic in $S^1 L(D)$.
\end{eg}

\begin{lem}\label{from_local_to_global_subanalyticity}
Let $(M,X,P)$ be a subanaltyic stratified space where $X$ is closed.
Let $\cF \in \Shhyp(X,\Poset)$.
Let $?\in \{<,=,\ast\}$.
Assume that
\begin{enumerate}\itemsep=0.2cm
\item  $\cF^{\ens} \in \ConsPhyp(X,\Setcat)$ ;
\item  $\cF$ is locally generated (\cref{locally_generated_hypersheaf}) ;
\item there exists a fundamental system of open  neighbourhoods $W\subset M$  such that for every $a,b\in \cF(W\cap X),$ the set $(W\cap X)_{a ? b}$, is subanalytic in $W$.
\end{enumerate}
Then,  for every open subset $U\subset X$ subanalytic in $M$, for every $a,b\in \cF(U)$, the set $U_{a?b}$ is locally closed subanalytic in $M$.
\end{lem}

\begin{proof}
Local closeness is automatic by \cref{Uasmallthan_rem}.
Let $x\in M$.
We need to show that  $U_{a<b}$ is subanalytic in a neighbourhood of  $x$ in $M$.
Since $X$ is closed, we can suppose that $x\in X$.
At the cost of replacing $M$ by a sufficiently small open neighbourhood of $x$ in $M$,  we can suppose by $(2)$  that $\cF$ is globally generated.
At the cost of shrinking $M$ further,  we can suppose that $P$ is finite.
Since $U$ is a subanalytic subset of $M$, so are the $U_p=U\cap X_p$ for $p\in P$.
On the other hand, the set of connected components of a subanalytic subset is locally finite.
Hence, at the cost of replacing $M$ by a smaller neighbourhood of $x$, we can suppose that the $U_p$ have only a finite number of connected components $C_{1,p},\dots,C_{n(p),p}$.
By global generation, for $p\in P$ and $1\leq i \leq n(p)$,  the sections $a|_{C_{i,p}}, b|_{C_{i,p}}$ extend  to $X$ as sections $\alpha_{i,p}, \beta_{i,p}$ of $\cF$.
At the cost of replacing $M$ by a smaller neighbourhood of $x$, we can suppose by (3) that the $X_{\alpha_{i,p} ?\beta_{i,p}}$ are subanalytic in $M$.
Moreover,
\[
U_{a? b}  = \bigsqcup_{p\in P} U_{a ? b}\cap U_p   
                           = \bigsqcup_{p\in P}\bigsqcup_{i=1}^{n(p)}  (C_{i,p})_{a|_{C_{i,p}} ? b|_{C_{i,p}}}   
                          =\bigsqcup_{p\in P}\bigsqcup_{i=1}^{n(p)}   X_{\alpha_{i,p} ? \beta_{i,p}} \cap C_{i,p}    \  .
\]
Since a finite union and intersection of subanalytic subsets is again subanalytic, \cref{from_local_to_global_subanalyticity} is thus proved. 
\end{proof}

\begin{cor}\label{good_sheaf_irregular_values_is_subanalytic}
Let $(X,D)$ be a strict normal crossing pair.
Let $\mathscr{I}\subset \cO_{X}(\ast D)/\cO_{X}$ be a good sheaf of unramified irregular values.
Let $\pi \colon  \widetilde{X}\to X$  be the real blow-up along $D$ and consider $\pi^{\ast}\mathscr{I} \in \Shhyp(\widetilde{X},\Poset)$.
For every open subset $U \subset \widetilde{X}$ subanalytic in $S^1 L(D)$, 
for every $a,b\in (\pi^{\ast}\mathscr{I})(U)$, the sets $U_{a< b}$,$U_{a= b}$,$U_{a\ast b}$ are locally closed  subanalytic in $S^1 L(D)$.
\end{cor}

\begin{proof}
Let $?\in \{<,=,\ast\}$.
We prove that $U_{a ? b}$ is locally closed subanalytic in $M$.
We check that the conditions of \cref{from_local_to_global_subanalyticity} are satisfied.
First observe that $\widetilde{X}$ is closed in $S^1 L(D)$.
Condition (1) is satisfied by \cref{Iens_constructible}.
Condition (2) is satisfied by  \cref{pullback_locally_generated_lem}.
To check (3),  we can suppose that $X\subset\mathbb{C}^n$  is a polydisc with $D$ defined by $z_1 \cdots z_l = 0$.
Let $x\in \widetilde{X}$.
We want to find a fundamental system of open neighbourhoods of  $x$ in $S^1 L(D)$ satisfying (3).
By \cref{prop:locally_contractible_strata}, it is enough to show that any open subset $W\subset S^1 L(D)$ such that $x$ is initial in $\Pi_{\infty}(W\cap \widetilde{X},\widetilde{D})$ does the job.
Indeed let $W\subset S^1 L(D)$ be such an open subset and put $U \coloneqq  W\cap \widetilde{X}$.
Let $a,b\in (\pi^{\ast}\mathscr{I})(U)$.
By \cref{global_section_real_blow_up_I}, the canonical restriction map
\[
\mathscr{I}(X) \to (\pi^{\ast}\mathscr{I})(\widetilde{X})
\]
is bijective with $\mathscr{I}$  and $\pi^{\ast}\mathscr{I}$  globally generated in virtue of  \cref{semi_globally_generated_irregular_values} and \cref{pullback_locally_generated_lem}.
Hence,  there is $\alpha,\beta  \in  \mathscr{I}(X)$ such that $ a_x=(\pi^{\ast}\alpha)_{x}$ and $b_x=(\pi^{\ast}\beta)_{x}$.
Since $x$ is initial in $\Pi_{\infty}(W\cap \widetilde{X},\widetilde{D})$, we obtain $a=(\pi^{\ast}\alpha)|_U$ and $b=(\pi^{\ast}\beta)|_U$.
Thus, we have 
\[
U_{a< b} = \widetilde{X}_{\pi^{\ast}\alpha ? \pi^{\ast}\beta}  \cap W    \ .
\]
Hence, to show that $U_{a ? b}$ is subanalytic in $W$, it is enough to show that  
$\widetilde{X}_{\pi^{\ast}\alpha ? \pi^{\ast}\beta}$ is subanalytic in $S^1 L(D)$. 
This case follows from \cref{stokes_loci_are_subanalytic}.
\end{proof}



\begin{lem}\label{finite_subanalytic_stratification_criterion}
Let $(M,X,P)$ be a subanalytic stratified space where $P$ is finite.
Let  $\cF\in \Shhyp(X,\Poset)$ such that $\cF^{\ens}$ is $P$-hyperconstructible and takes values in finite sets.
Assume the existence of a finite cover of $X$ by open subanalytic subsets $U\subset X$ such that
\begin{enumerate}\itemsep=0.2cm
\item $\cF|_U$ is globally generated ;

\item for every $a,b\in \cF(U)$, the sets $U_{a<b}$, $U_{a = b}$ and $U_{a \ast b}$  are locally closed subanalytic in $M$.
\end{enumerate}
Then, there is a finite subanalytic refinement $Q\to P$ such that $\cF\in \ConsQhyp(X,\Poset)$.
\end{lem}

\begin{proof}
Let $U\subset X$ be an open subanalytic subset satisfying (1) and (2).
For $f \colon  \cF(U)\times \cF(U) \to \{<,=,\ast,>\}$ and $p\in P$, put
\[
U_{f,p} \coloneqq   U_p\bigcap \bigcap_{(a,b)\in \cF(U)^2}  U_{a f(a,b) b} \ .
\]
Note that $U_p$ is a subanalytic subset of $M$ since $U$ and $X_p$ are.
Since $\cF(U)$ is finite,  item (2) implies that $U_{f,p}$ is a locally closed subanalytic subset of $M$.
By  assumption, we have $\cF^{\ens}|_{U_p}\in \Loc^{\hyp}(X_p,\Setcat)$.
By (1), we deduce $\cF|_{U_{f,p}}\in \Loc^{\hyp}(U_{f,p},\Poset)$.
Now any finite subanalytic common efinement of $P$ and  $\{U_{f,p}\}_{f,p}$ does the job.
\end{proof}

\begin{cor}\label{good_sheaf_constructible_finite_stratification}
Let $(X,D)$ be a strict normal crossing pair where $X$ admits a smooth compactification.
Let $\mathscr{I}\subset \cO_{X}(\ast D)/\cO_{X}$ be a good sheaf of unramified irregular values.
Let $\pi \colon  \widetilde{X}\to X$  be the real blow-up along $D$.
Then, there exists a finite subanalytic stratification 
$\widetilde{X}\to P$ refining $(\widetilde{X},\widetilde{D})$ such that $\pi^{\ast}\mathscr{I}\in \ConsPhyp(\widetilde{X},\Poset)$.
\end{cor}

\begin{proof}
By \cref{Iens_constructible}, $(\pi^{\ast}\mathscr{I})^{\ens}$ is  hyperconstructible on $(\widetilde{X},\widetilde{D})$.
Let $ X\hookrightarrow Y$ be a smooth compactification of $X$.
At the cost of applying resolution of singularities,  we can suppose that $Z\coloneqq Y\setminus X$ is a divisor such that $E\coloneqq Z+D$ has strict normal crossings.
Hence, $X$ admits a finite cover by open subanalytic subsets $U\simeq  \Delta^{n-k} \times (\Delta^*)^k$ with coordinates $(z,y)$ such that $D\cap U$ is defined by $z_1\cdots z_l = 0$, where $\Delta\subset \mathbb{C}$ is the unit disc.
Let $S_+, S_- \subset \Delta^*$ be a cover by open sectors.
For $\varepsilon \colon  \{1,\dots, k\} \to \{-,+\}$, put 
\[
U_\varepsilon \coloneqq   \Delta^{n-k} \times S_{\varepsilon(1)} \times \cdots   \times S_{\varepsilon(k)}  
\]
and $\widetilde{U}_\varepsilon  \coloneqq  \pi^{-1}(U_\varepsilon)$.
Note that $\widetilde{U}_\varepsilon$ is a subanalytic subset of $S^1L(D)$ since $U_\varepsilon \subset X$ is subanalytic.
To conclude, it is enough to show that $\widetilde{U}_\varepsilon$ satisfies the conditions (1) and (2) of \cref{finite_subanalytic_stratification_criterion}.
By \cref{semi_globally_generated}, the sheaf $\mathscr{I}|_{U_\varepsilon}$ is globally generated.
By \cref{pullback_locally_generated_lem}, we deduce that $(\pi^{\ast}\mathscr{I})|_{\widetilde{U}_\varepsilon}$ is globally generated.
Let $a,b\in (\pi^{\ast}\mathscr{I})(\widetilde{U}_{\varepsilon})$ and $? \in \{<,=,\ast\}$.
By \cref{good_sheaf_irregular_values_is_subanalytic},
the set $\widetilde{U}_{\varepsilon, a ? b}$ is  subanalytic in $S^1L(D)$.
By \cref{Uasmallthan_rem}, it is locally closed in $S^1L(D)$.
Then,   \cref{good_sheaf_constructible_finite_stratification}  follows from  \cref{finite_subanalytic_stratification_criterion}.
\end{proof}

\subsection{Level structure}

\begin{construction}\label{level_structure_diff}
The goal of what follows is to construct a local level structure for good sheaves of unramified irregular values.
Assume that $X\subset \mathbb{C}^n$ is a polydisc with coordinates $(z,y)=(z_1,\dots, z_l,y_1,\dots, y_{n-l})$.
Let $D$ be the divisor defined by $z_1 \cdots z_l = 0$.
Let $\mathscr{I}\subset \cO_{X}(\ast D)/\cO_{X}$ be a good sheaf of unramified irregular values.
Let $\pi \colon  \widetilde{X}\to X$  be the real-blow up along $D$ and let $\widetilde{X} \to P$ be a finite subanalytic stratification adapted to $\mathscr{I}$.
Let $\widetilde{X}\to P$ be a finite subanalytic stratification such that $\pi^{-1}\mathscr{I}$ is $P$-constructible.
By condition (3) from \cref{goodness}, the set $\{a-b, a,b\in I\}$ is totally ordered with respect to the partial order on $\mathbb{Z}^l$.
Hence, there exists a sequence 
\begin{equation}\label{auxiliary_sequence}
m(0)<m(1)<\dots <m(d)=0
\end{equation} 
in $\mathbb{Z}^l$ such that for every $k=0,\dots, d-1$, the vectors $m(k)$ and $m(k+1)$ differ only by $1$ at exactly one coordinate and every $\order(a-b)$ for $a,b\in \mathscr{I}(X)$ distinct appears in this sequence
(such a sequence is referred to as an auxiliary sequence in \cite[§2.1.2]{Mochizuki1}).
Fix $k=0,\dots, d$ and put
\[
\mathscr{I}^k \coloneqq  \Image(\mathscr{I} \to \cO_{X}(\ast D)/z^{m(k)}\cO_{X} )  \ .
\] 
Then, $\mathscr{I}^k$ is a constructible sheaf in finite sets on $(X,D)$.
The goal of what follows is to endow $\pi^{\ast}\mathscr{I}^k$ with a canonical structure of sheaves in finite posets.
For a section $a\in \mathscr{I}$ we denote by $[a]_k$ its image under $\mathscr{I} \to \mathscr{I}^k$.

\begin{lem}\label{order_on_quotient}
Let $x\in \widetilde{X}$.
Let  $a,b \in \mathscr{I}_{\pi(x)}$, such that $a <_x b$ and $[a]_k\neq [b]_k$.
Then for every $a',b' \in \mathscr{I}_{\pi(x)}$ with $[a]_k=[a']_k$ and $[b]_k=[b']_k$, we have $a' <_x b'$.
\end{lem}

\begin{proof}
We can suppose that $\pi(x)=0$.
By assumption $a\neq b$.
Write 
\[
a-b \coloneqq  f(y) z^{\order(a-b)}+ \sum_{m>\order(a-b)}   (a-b)_m(y) z^m   \ .
\]
where $f(0)\neq 0$.
Put $x=(\theta_1,\dots, \theta_l) \in \pi^{-1}(0)$ and write $\order(a-b)=(m_1,\dots, m_l)$.
Then, the assumption $a <_x b$ means
\[
\Re (f(0) e^{m_1 \theta_1 + \cdots + m_l \theta_l} ) <0   \ .
\]
Now let $a',b' \in \mathscr{I}_{\pi(x)}$ with $[a]_k=[a']_k$ and $[b]_k=[b']_k$.
In particular $[a-b]_k=[a'-b']_k$, that is 
\begin{align*}
a'-b' & = a-b + z^{m(k)}g    \  , g\in    \cO_{X,0}      \\
           & = f(y) z^{\order(a-b)} + z^{m(k)}g + \sum_{m>\order(a-b)}   (a-b)_m(y) z^m 
\end{align*}
Since  $[a]_k\neq [b]_k$, we have $m(k)> \order(a-b)$.
Hence $\order(a-b)=\order(a'-b')$ and
\[
a'-b'  = f(y) z^{\order(a-b)}  + \sum_{m>\order(a'-b')}   (a'-b')_m(y) z^m  \ .
\]
Thus, we also have $a' <_x b'$.
\end{proof}

\begin{cor}\label{quotient_is_a_level_morphism}
For $x\in \widetilde{X}$, there is a unique order $\leq_x^k$ on $\mathscr{I}^k_{\pi(x)}$ such that 
\[
(\mathscr{I}_{\pi(x)},\leq_x)\to (\mathscr{I}^k_{\pi(x)},\leq_x^k)
\] 
is a level morphism of posets in the sense of \cref{level_intro_def}.
\end{cor}

\begin{proof}
The uniqueness is obvious since $\mathscr{I}_{\pi(x)}\to \mathscr{I}^k_{\pi(x)}$ is surjective.
For $\alpha, \beta \in \mathscr{I}^k_{\pi(x)}$, put $\alpha \leq_x^k \beta$ if $\alpha = \beta$ or if $\alpha \neq  \beta$ and there exists  $a,b \in \mathscr{I}_{\pi(x)}$ with $\alpha=[a]_k$ and  $\beta=[b]_k$ such that $a <_x b$ and $[a]_k\neq [b]_k$.
Then, \cref{quotient_is_a_level_morphism} follows from \cref{order_on_quotient}.
\end{proof}

We stay in the setting of \cref{level_structure_diff}.
For every open subset $U\subset \widetilde{X}$, we define a partial order $\leq_U$ on $(\pi^{\ast}\mathscr{I}^k )(U)$ by
\[
\text{$a\leq_U b$ if and only if $a\leq_x^k b$ in $\mathscr{I}^k_{\pi(x)}$ for every $x\in U$.}
\]
Then, $\pi^{\ast}\mathscr{I}^k \in \Cons_P(\widetilde{X}, \Poset)$ and the canonical morphism
\[
\pi^{\ast}\mathscr{I}  \to \pi^{\ast}\mathscr{I}^k 
\]
is a morphism of $P$-constructible sheaves in finite posets on $\widetilde{X}$.
The chain
\[
\cO_{X}(\ast D)/\cO_{X} \to \cO_{X}(\ast D)/z^{m(d-1)}\cO_{X} \to \dots \to \cO_{X}(\ast D)/z^{m(0)}\cO_{X} 
\]
induces a chain of constructible sheaves on $(X,D)$
\[
\mathscr{I}=\mathscr{I}^d \to \mathscr{I}^{d-1} \to \dots \to \mathscr{I}^0=\ast   
\]
which in turn induces a chain 
\[
\pi^{\ast}\mathscr{I}=\pi^{\ast}\mathscr{I}^d \to \pi^{\ast}\mathscr{I}^{d-1} \to \dots \to \pi^{\ast}\mathscr{I}^0=\ast   
\]
of $P$-constructible sheaves in finite posets over $\widetilde{X}$.
By \cref{quotient_is_a_level_morphism}, the corresponding chain of cocartesian fibrations in finite posets on $\Pi_{\infty}(\widetilde{X},P)$
\begin{equation}\label{level_true_diff}
\cI=\cI^d \to \cI^{d-1} \to \dots \to \cI^0=\ast 
\end{equation}
 is a level structure on $(\widetilde{X},P,\cI)$ relative to $(X,D)$ in the sense of \cref{def:level_structure_length}.
\end{construction}

\begin{rem}
The level structure \eqref{level_true_diff} depends on a choice of auxiliary sequence \eqref{auxiliary_sequence}.
\end{rem}

\subsection{Piecewise elementarity}

\begin{lem}\label{preparation_level_structure_diff}
Fix $X\subset \mathbb{C}^n$ be a polydisc with coordinates $(z,y)=(z_1,\dots, z_l,y_1,\dots, y_{n-l})$.
Let $D$ be the divisor defined by $z_1 \cdots z_l = 0$ and put $I=\{1,\dots, l\}$.
Let $\mathscr{I}\subset \cO_{X}(\ast D)/\cO_{X}$ be a good sheaf of unramified irregular values.
Let $\pi \colon  \widetilde{X}\to X$  be the real-blow up along $D$.
Let $x\in \widetilde{X}$ such that $\pi(0)=0$.
Let $m\in \mathbb{Z}_{\leq 0}$ non zero.
Then, there is a closed subanalytic neighbourhood $S\subset \widetilde{X}_I^{\circ}$ of $x$  mapping to  a closed subanalytic neighbourhood $\overline{B}\subset D_I^{\circ}$ of $0$ such that for every $y\in \overline{B}$, the following holds ;
\begin{enumerate}\itemsep=0.2cm
\item the fibre $S_y= S\cap \pi^{-1}(y)$ is homeomorphic  to a closed cube in $\mathbb{R}^l$,

\item Via the homeomorphism from (1), for every $a,b \in \mathscr{I}$ defined on $\overline{B}$ with $\order(a-b)=m$, the Stokes locus $(S_y)_{a,b}$ is a hyperplane whose complement has exactly two components $C_1$ and $C_2$ such that $a<_z b$ for every $x\in C_1$ and $b<_z a$ for every $z\in C_2$.
\end{enumerate}

\end{lem}

\begin{proof}
We have $\widetilde{X}_I^{\circ} = (S^1)^l\times \Delta$ where $\Delta \subset \mathbb{C}^{n-l}$ is a polydisc and we see $(S^1)^l$ as the quotient of $\mathbb{R}^l$.
Put $m=(m_1,\dots, m_l)$.
Let  $A\subset \mathbb{R}$ be a finite set.
For $\alpha \in A$,  the locus of points $\theta\in (S^1)^l$  satisfying 
\[
 \cos(\alpha +m_1 \theta_1 + \cdots + m_l \theta_l) =0   
\]
is the image under the canonical projection $\mathbb{R}^l\to (S^1)^l$ of the set of affine hyperplanes $H(\alpha,k)\subset \mathbb{R}^l$,$k\in \mathbb{Z}$ defined by 
\[
\alpha +m_1 \theta_1 + \cdots + m_l \theta_l = \pi/2 + k\pi     \ .
\]
Let $\tilde{x}\in \mathbb{R}^l$ mapping to $x$.
Note that for every $\alpha\in A$ and $k\in \mathbb{Z}$, the hyperplanes $H(\alpha,k)$ and $H(\alpha,k+1)$ are parallel and distant by $\pi/\|m\|$.
Hence, for every sufficiently generic choice of point  $z$ close enough to $\tilde{x}$, the closed cube $C(x,A) \subset \mathbb{R}^l$ centred at $z$  with edges of length $ \pi/\|m\|$ and with two faces parallel to the above hyperplanes satisfies
\begin{enumerate}\itemsep=0.2cm
\item[(a)] for every $\alpha \in A$, there is a unique $k_{\alpha}\in \mathbb{Z}$ such that  $C(x,A)$  meets $H(\alpha, k_{\alpha})$.

\item[(b)] $C(x,A) \setminus H(\alpha, k_{\alpha})$ has exactly two connected components.
\end{enumerate}
Since $p \colon \mathbb{R}^l \to (S^1)^l$ is a diffeomorphism in a neighbourhood of $C(x,A)$, its image $p(C(x,A))$ is a closed subanalytic subset of $(S^1)^l$.
For $a,b \in \mathscr{I}$ defined in a neighbourhood of $0$, write
\[
a-b \coloneqq  f_{a,b}(y) z^{\order(a-b)}+ \sum_{m'>\order(a-b)}   (a-b)_{m'}(y) z^{m'}   \ .
\]
Choose some argument $\alpha_{a,b}\in \mathbb{R}$ for $f_{a,b}(0)$ and put 
\[
A\coloneqq \{\alpha_{a,b}, a,b \in \mathscr{I}\text{ defined in a neighbourhood of $0$ with} \order(a-b)=m\}   \ .
\]
Fix $\varepsilon > 0$ small enough and put
\[
S\coloneqq  p(C(x,A)) \times \overline{B(0,\varepsilon)}  \subset \widetilde{X}_I^{\circ} 
\]
where $B(0,\varepsilon) \subset \Delta$ is the polydisc of radius $\varepsilon $ centred at $0$.
Note that (1) is satisfied for every $y\in \overline{B(0,\varepsilon)}$.
Since the conditions $(a)$ and $(b)$ are satisfied for $C(x,A)$, observe that $S$ satisfies (2) for $y=0$.
Since the conditions $(a)$ and $(b)$ are open in the choice of $A$, we deduce the existence of $\varepsilon>0$ such that (2) holds for every $y\in \overline{B(0,\varepsilon)}$.
This concludes the proof of \cref{preparation_level_structure_diff}.
\end{proof}

\begin{prop}\label{piecewise_elementary_diff}
Let $X\subset \mathbb{C}^n$ be a polydisc with coordinates $(z,y)=(z_1,\dots, z_l,y_1,\dots, y_{n-l})$.
Let $D$ be the divisor defined by $z_1 \cdots z_l = 0$.
Let $\mathscr{I}\subset \cO_{X}(\ast D)/\cO_{X}$ be a good sheaf of unramified irregular values.
Let $\pi \colon  \widetilde{X}\to X$  be the real-blow up along $D$ and let $\widetilde{X} \to P$ be a finite subanalytic stratification adapted to $\mathscr{I}$.
Let $m(0)<m(1)<\dots <m(d)=0$ be an auxiliary sequence as in \eqref{auxiliary_sequence}.
Then, the  level structure \eqref{level_true_diff} is strongly piecewise elementary (\cref{level_structure}).
\end{prop}

\begin{proof}
Let $k=1,\dots, d$ and consider $p \colon  \cI^k \to \cI^{k-1}$ and the pullback square
\begin{equation}\label{eq_piecewise_elementary_diff}
\begin{tikzcd}
	\mathscr{I}_p^{k} \arrow{r} \arrow{d}{\pi} & \mathscr{I}^{k} \arrow{d}{p} \\
	\mathscr{I}^{k-1,\ens} \arrow{r} & \mathscr{I}^{k-1} \ .
\end{tikzcd} 
\end{equation}
Denote by $D_1,\dots, D_l$ the components of $D$ and fix $I\subset \{1,\dots, l\}$.
Then, we have to show that $(\widetilde{X}_I^{\circ},P,\cI_p^k|_{\widetilde{X}_I^{\circ}})\to D_I^{\circ}$ is strongly piecewise elementary at every point $x\in  \widetilde{X}_I^{\circ}$ in the sense of \cref{defin_piecewise_elementary}.
Since this is a local question on $D_I^{\circ}$, we can suppose that $\pi(x)=0$ and that $\mathscr{I}^{\ens}$ is constant on $D_I^{\circ}$.
That is, we can suppose that  $I= \{1,\dots, l\}$.
By \cref{preparation_level_structure_diff}, there is a closed subanalytic neighbourhood $S\subset \widetilde{X}_I^{\circ}$ of $x$  mapping to  a closed subanalytic neighbourhood $\overline{B}\subset D_I^{\circ}$ of $0$ such that for every $y\in \overline{B}$, the following holds ;
\begin{enumerate}\itemsep=0.2cm
\item the fibre $S_y= S\cap \pi^{-1}(y)$ is homeomorphic  to a closed cube in $\mathbb{R}^l$,

\item via the homeomorphism from (1), for every  $a,b \in \mathscr{I}^{\ens}(D_I^{\circ})$ with $\order(a-b) =  m(k-1)$, the Stokes locus $(S_y)_{a,b}$ is a hyperplane whose complement has exactly two components $C_1$ and $C_2$ such that $a<_z b$ for every $z\in C_1$ and $b<_z a$ for every $z\in C_2$.
\end{enumerate}
Let $y\in \overline{B}$ and let us show that $(S_y,P,\cI_p^{k}|_{S_y})$ is elementary.
Since $\mathscr{I}^{\ens}$ is constant on $D_I^{\circ}$, so is $\mathscr{I}^{k-1,\ens}$. 
Hence, $\cI^{k-1,\ens}$ is a finite coproduct of trivial cocartesian fibrations.
Thus, there is a finite decomposition of cocartesian fibrations in posets
\begin{equation}\label{coproduct_cocart_poset}
\cI_p^k|_{\widetilde{X}_I^{\circ}} = \bigsqcup_{\alpha \in \mathscr{I}^{k-1,\ens}(D_I^{\circ})} \cI_\alpha 
\end{equation}
where $\cI_\alpha$ is the pullback of $\cI_p^k|_{\widetilde{X}_I^{\circ}}$ along $\alpha$.
Hence, we are left to show that $(S_y,P,\cI_\alpha|_{S_y})$ is elementary.
To do this, it is enough to show that $(S_y,P,\cI_\alpha|_{S_y})$ satisfies the conditions of \cref{cor_induction_for_adm_Stokes_triple}.
Observe that 
$$
\mathscr{I}^{\ens}(D_I^{\circ}) \to \pi^*\mathscr{I}^{\ens}(S_y) 
$$
is bijective.
Let $a,b \in \mathscr{I}^{\ens}(D_I^{\circ})$ such that  $[a]_k,  [b]_k  $ are distinct and $[a]_{k-1} =  [b]_{k-1}=\alpha$.
Then,  $\order(a-b) < m(k)$ and  $\order(a-b) \geq  m(k-1)$, so that $\order(a-b) =m(k-1)$.
Since the Stokes loci of $[a]_k,  [b]_k$  and $a,b$ are the same, the proof is complete.
\end{proof}

\begin{cor}\label{strongly_proper_piecewise_level_structure}
Let $(X,D)$ be a normal crossing pair where $X$ admits a smooth compactification.
Let $\mathscr{I}\subset \cO_{X}(\ast D)/\cO_{X}$ be a good sheaf of unramified irregular values.
Let $\pi \colon  \widetilde{X}\to X$  be the real-blow up along $D$ and let $\widetilde{X} \to P$ be a finite subanalytic stratification such that $\pi^{\ast}\mathscr{I}\in\Cons_P(\widetilde{X}, \Poset)$.
Let $(\widetilde{X},P,\cI)$ be the associated Stokes analytic stratified space.
Then, $\pi \colon  (\widetilde{X},P,\cI)\to (Y,Q)$ is a strongly proper family of Stokes analytic stratified spaces in finite posets locally admitting a strongly piecewise elementary level structure.
\end{cor}

\begin{proof}
Combine \cref{blow_up_strongly_proper} with \cref{piecewise_elementary_diff}.
\end{proof}

\subsection{Sheaf of (ramified) irregular values}
We now enhance \cref{secSheaf_of_unramifie_irregular_values}
 to the ramified setting.
 Since this requires to work directly on  $\widetilde{X}$, we first transport the notion of sheaf of unramified irregular values from $X$ to $\widetilde{X}$.

\begin{lem}\label{transfer_irregular_value_sheaf}
Let $(X,D)$ be a strict normal crossing pair.
Let $\pi \colon \widetilde{X}\to X$ be the real blow-up along $D$.
Let $\mathscr{I} \subset  \pi^{\ast}(\cO_{X}(\ast D)/\cO_{X})$ be a sheaf.
Then, the following are equivalent:
\begin{enumerate}\itemsep=0.2cm
\item There is a sheaf of unramified irregular values $\mathscr{J}\subset \cO_{X}(\ast D)/\cO_{X}$ such that $\mathscr{I}\simeq \pi^{\ast}\mathscr{J}$.
\item The direct image $\pi_*\mathscr{I}\subset \cO_{X}(\ast D)/\cO_{X}$ is a sheaf of unramified irregular values and the counit transformation $\pi^{\ast} \pi_* \mathscr{I}\to \mathscr{I}$ is an equivalence.
\end{enumerate}
\end{lem}

\begin{proof}
Immediate from  \cref{unit_irr_value_equivalence}.
\end{proof}

\begin{defin}\label{goodness_upper_unramified_case}
If the equivalent conditions of \cref{goodness_ramified_case_lem} are satisfied, we say that $\mathscr{I} \subset  \pi^{\ast}(\cO_{X}(\ast D)/\cO_{X})$ is a sheaf of unramified irregular values.
If furthermore $\pi_* \mathscr{I}$ is a good sheaf of unramified irregular values, we say that $\mathscr{I}$  is a good sheaf of unramified irregular values.
\end{defin} 
 
 \begin{rem}
By design, \cref{transfer_irregular_value_sheaf} and  \cref{unit_irr_value_equivalence} imply that $(\pi^{\ast},\pi_*)$ induce  a bijection between (good) sheaves of irregular values on $\widetilde{X}$ and (good) sheaves of irregular values on $X$. 
 \end{rem}

\begin{construction}[{\cite[9.c]{Stokes_Lisbon}}]\label{remified_irregular_values}
Let $X\subset \mathbb{C}^n$ be a polydisc with coordinates 
$(z_1,\dots, z_n)$.
Let $D$ be the divisor defined by $z_1 \cdots z_l = 0$ and put $U\coloneqq X\setminus D$.
Let $\pi \colon  \widetilde{X}\to X$ be the real blow-up along $D$.
Let $j \colon U \hookrightarrow \widetilde{X}$ be the canonical inclusion.
Define $\rho   \colon X_d\to X$ by 
$(z_1,\dots, z_n)\to (z_1^d,\dots, z_l^d, z_{l+1}  , \dots,  z_n)$ for $d\geq 1$ and consider the (not cartesian for $d>1$)  commutative square 
\[
\begin{tikzcd}
	\widetilde{X}_d\arrow{r}{\widetilde{\rho}} \arrow{d}{\pi_d} & \widetilde{X} \arrow{d}{\pi} \\
	X_d    \arrow{r}{\rho} & X   
\end{tikzcd} 
\]
of real blow-up along $D$. 
Observe in particular that the above square satisfies the conditions from \cref{stratified_Galois_cover} making it eligible to underlie a vertically finite Galois cover.
The unit transformation $\cO_{U} \hookrightarrow \rho_{\ast}   \cO_{U_d}$ yields an inclusion 
\[
j_{\ast}\cO_{U} \hookrightarrow j_{\ast}\rho_{\ast}   \cO_{U_d} \ .
\]
On the other hand, the unit transformation  $\pi^{\ast}_d\cO_{X_d}(\ast D) \hookrightarrow  j_{d,\ast}\cO_{U_d}$ yields 
\[
\widetilde{\rho}_{\ast}\pi^{\ast}_d\cO_{X_d}(\ast D) \hookrightarrow  \widetilde{\rho}_{\ast} j_{d,\ast}\cO_{U_d} = j_* \rho_{\ast}\cO_{U_d} \ .
\]
Put
\[
IV_d \coloneqq j_{\ast}\cO_{U} \cap\widetilde{\rho}_{\ast}\pi^{\ast}_d\cO_{X_d}(\ast D)  \subset j_{\ast}\cO_{U}  \ .
\]
As in \cref{meromorphic_order}, we have 
\[
IV_d \cap (j_{\ast}\cO_{U})^{\lb}  = j_{\ast}\cO_{U} \cap \widetilde{\rho}_{\ast}\pi^{\ast}_d\cO_{X_d} \ .
\]
We put
\[
\IV_d \coloneqq IV_d/(IV_d\cap  (j_{\ast}\cO_{U})^{\lb}) \subset     (j_{\ast}\cO_{U})/ (j_{\ast}\cO_{U})^{\lb}  \ .
\]
For an arbitrary strict normal crossing pair $(X,D)$, the  $\IV_d$, $d\geq 1$ are defined locally and glue into subshseaves 
\[
\IV_d(X,D) \subset    (j_{\ast}\cO_{U})/ (j_{\ast}\cO_{U})^{\lb}  
\] 
for $d\geq 1$.
By \cref{order_general}, we view  $\IV_d(X,D)$ as an object of  $\Shhyp(\widetilde{X},\Poset)$.
\end{construction}

\begin{eg}
In the setting of \cref{remified_irregular_values}, we have 
\[
\IV_1(X,D) = \pi^{\ast}(\cO_{X}(\ast D)/\cO_{X})
\]
in virtue of \cref{meromorphic_order}.
\end{eg}

\cref{remified_irregular_values} suggests to introduce the following

 \begin{defin}\label{Kummer_cover}
Let $(X,D)$ be a strict normal crossing pair.
Let $d\geq 1$ be an integer.
A \textit{$d$-Kummer cover of $(X,D)$} is an holomorphic map $\rho \colon X\to X $ such that there is a cover  by open subsets $U\subset X$ with $\rho(U)\subset U$ where $\rho|_U$ reads as 
\begin{equation}\label{standard_Kummer}
(z_1,\dots, z_n)\to (z_1^d,\dots, z_l^d, z_{l+1}  , \dots,  z_n)
\end{equation}
for some choice of local coordinates $(z_1,\dots, z_n)$ with $D$ defined by $z_1\cdots z_l = 0$.
 \end{defin}
 
 \begin{rem}
 Following \cite{{Stokes_Lisbon}},  in the setting of \cref{Kummer_cover}, we will denote the source of $\rho$ by $X_d$ instead of $X$.
 \end{rem}

\begin{lem}[{\cite[Lemma 9.6]{Stokes_Lisbon}}]
Let $(X,D)$ be a strict normal crossing pair.
Let $\pi \colon  \widetilde{X}\to X$ be the real blow-up along $D$.
Let $j \colon U \hookrightarrow \widetilde{X}$ be the canonical inclusion.
Let $d\geq 1$ be an integer and let $\rho \colon X_d\to X $ be a $d$-Kummer cover of $(X,D)$.
Then, via the inclusion 
\[
\widetilde{\rho}^{\ast} j_{\ast}\rho_{\ast}   
\cO_{U_d}=\widetilde{\rho}^{\ast}
\widetilde{\rho}_{*} j_{\ast} \cO_{U_d}
\hookrightarrow j_{\ast} \cO_{U_d} \ ,
\]
we have 
\[
\widetilde{\rho}^{\ast} (\IV_d(X,D)) =\pi_d^{\ast}(\cO_{X_d}(\ast D)/\cO_{X_d})   \
 \]
 in $\Shhyp(\widetilde{X}_d,\Poset)$.
\end{lem}

\begin{lem}\label{goodness_ramified_case_lem}
Let $(X,D)$ be a strict normal crossing pair.
Let $d\geq 1$ be an integer and let $\mathscr{I} \subset \IV_d(X,D)$ be a sheaf.
Then, the following are equivalent:
\begin{enumerate}\itemsep=0.2cm
\item For every $x\in X$, there exist local coordinates $(z_1,\dots, z_n)$ centred at $x$ with $D$ defined by $z_1\dots z_l = 0$ such that for the map $\rho$ given by \eqref{standard_Kummer}, the pullback $\widetilde{\rho}^{\ast}\mathscr{I}$ is a sheaf of unramified irregular values (\cref{goodness_upper_unramified_case}).
\item For every open subset $U\subset X$ and every  $d$-Kummer cover $\rho \colon U_d \to U$, the pullback   $\widetilde{\rho}^{\ast}\mathscr{I}$ is a sheaf of unramified irregular values (\cref{goodness_upper_unramified_case}).
\end{enumerate}
\end{lem}
\begin{proof}
Left to the reader.
\end{proof}

\begin{defin}\label{goodness_ramified_case}
If the equivalent conditions of \cref{goodness_ramified_case_lem} are satisfied, we say that $\mathscr{I} \subset \IV_d(X,D)$ is a \textit{sheaf of irregular values}.
If furthermore the $\widetilde{\rho}^{\ast}\mathscr{I}$ are good sheaves of unramified irregular values, we say that $\mathscr{I}$  is a good sheaf of  irregular values.
\end{defin}

\begin{lem}\label{finitely_generated_criterion}
Let $f \colon (N,Y,Q)\to (M,X,P)$ be a morphism of analytic stratified spaces such that the induced morphism $f\colon Y\to X$ is open surjective.
Let  $\cF\in \ConsPhyp(X,\Cat_{\infty})$.
Then,  $\cF$ is locally generated if and only so is $f^{*}(\cF)$.
\end{lem}

\begin{proof}
The direct implication follows from \cref{pullback_locally_generated_lem}.
Assume that $f^{*}(\cF)$  is locally generated.
To show that $\cF$ is locally generated, it is enough to show in virtue of \cref{prop:locally_contractible_strata} that every open subset $U\subset X$ such that $\Pi_{\infty}(U,P)$ admits an initial object $x$ contains an open neighbourhood of $x$ on which $\cF$ is globally generated.
By surjectivity,  choose $x'\in Y$ above $x$.
Since  $f^{*}(\cF)$ is locally generated, we can choose an open subset $V'\subset Y$ containing $x'$ on which $f^{*}(\cF)$ is globally generated.
By \cref{pullback_locally_generated_lem},  we can suppose that $V'\subset  f^{-1}(U)$.
At the cost of shrinking $V'$ further, we can suppose by \cref{prop:locally_contractible_strata} that $x'$ is initial in $\Pi_{\infty}(V',Q)$.
Put $V\coloneqq f(V')\subset U$.
Note that $V$ is an open neighbourhood of $x$ by openness of $f \colon Y\to X$.
To conclude, let us show that $\cF|_V$ is globally generated.
For $y\in V$, let us show that $\cF(V)\to \cF_y$ is essentially surjective.
Choose $y'\in V'$ above $y$.
Then, by design of $U$ and $V'$ there is a commutative diagram
\[ 
\begin{tikzcd}
& \cF(U) \arrow{r}{\sim} \arrow{d}  \arrow[ldd, bend right = 30]  &   \cF_x \simeq   (f^{*}(\cF))_{x'}  \\
	&\cF(V) 	\arrow{r} \arrow{ld}  & (f^{*}(\cF))(V')  \arrow[lld, twoheadrightarrow, bend left = 30]  \arrow{u}{\wr}  \\ 
	 \cF_y \simeq   (f^{*}(\cF))_{y'}  &                 & 
\end{tikzcd} 
\]
The conclusion thus follows.
\end{proof}

\begin{prop}\label{ramified_irregular_values_property}
Let $(X,D)$ be a strict normal crossing pair.
Let $\mathscr{I} \subset \IV_d(X,D)$ be a sheaf of irregular values for some $d\geq 1$.
Then:
\begin{enumerate}\itemsep=0.2cm
\item $\mathscr{I}^{\ens}$ is hyperconstructible on $(\widetilde{X},\widetilde{D})$;

\item $\mathscr{I}$ is locally generated;
\end{enumerate}

 If furthermore $\mathscr{I}$ is good, then
\begin{enumerate}
\item[(3)] for every open subset $U \subset \widetilde{X}$ subanalytic in $S^1 L(D)$, 
for every $a,b\in \mathscr{I}(U)$, the sets $U_{a< b}$,$U_{a= b}$,$U_{a\ast b}$ are locally closed subanalytic in $S^1 L(D)$;
\end{enumerate}

If furthermore  $X$ admits a smooth compactification, then 

\begin{enumerate}
\item[(4)] there exists a finite subanalytic stratification 
$\widetilde{X}\to P$ refining $(\widetilde{X},\widetilde{D})$ such that $\mathscr{I}\in \ConsPhyp(\widetilde{X},\Poset)$.
\end{enumerate}

\end{prop}

\begin{proof}
Item (1) follows from the fact that local hyperconstancy can be check locally for the étale topology.
Item (2) is a local question.
Hence, we can assume the existence of a surjective $d$-Kummer  cover $\rho \colon X_d \to X$ of $(X,D)$ of the form \eqref{standard_Kummer} such that $\widetilde{\rho}^{\ast}\mathscr{I}$ is a sheaf of unramified irregular values.
In particular $\widetilde{\rho}^{-1}\mathscr{I}$ is locally generated.
Observe that $\widetilde{\rho}$ is open and surjective.
Then, (2) follows from \cref{finitely_generated_criterion}.
Let us prove (3).
We are going to apply \cref{from_local_to_global_subanalyticity}.
Conditions (1) and (2) from \cref{from_local_to_global_subanalyticity} are satisfied.
To show that \cref{from_local_to_global_subanalyticity}-(3) is satisfied, we can suppose the existence of a surjective Kummer cover $\rho \colon X_d \to X$ such that $\widetilde{\rho}^{\ast}\mathscr{I}$ is a sheaf of unramified irregular values.
Let $W\subset S^1L(D)$ be an open \textit{subanalytic} subset.
Let $? \in \{<,=,\ast\}$ and let $a,b\in \mathscr{I}(W\cap X)$.
We want to show that $(W\cap \widetilde{X})_{a ? b}$ is a subanalytic subset of $W$.
Since $W$ and $\widetilde{X}$ are subanalytic in $S^1L(D)$, so is $W\cap \widetilde{X}$.
Hence $\widetilde{\rho}^{\ast}(W\cap \widetilde{X}) \subset \widetilde{X}_d$ is subanalytic as well.
By \cref{good_sheaf_irregular_values_is_subanalytic} applied to $\widetilde{\rho}^{-1}\mathscr{I}$, 
we know that $(\widetilde{\rho}^{\ast}(W\cap \widetilde{X}))_{\pi^{\ast}a ? \pi^{\ast}b}$ is subanalytic.
On the other hand, we have 
\[
(W\cap \widetilde{X})_{a ? b} = \widetilde{\rho}((\widetilde{\rho}^{\ast}(W\cap \widetilde{X}))_{\pi^{\ast}a ? \pi^{\ast}b}) 
\]
Since the image of a subanalytic subset by a proper map is again subanalytic, we conclude that $(W\cap \widetilde{X})_{a ? b}$  is subanalytic and (3) is proved.
We know prove (4).
Let $ X\hookrightarrow Y$ be a smooth compactification of $X$.
At the cost of applying resolution of singularities,  we can suppose that $Z\coloneqq Y\setminus X$ is a divisor such that $E\coloneqq Z+D$ has strict normal crossings.
Hence, $X$ admits a finite cover by open subanalytic subsets $U\simeq  \Delta^{n-k} \times (\Delta^*)^k$ with coordinates $(z,y)$ such that $\Delta\cap U$ is defined by $z_1\cdots z_l = 0$, where $\Delta\subset \mathbb{C}$ is the unit disc.
Let $S_+,S_-\subset \Delta^*$ be a cover  by open sectors.
For $\varepsilon \colon  \{1,\dots, k\} \to \{-,+\}$, put 
\[
U_\varepsilon \coloneqq   \Delta^{n-k} \times S_{\varepsilon(1)} \times \cdots   \times S_{\varepsilon(k)}   \subset U
\]
Let $(I_+,I_{-}  )\subset S^1$ be a cover by strict open intervals.
For $\varepsilon \colon  \{1,\dots, k\} \to \{-,+\}$ and $\eta \colon  \{1,\dots, l\}  \to \{-,+\}$, put 
\[
V_{\varepsilon,\eta} \coloneqq [0,1)^l\times I_{\eta_1}\times \cdots \times I_{\eta_l} \times \Delta^{n-l-k}\times  S_{\varepsilon(1)} \times \cdots   \times S_{\varepsilon(k)}   \subset \pi^{-1}(U_\varepsilon)
\]
Note that $V_{\varepsilon,\eta}$ is a subanalytic subset of $S^1L(D)$.
To prove (4), it is enough to show that the $V_{\varepsilon,\eta} $ satisfy the conditions of \cref{finite_subanalytic_stratification_criterion}.
This follows from the above points (1) (2) (3) and \cref{semi_globally_generated} applied to \cref{global_generation_Xtilde}.
\end{proof}

\begin{prop}\label{existence_ramified_piec_el_level_structure}
Let $(X,D)$ be a normal crossing pair where $X$ admits a smooth compactification.
Let $\pi \colon  \widetilde{X}\to X$  be the real-blow up along $D$.
Let $\mathscr{I} \subset \IV_d(X,D)$ be a good sheaf of irregular values for some $d\geq 1$.
Let $\widetilde{X} \to P$ be a finite subanalytic stratification such that $\mathscr{I}\in\Cons_P(\widetilde{X}, \Poset)$.
Let $(\widetilde{X},P,\cI)$ be the associated Stokes analytic stratified space.
Then, $\pi \colon  (\widetilde{X},P,\cI)\to (Y,Q)$ is a strongly proper family of Stokes analytic stratified spaces in finite posets locally admitting a ramified strongly piecewise elementary level structure  (\cref{ramified_level_structure}).
\end{prop}
\begin{proof}
Immediate from \cref{strongly_proper_piecewise_level_structure}.
\end{proof}

\cref{existence_ramified_piec_el_level_structure} unlocks all the results proved in \cref{nc_space} and \cref{geometricity_section}.
In particular, we have the following

\begin{thm}\label{geometricit_Stokes_classical_case}
In the setting of \cref{existence_ramified_piec_el_level_structure}, let $k$ be an  animated commutative ring.
Then, $\bfSt_{\cI}$ is locally geometric locally of finite presentation. 
Moreover, for every animated commutative $k$-algebra $A$ and every morphism
	\[ x \colon \Spec(A) \to \bfSt_{\cI} \]
	classifying a Stokes functor $F \colon \cI \to \Perf_A$, there is a canonical equivalence
	\[ x^\ast \mathbb T_{\bfSt_{\cI}} \simeq \Hom_{\Fun(\cI,\Mod_A)}( F, F )[1] \ , \]
	where $\mathbb T_{\bfSt_\cI}$ denotes the tangent complex of $\bfSt_\cI$ and the right hand side denotes the $\Mod_A$-enriched $\Hom$ of $\Fun(\cI,\Mod_A)$.
\end{thm}

\begin{proof}
Combine \cref{strongly_proper_piecewise_level_structure} with \cref{Representability_via_toen_vaquie}.
\end{proof}

\subsection{Comparison with wild character varieties in dimension $1$}\label{subsec:comparison}

In this section we take $k = \C$, we specialize our construction in dimension $1$ and we compare it with the classical construction of wild character varieties, as outlined in \cite[\S13]{Boalch_Topology_Stokes}.

\medskip

Let $X$ be a smooth compact complex curve.
In this case a normal crossing divisor $D$ consists of a finite number of points, and the real blow-up $\pi \colon \widetilde{X} \to X$ is a $\R$-analytic surface and its boundary
\[ \partial \widetilde{X} \coloneqq \pi\inv( D ) \simeq \coprod_{ a \in D } S^1_a \]
is a disjoint union of circles, one per each point of $D$.
As in \cite[\S5]{Boalch_Topology_Stokes}, we restrict the discussion to the case where $D = \{a\}$ consists of a single point, as it is straightforward to extend the comparison to the general case.

\medskip

To keep notational clash to a minimum, we denote by $\mathbb E \to \partial \widetilde{X}$ the exponential local system (see \cite[\S5.1]{Boalch_Topology_Stokes}), whose local sections on $\partial \widetilde{X}$ are given by Puiseaux series around $a$.
Fix an irregular class $\Theta \colon \mathbb E \to \mathbb N$ in the sense of \cite[\S5.2]{Boalch_Topology_Stokes}, and we set
\[ I \coloneqq \Theta\inv(\mathbb N_{>0}) \]
to be the set of \emph{active exponentials}.

\medskip

Given $q \in I$, we can find an expansion locally around $a \in D$ as a Puiseux series
\[ q = \sum \lambda_i z^{-k_i} \ , \]
where $\lambda_i \in \C$ and $k_i \in \Q_{>0}$.
Let $d_q$ be the lowest common multiple of the denominators of the $k_i$.
Then $q$ can be interpreted as a function on a $d_q$-Kummer cover $X_{d_q}$ of $X$.
Letting $d$ to be the lowest common multiple of the set $\{d_q\}_{q \in \cI}$, we can interpret all active exponentials as functions on $X_d$.
In particular, \cref{goodness_ramified_case_lem} allows to interpret the set $I$ as a sheaf of irregular values $\mathscr I$ in the sense of \cref{goodness_ramified_case}.
Since $X$ is a complex curve, $\mathscr I$ is automatically a good sheaf of irregular values.
The Stokes directions defined in \cite[\S5.4]{Boalch_Topology_Stokes} correspond exactly to the stratification $\mathbb S$ of $\widetilde{X}$ by Stokes loci in the sense of \cref{Stokes_locus}.
The Stokes arrows defined in \cite[5.5]{Boalch_Topology_Stokes} coincide with the order on $\cI$ of \cref{order_general} (see also \cref{meromorphic_order}).

\medskip

In this setting, both Stokes filtered local systems in the sense of Boalch \cite[\S6]{Boalch_Topology_Stokes} (associated to the irregular class $\Theta$) and Stokes functors in the sense of \cref{def:Stokes_sheaf} (associated to $\cI$) are defined.
The comparison between the two notions only makes sense when we take as category of coefficients the abelian category $\cE \coloneqq \Mod_\C^\heartsuit$ of $\C$-vector spaces.

\medskip

To begin with we show how to produce a Stokes functor out of a Stokes filtered local system.
The key ingredient is the following observation, which allows to recast the Stokes condition for two filtrations of \cite[Definition 3.9]{Boalch_Topology_Stokes} in our language.

\begin{notation}\label{notation:underlying_filtered_object}
	Let $J$ be a poset and let $F \colon J \to \cE$ be a functor.
	We write $|F|$ for the colimit of $F$.
	In what follows, when $\cE = \Mod_\C^\heartsuit$, we think of $|F|$ as a vector space filtered by $F$.
	Notice that a priori this filtration is not by subspaces, so it is not a filtration in the stricter sense of \cite[\S3.2]{Boalch_Topology_Stokes} (however, when the pointwise split condition is imposed, this filtration will automatically be by subspaces, see \cref{observation:filtration_by_monomorphisms} below).
	When the order on $J$ is trivial, we rather say that $|F|$ is graded by $F$.
\end{notation}

\begin{observation}[The Stokes condition for two filtrations]\label{observation:Stokes_condition}
	Consider three posets $J$, $J^+$ and $J^-$ together with morphisms of posets
	\[ \begin{tikzcd}
		J^- & J \arrow{l}[swap]{g} \arrow{r}{f} & J^+ \ .
	\end{tikzcd} \]
	Equivalently, exodromy allows to interpret this as a constructible sheaf of posets $\cJ$ on the open interval $(0,1)$ stratified in a single point.
	Further assume that $f$ and $g$ induce the identity on the underlying sets.
	Consider a vector space $V$ equipped with two filtrations $F^-$ and $F^+$, indexed respectively by $J^-$ and $J^+$.
	These two filtrations satisfy the Stokes condition in the sense of \cite[Definition 3.9]{Boalch_Topology_Stokes} if and only if there exists a grading $G$ of $V$ indexed by $J^{\ens}$ together with identifications
	\[ g_!( i^{\ens}_{J,!}(G) ) \simeq F^- \ , \qquad f_!( i^{\ens}_{J,!}(G) ) \simeq F^+ \ . \]
	Indeed, unraveling the formulae for the left Kan extensions, we see that the left hand side coincide with the filtrations induced from the grading in the sense of Boalch.
	Notice that in this case \cref{eg:Stokes_structures_at_a_point} allows to identify $i^{\ens}_{J,!}(G)$ with a Stokes functor over $((0,1),P,\cJ)$.
\end{observation}

\begin{construction}[From Stokes filtered local systems to Stokes functors]\label{construction:from_Stokes_filtered_to_Stokes_functors}
	Let $(\cV,\cF)$ be a filtered local system in the sense of Boalch, associated with the irregular class $\Theta$.
	Let $\cI \to \Pi_\infty(\widetilde{X}, \mathbb S)$ be the cocartesian fibration in posets associated to $\mathscr I$.
	Notice that $\Pi_\infty( \partial\widetilde{X} \smallsetminus \mathbb S )$ is just a set, and by definition of Stokes directions the restriction $\cI |_{\partial \widetilde{X} \smallsetminus \mathbb S}$ is locally constant.
	The filtration $\cF$ considered in \cite[\S6]{Boalch_Topology_Stokes}, thanks to condition \textbf{SF1}) in \emph{loc.\ cit.}, consists exactly of:
	\begin{enumerate}\itemsep=0.2cm
		\item a functor
		\[ F \colon \cI_{\partial \widetilde{X} \smallsetminus \mathbb S} \longrightarrow \Mod_\C^\heartsuit \ , \]
		which furthermore factors through the full subcategory of finite-dimensional $\C$-vector spaces.
		
		\item an isomorphism
		\[ |F| \simeq \cV |_{\cI_{\partial \widetilde{X} \smallsetminus \mathbb S}} \ , \]
		where $|-|$ denotes the induction along the structural morphism $\cI_{\partial \widetilde{X} \smallsetminus \mathbb S} \to \Pi_\infty(\partial \widetilde{X} \smallsetminus \mathbb S)$.
		Concretely, this consists in providing an isomorphism of $\cV$ with the highest piece of the filtration $F_\theta$ of $V_\theta$ for all $\theta \in \partial \widetilde{X} \smallsetminus \mathbb S$.
	\end{enumerate}
	Notice that since the orders on $\partial \widetilde{X} \smallsetminus \mathbb S$ are total, and since we work over a field, the above data automatically defines a Stokes functor over $(\partial \widetilde{X} \smallsetminus \mathbb S, \ast, \cI |_{\partial \widetilde{X} \smallsetminus \mathbb S})$.
	Since Stokes functors form a sheaf on $\partial\widetilde{X}$ (a property that holds in virtue of the very \cref{def:Stokes_sheaf}), in order to extend these filtrations to a Stokes functor 
	In order to extend it to a Stokes functor
	\[ \cI |_{\partial \widetilde{X}} \longrightarrow \Mod_\C^\heartsuit \ , \]
	it is enough to work locally around a Stokes direction $\theta \in \mathbb S$.
	In particular, we can work in a small sector around $\theta$ that contains no other Stokes directions.
	In this case, we are in the setting of \cref{observation:Stokes_condition}, and therefore the grading $G$ whose existence is guaranteed by \textbf{SF2}) allows to extend the two nearby filtrations into a Stokes functor (concretely given by $i_{\cI_{\theta},!}^{\ens}(G)$ and then extended to the sector via the equivalence supplied by \cref{eg:Stokes_structures_at_a_point}).
\end{construction}

Let us now explain how to produce a Stokes filtered local system starting with a Stokes functor.
Before giving the construction, we need a couple of preliminary observations.

\begin{observation}\label{observation:underlying_local_system}
	Let
	\[ p \colon \cI \longrightarrow \Pi_\infty(\partial \widetilde{X}, \mathbb S) \]
	be the structural morphism of the cocartesian fibration associated to the $\mathbb S$-constructible sheaf of posets $\mathscr I$.
	Let $\cE$ be a presentable $\infty$-category and let
	\[ F \colon \cI \longrightarrow \cE \]
	be a Stokes functor.
	Seeing $\Pi_\infty(\partial\widetilde{X}, \mathbb S)$ as a trivial fibration over itself and applying \cref{cor:stokes_functoriality_IHES}, we see that
	\[ |F| \coloneqq p_!(F) \colon \Pi_\infty(\partial\widetilde{X}, \mathbb S) \longrightarrow \cE \]
	is again a Stokes functor.
	In particular, \cref{Stokes_sheaf_trivial_fibration} implies that $|F|$ is a local system with coefficients in $\cE$.
	In line with \cref{notation:underlying_filtered_object}, we think of $|F|$ as a local system equipped with the extra structure of a Stokes functor $F$.
\end{observation}

\begin{observation}\label{observation:filtration_by_monomorphisms}
	Let
	\[ F \colon \cI \longrightarrow \Mod_\C^\heartsuit \]
	be a Stokes functor.
	Let $\theta \in \partial\widetilde{X}$.
	For $q <_\theta q' \in \cI_\theta$, the morphism $F_\theta(q) \to F_\theta(q')$ is a monomorphism.
	This is not imposed directly as part of our definition, but since $F$ is split at $\theta$, this condition follows automatically.
\end{observation}

\begin{construction}[From Stokes functors to Stokes filtered local systems]\label{construction:from_Stokes_functors_to_Stokes_filtered}
	Let $F \colon \cI \to \Mod_\C^\heartsuit$ be a Stokes functor.
	Assume that $F$ takes values in finite dimensional $\C$-vector spaces.
	We define a Stokes filtered local system $(\cV,\cF)$ as follows.
	We take $\cV \coloneqq |F| \coloneqq p_!(F)$, which is a local system in virtue of \cref{observation:underlying_local_system}.
	Whenever $\theta \in \partial\widetilde{X} \smallsetminus \mathbb S$, the restriction $F |_{\cI_\theta}$ gives a filtration of $\cV_\theta$ by subspaces, as remarked in \cref{observation:filtration_by_monomorphisms}.
	Therefore, these are filtrations in the more restrictive sense of \cite[\S3.2]{Boalch_Topology_Stokes}.
	Besides, the exodromy equivalence guarantees that condition \textbf{SF1}) of \cite[\S6]{Boalch_Topology_Stokes} is satisfied.
	On the other hand, condition \textbf{SF2}) is also automatically satisfied thanks to \cref{observation:Stokes_condition}.
\end{construction}

\begin{thm}\label{thm:comparison}
	Constructions~\ref{construction:from_Stokes_filtered_to_Stokes_functors} and \ref{construction:from_Stokes_functors_to_Stokes_filtered} induce an equivalence between the category of Stokes filtered local systems in the sense of \cite[\S6]{Boalch_Topology_Stokes} and the category of Stokes functors with values in $\Mod_\C$ that are $\C$-flat and have finite dimensional stalks.
\end{thm}

\begin{proof}
	Notice that both categories are abelian (in particular, $1$-categories).
	It is then straightforward to verify that the two constructions are functorial and that they are inverse to each other.
\end{proof}

\begin{rem}\label{rem:comparison}
	It is immediate from \cref{thm:comparison} to deduce that the wild character \emph{stack} constructed in \cite[\S13]{Boalch_Topology_Stokes} coincides with the (classical truncation of) $\bfStflat_{\cI}$.
	In particular, the work of Boalch \cite{Boalch_curve} proves the existence of a good moduli space for the open and closed substack $\bfStflat_\cI$ that corresponds to fixing the rank of the underlying local system.
	From the point of view of the present paper, the good moduli space of $\bfStflat_\cI$ can be constructed intrinsically via the results of \cite{AHLH}, and we expect that reasoning along these lines will allow to construct good moduli spaces for $\bfStflat_\cI$ in arbitrary dimension.
\end{rem}

\section{Appendix: Stability properties for smooth and proper stable $\infty$-categories}\label{stability_properties}

Fix an animated ring $k$.
Recall that $\Mod_k \in \CAlg(\PrLomega)$ (see e.g. \cite[Proposition 2.4]{Binda_Porta_Azumaya}).
We set
\[ \PrLomega_k \coloneqq \Mod_{\Mod_k}( \PrLomega ) \qquad \text{and} \qquad \PrL_k \coloneqq \Mod_{\Mod_k}(\PrL) \ . \]
Given $\cC \in \PrLomega_k$, we write
\[ \Hom_\cC \colon \cC\op \times \cC \to  \Mod_k  \]
for the canonical enrichment over $\Mod_k$.
Recall also that $\cC$ is dualizable in $\PrL_k$, with dual $\cC^{\vee}$ given by $\Ind((\cC^\omega)\op)$ and write
\[ \mathrm{coev}_\cC \colon \Mod_k \to  \cC^\vee \otimes_k \cC \]
for the coevaluation map in $\PrL_k$.
Recall the following definitions:

\begin{defin}\label{finite_type_defin}
	A compactly generated $k$-linear stable $\infty$-category $\cC\in \PrLomega_k$ is said to be:
	\begin{enumerate}\itemsep=0.2cm
		\item \emph{of finite type} if it is a compact object in $\PrLomega_k$;
		
		\item \emph{proper} if for every compact objects $x, y \in \cC^\omega$, $\Hom_\cC(x,y)$ belongs to $\Perf(k)$;
		
		\item \emph{smooth} if $\mathrm{coev}_\cC$ preserves compact objects.
	\end{enumerate}
\end{defin}

\begin{rem}
	Let $\cC \in \PrLomega_k$.
	If $\cC$ is of finite type, then it is smooth.
	On the other hand, if $\cC$ is smooth and proper, then it is of finite type.
\end{rem}

\begin{lem}\label{finite_limit_prLRomega}
	Let $\cC_\bullet \colon A \to \PrLR_k$ be a diagram such that $\cC_a$ is compactly generated for every $a\in A$.
	Set
	\[ \cC \coloneqq \lim_{a \in A} \cC_a \ , \]
	the limit being computed in $\PrL$.
	Then $\cC$ is compactly generated.
	Furthermore, if $\cC_a$ is of finite type for every $a \in \cC$ and $A$ is a compact $\infty$-category, then $\cC$ is of finite type as well.
\end{lem}

\begin{proof}
	Since the limit is computed in $\PrL_R$, \cite[Corollary 3.4.3.6]{Lurie_Higher_algebra} and \cite[Proposition 5.5.3.13]{HTT} show that it can alternatively be computed in $\CAT_\infty$.
	Since all the transition morphisms are in $\PrR$ as well, \cite[Theorem 5.5.3.18]{HTT} guarantees that the limit can be also computed in $\PrR$.
	Using the equivalence $\PrR \simeq (\PrL)\op$, we conclude that passing to left adjoints we can write
	\[ \cC \simeq \colim_{a \in A\op} \cC_a \ , \]
	the colimit being computed in $\PrL$.
	Notice that the transition maps in this colimit diagram, being left adjoints to colimit-preserving functors, automatically preserve compact objects.
	Thus, \cite[Lemma 5.3.2.9]{Lurie_Higher_algebra} shows that this colimit can be computed in $\PrLomega$.
	It follows that $\cC$ is compactly generated.
	Besides, \cite[Corollary 3.4.4.6]{Lurie_Higher_algebra} implies that this colimit can also be computed in $\PrLomega_k$, so the second half of the statement follows from the fact that compact objects are closed under finite colimits and retracts.
\end{proof}

\begin{cor}\label{cocart_finite_type}
Let $(X,P,\cI)$ be a Stokes stratified spaces in finite posets such that $(X,P)$ is categorically compact.
	Let $\cE$ be a compactly generated $k$-linear stable $\infty$-category  of finite type.
	Then,  so is $\Funcocart(\cI,\cE)$.
\end{cor}
\begin{proof}
	Let $\Upsilon_\cI \colon  \PiinftySigma(X,P) \to \PrL$ be the straightening of $p \colon  \cI\to  \PiinftySigma(X,P)$ and consider the diagram
	\[
	\Fun_!(\Upsilon_\cI(-), \cE) \colon  \PiinftySigma(X,P) \to  \PrL 
	\]
	where $\Fun_!$ denotes the functoriality given by left Kan extensions.
	From \cite[3.3.3.2]{HTT}, there is a canonical equivalence 
	\[
	\Funcocart(\cI,\cE) \simeq \lim_{\cX}  \Fun_!(\Upsilon_\cI(-), \cE) 
	\]
	By \cref{induction_limits},  the transition functors of the above diagram are left and right adjoints.
	Furthermore, $\Fun_!(\cI_x, \cE)$ is of finite type for every $x\in \cX$.
	Then, \cref{cocart_finite_type} follows from \cref{finite_limit_prLRomega}.
\end{proof}

\bibliographystyle{plain}
\bibliography{dahema}

\end{document}